\documentclass[11 points,reqno]{amsart}

\usepackage{amssymb}
\usepackage{float}
\usepackage[nobysame,alphabetic]{amsrefs}
\usepackage{tikz}
\usetikzlibrary{arrows}
\usepackage{enumerate}

\def\a{\alpha}

\def\A{\mathcal{A}}
\def\b{\beta}

\def\g{\gamma}

\def\e{\epsilon}

\def\cE{\mathcal{E}}

\def\N{\mathbb{N}}

\def\R{\mathbb{R}}

\def\Z{\mathbb{Z}}

\def\cross{\times}

\def\P{\mathbb{P}}
\def\le{\leqslant}
\def\ge{\geqslant}

\def\stab{\text{Stab}}

\newcommand{\norm}[1]{|#1|}
\newcommand{\setnorm}[1]{\#|#1|}

\newcommand{\gp}[3]{(#2\cdot#3)_{#1}}
\newcommand\rnu{\check \nu}
\newcommand\rmu{\check \mu}
\newcommand\CP{\mathbb{P}^2(\mathbb{C})}

\newcommand\Mod{\textup{Mod}}

\renewcommand{\geq}{\geqslant}
\renewcommand{\leq}{\leqslant}

\newtheorem{theorem}{Theorem}[section]
\newtheorem{lemma}[theorem]{Lemma}
\newtheorem{corollary}[theorem]{Corollary}
\newtheorem{proposition}[theorem]{Proposition}

\theoremstyle{definition}
\newtheorem{definition}[theorem]{Definition}

\newtheorem{remark}[theorem]{Remark}

\newtheorem{claim}[theorem]{Claim}

\begin{document}

\title{Random walks, WPD actions, and the Cremona group}
\author{Joseph Maher, Giulio Tiozzo}

\begin{abstract}
We study random walks on the Cremona group.  We show that almost
surely the dynamical degree of a sequence of random Cremona
transformations grows exponentially fast, and a random walk produces
infinitely many different normal subgroups with probability $1$.
Moreover, we study the structure of such random subgroups.

We prove these results in general for groups of isometries of (non-proper) hyperbolic spaces which possess at least one WPD element. 
As another application, we answer a question of D. Margalit showing that a random normal subgroup of the mapping class group is free. 
\end{abstract}

\maketitle

\section{Introduction}

The \emph{Cremona group} is the group $G = \textup{Bir} \ \CP$ of
birational transformations of the projective plane.  Its study has
been initiated by E. De Jonqui\'eres, L. Cremona, and M. Noether in
the 1800's (see \cite{Fa} for a survey), and a great deal of progress
has been obtained in the last decade.  In particular, Cantat and Lamy
\cite{CL} proved a conjecture of Mumford, showing that the Cremona
group is not simple. In fact, they produced infinitely many different
normal subgroups.

A technique to produce many examples of a mathematical structure is to
use probability; indeed, even if it is hard to construct an
explicit example, it may be simpler to show that \emph{almost all}
objects satisfy the desired property (a famous example is
\emph{expander graphs}, see e.g. \cite{lubo-book}, Section 1.2).

In this paper, we prove the following strengthening of \cite{CL} by
looking at random walks.  To define a random walk, let us fix a
probability measure $\mu$ on the Cremona group, with countable
support.  Let us denote as $\Gamma_\mu$ the semigroup generated by the
support of $\mu$.  Then let us draw a sequence $(g_n)$ of elements
independently with distribution $\mu$, and consider the random product
$$w_n := g_1 g_2 \dots g_n.$$
We prove the following. 

\begin{theorem} \label{T:Cre2}
Let $\mu$ be a probability measure on the Cremona group $G = \textup{Bir} \ \CP$ so that $\Gamma_\mu$ is a primitive subgroup which contains a WPD element. 
For any sample path $\omega = (w_n)$, consider the normal closure $N_n(\omega) := \langle \langle w_n \rangle \rangle$.
Then we have: 
\begin{enumerate}
\item for almost every sample path $\omega$, the sequence 
$$( N_1(\omega), N_2(\omega), \dots, N_n(\omega), \dots )$$
contains infinitely many different normal subgroups of $\textup{Bir} \ \CP$.

\item Let the \emph{injectivity radius} of a subgroup $H < G$ be defined as
$$\textup{inj}(H) := \inf_{f \in H \setminus \{1 \}} \textup{deg }f.$$
Then, for any $R > 0$ the probability that $\textup{inj}(N_n) \geq R$ tends to $1$ as $n \to \infty$; 
\item 
The probability that the normal closure $\langle \langle w_n \rangle \rangle$ of $w_n$ in $G$ is free satisfies
$$\mathbb{P}(\langle \langle w_n \rangle \rangle \textup{ is free}) \to 1$$
as $n \to \infty$.
\end{enumerate}
\end{theorem}

We will in fact provide estimates on the rate of convergence in (3) (see Theorem \ref{T:normal-intro}), and discuss the non-primitive case in detail. 
Let us now introduce some definitions. 


\subsection{The dynamical degree}

Let $f : \CP \to \CP$ be a birational map. Then $f$ is given in homogeneous coordinates by
$$f([x : y : z]) := [P : Q : R]$$
where $P, Q, R$ are polynomials of degree $d$ without common factors. We call $d$ the \emph{degree} of $f$, and we denote it as $\textup{deg }f$.

Now, one notes that $\textup{deg}(f^{n+m}) \leq \textup{deg}(f^n) \cdot \textup{deg}(f^m)$, but the equality need not hold: the most famous example is 
the \emph{Cremona involution} 
$$g([x: y: z]) := [yz : xz : xy]$$ 
which has degree $2$, but $g^2$ is the identity; the Cremona group is in fact generated by degree $1$ transformations and the Cremona involution. 
Hence, following \cite{Fri}, \cite{RS} we define the \emph{dynamical degree} of $f$ as
$$\lambda(f) := \lim_{n \to \infty} \left( \textup{deg } f^n \right)^{1/n}.$$
The dynamical degree is always an algebraic integer \cite{DF}, and it is related to the topological entropy by $h_{top}(f) \leq \log \lambda(f)$. 
In fact, equality is conjectured \cite{Fri}. 

The Cremona group acts by isometries on an infinite dimensional hyperbolic space $\mathbb{H}_{\mathbb{P}^2}$
which is contained in the \emph{Picard-Manin space} (see Section \ref{S:Cremona}). 
Thus, Cremona transformations can be classified as \emph{elliptic}, \emph{parabolic}, or \emph{loxodromic} (\cite{DF} \cite{Can}). In particular, 
a Cremona transformation is \emph{loxodromic} if $\lambda(f) > 1$, and we say it is \emph{WPD} if it is loxodromic and not conjugate to a monomial transformation. A subgroup $\Gamma < G$ is \emph{primitive} if no non-trivial element of $\Gamma$ fixes the limit set $\Lambda(\Gamma) \subseteq \partial \mathbb{H}_{\mathbb{P}^2}$ pointwise. There are many such subgroups (see Remark \ref{R:prim}).

A measure $\mu$ on the Cremona group has \emph{finite first moment} if 
$\int \log \textup{deg }f \ d\mu(f)  < + \infty$, and is \emph{bounded} if there exists $D < +\infty$ such that $\textup{deg }f \leq D$ for any $f \in \textup{supp}(\mu)$. 
Moreover, it is \emph{non-elementary} if $\Gamma_\mu$ contains two loxodromic elements with disjoint fixed sets. 

We prove that the degree and dynamical degree of a random Cremona transformation grow exponentially fast.

\begin{theorem} \label{T:Cre1}
Let $\mu$ be a countable non-elementary probability measure on the
Cremona group with finite first moment.  Then there exists $L > 0$
such that for almost every random product $w_n = g_1 \dots g_n$ of
elements of the Cremona group we have the limit
$$\lim_{n \to \infty} \frac{1}{n} \log \textup{deg}(w_n) = L.$$
Moreover, if $\mu$ is bounded then for almost every sample path we have
$$\lim_{n \to \infty} \frac{1}{n} \log \lambda(w_n) = L.$$
\end{theorem}

Moreover, we obtain the following characterization of the Poisson boundary (see Section \ref{S:Poiss}). 

\begin{theorem} \label{T:Poisson}
Let $\mu$ be a non-elementary probability measure on the Cremona group with finite entropy and finite logarithmic moment, and 
suppose that $\Gamma_\mu$ contains a WPD element.
Then the Gromov boundary of the hyperboloid $\mathbb{H}_{\mathbb{P}^2}$ with the hitting measure
is a model for the Poisson boundary of $(G, \mu)$.
\end{theorem}

Note that  for simplicity we deal with the Cremona group over $\mathbb{C}$, but Theorems \ref{T:Cre1}, \ref{T:Cre2}, and \ref{T:Poisson} are still true (and with the same proofs) for the Cremona group over any algebraically closed field $k$. 

\subsection{General setup. WPD actions}

We will actually prove our results on the Cremona group under the more general framework of groups of isometries of non-proper 
hyperbolic spaces. 

Recall that a metric space $(X, d)$ is \emph{$\delta$-hyperbolic} if geodesic triangles are $\delta$-thin, 
and is \emph{proper} if closed balls are compact.
Let us consider a group $G$ acting by isometries on $X$. 
 
If the group $G$ is not hyperbolic, then it cannot admit a proper, cocompact action on a hyperbolic metric space, but there are many interesting actions on \emph{non-proper} hyperbolic metric spaces. 
Notable examples include \emph{relatively hyperbolic groups} which act on the coned-off Cayley graph (\cite{farb}, \cite{osin}); 
\emph{right-angled Artin groups}, acting on the extension graph \cite{kim-koberda2}; the \emph{mapping class group} of a surface, which acts on the \emph{curve complex} (\cite{mm1}, \cite{bow}); and  the group $\textup{Out} (F_n)$ of outer automorphisms of the free group (\cite{bestvina-feighn}, \cite{HM13}).

Recall that a $\delta$-hyperbolic space $X$ is equipped with the \emph{Gromov boundary} $\partial X$ given by asymptote classes of quasigeodesic rays. 
Under mild conditions on $\mu$, we proved in \cite{MT} that almost every sample path $(w_n x)$ converges to a point on the boundary
$\partial X$, and that the random walk has positive drift.

Since the spaces on which $G$ acts are not proper, some weak notion of properness is still needed in order to be able to extract information 
on the group from the action, and several candidate notions have been proposed in the last two decades. 

First of all, following \cite{sela}, \cite{bow}, \cite{osin}, the
action of a group $G$ on $X$ is \emph{acylindrical} if for any two
points $x, y$ in $X$ which are sufficiently far apart, the set of
group elements which coarsely fixes both $x$ and $y$ has bounded
cardinality.  More precisely, given a constant $K \ge 0$, we define
the \emph{joint coarse stabilizer} of $x$ and $y$ as
$$\textup{Stab}_K(x, y) := \{ g \in G \ : \ d(x, g x) \leq K \textup{
  and }d(y, g y) \leq K \}.$$
Then the action of $G$ on $X$ is acylindrical if for any $K \ge 0$,
there are constants $R(K)$ and $N(K)$ such that for all points $x$ and
$y$ in $X$ with $d(x, y) \ge R(K)$, we have the following bound (where
$\setnorm A$ is the cardinality of $A$):
\begin{equation}
\setnorm{ \stab_K(x, y) } \leq N(K).
\end{equation}
This condition is quite useful, and it is verified in certain important cases (e.g. the action of the mapping class group on the curve complex 
\cite{bow}, 
or the action of a RAAG on its extension graph \cite{kim-koberda2}). 

However, there are several interesting actions of groups on hyperbolic spaces which are not acylindrical; in particular, certain actions of 
$\textup{Out}(F_n)$ and of the Cremona group. 
For this reason, in this paper we will consider group actions which satisfy the \emph{weak proper discontinuity (WPD)} property, 
a weaker notion introduced by Bestvina and Fujiwara \cite{bestvina-fujiwara} in the context of mapping class groups.
Intuitively, an element is WPD if it acts properly on its axis.
In formulas, an element $g \in G$ is \emph{WPD} if for any $x \in X$ and any $K \geq 0$ there exists $N > 0$ such that 
\begin{equation}
\setnorm{ \stab_K(x, g^N x) } < + \infty.
\end{equation}
In other words, the finiteness condition is not required of all pairs of points in the space, but only of points along the axis of a given loxodromic element.

Let $\mu$ be a probability measure on the group $G$. We say that $\mu$ is \emph{countable} if the support of $\mu$ is countable, 
and we denote as $\Gamma_\mu$ the semigroup generated by the support of $\mu$. 
In this paper we show that as long as the semigroup $\Gamma_\mu$ contains \emph{at least one} WPD element, then 
generic elements have all the properness properties one could wish for. In particular, one can identify the Poisson boundary and study the normal 
closure of random elements.
As an application, we will use this condition to derive results on the Cremona group. 

Note that the action of the Cremona group on the infinite-dimensional hyperbolic space is not acylindrical, but WPD elements 
actually exist: in particular, by Shepherd-Barron \cite{SB}, a loxodromic map is WPD if and only if it is not conjugate to a monomial map (see also \cite{Ur}). 
Moreover, by (\cite{lonjou}, Proposition 4), for each $n \geq 2$,
the transformation given in affine coordinates by
$(x, y) \mapsto (y, y^n - x)$ is WPD.  

A related notion to WPD is the notion of \emph{tight} element from \cite{CL}. In fact, in order to produce new normal subgroups, 
Cantat and Lamy take the normal closure of tight elements. 
Let us note that in the Cremona group, centralizers of loxodromic elements are virtually cyclic; as a consequence, 
if an element is tight then it is also WPD. 

\subsection{Normal closure} \label{S:norma}

Let us now formulate our theorem on the normal closure for general WPD actions on a hyperbolic space.
In order to state the theorem, we need some assumptions. We call a measure $\mu$  \emph{reversible} if the semigroup $\Gamma_\mu$ 
generated by the support of $\mu$ is indeed a group. 
This condition is satisfied e.g. when the support of $\mu$ is closed under taking inverses. 
A measure $\mu$ on $G$ is \emph{admissible} with respect to an action on $X$ if it is countable, non-elementary, reversible, bounded, and WPD.

Given a subgroup $H < G$, we define its \emph{injectivity radius} as 
$$\textup{inj}(H) := \inf_{\stackrel{g \in H \setminus \{1 \}}{x \in X}} d(x, gx).$$

We prove that the injectivity radius of the normal closure of a random element is almost surely unbounded, 
and taking the normal closure of random elements yields many different normal subgroups. 

To be precise, let us denote as $\Lambda_\mu \subseteq \partial X$ the limit set of the group $\Gamma_\mu$, and 
$E_\mu := \{ g \in G \ :  \ gx = x \textup{ for all }x \in \Lambda_\mu \}$ the pointwise stabilizer of $\Lambda_\mu$. 
Note that if $G = \Gamma_\mu$, then $E_\mu = E(G)$ is the maximal finite normal subgroup of $G$ (i.e., the largest finite subgroup of $G$ which is normal: that such a subgroup exists is a consequence of the WPD property). 

Since $E_\mu$ is normal in $\Gamma_\mu$, conjugacy yields a 
homomorphism 
$$\Gamma_\mu \to \textup{Aut }E_\mu.$$
Let us denote as $H_\mu$ the image of $\Gamma_\mu$ in $\textup{Aut }E_\mu$. 
Then the \emph{characteristic index} $k(\mu)$ of $\mu$ is the cardinality of $H_\mu$. 

\begin{theorem}(Abundance of normal subgroups.)  \label{T:inj-intro}
Let $G$ be a group acting on a  Gromov hyperbolic space $X$,
and let $\mu$ be an admissible probability measure on $G$.  Let $k = k(\mu)$ be the characteristic index of $\mu$. 
Then, if we consider the normal closure
$N_n(\omega) := \langle \langle w_n^k \rangle \rangle$, we have:
\begin{enumerate}
\item for any $R > 0$, the probability that $\textup{inj}(N_n) \geq R$ tends to $1$ as $n \to \infty$; 
\item for almost every sample path $\omega$, the sequence 
$$( N_1(\omega), N_2(\omega), \dots, N_n(\omega), \dots )$$
contains infinitely many different normal subgroups of $G$.
\end{enumerate}
\end{theorem}

The characteristic index also determines the structure of the normal closure of a random element, in particular whether it is free.

\begin{theorem}(Structure of the normal closure.)  \label{T:normal-intro} 
Let $G$ be a group acting on a Gromov hyperbolic space $X$, and let $\mu$ be an admissible probability measure on $G$
with characteristic index $k(\mu)$. Then:
\begin{enumerate}
\item the probability that the normal closure $\langle \langle w_n \rangle \rangle$ of $w_n$ in $G$ is free satisfies
$$\mathbb{P}(\langle \langle w_n \rangle \rangle \textup{ is free}) \to \frac{1}{k(\mu)}$$
as $n \to \infty$.
\item
Moreover, if $k = k(\mu)$, then 
$$\mathbb{P}(\langle \langle w_n^k \rangle \rangle \textup{ is free}) \to 1$$
as $n \to \infty$, and indeed there exist constant $B > 0, c< 1$ such that
$$\mathbb{P}(\langle \langle w_n^k \rangle \rangle \textup{ is free}) \geq 1 - B c^{n}$$
for any $n$.
\end{enumerate}
\end{theorem}

Moreover, as a corollary of Theorem \ref{T:normal-intro},  
the probability that the normal closure of a random element is free detects the following algebraic property of the group:

\begin{corollary}
Let $G$ be a group acting on a Gromov hyperbolic space $X$, and let $\mu$ be an admissible probability measure on $G$.
If $\Gamma_\mu = G$, then 
$$\mathbb{P}\left( \langle \langle w_n \rangle \rangle \textup{ is free}\right) \to 1\qquad \textup{as }n \to \infty$$
if and only if the maximal finite normal subgroup $E(G)$ equals the center $Z(G)$.
\end{corollary} 

In particular, we will show later that this is the case for mapping class groups.

\begin{remark} \label{R:prim}
Let us note that it is not hard (e.g. in the Cremona group) to choose a measure $\mu$ such that $\Gamma_\mu$ is primitive, i.e. $k(\mu) = 1$. 
Indeed, let $f$ be a loxodromic WPD element. 
Let us now pick $g \notin E^+(f) = \textup{Stab}_G(\textup{Fix}(f))$. Then $E := E^+(f) \cap E^+(gfg^{-1})$ is a finite group. For each $g_i \in E$, the
set $\textup{Fix}(g_i)$ of fixed points of $g_i$ on the boundary of $\mathbb{H}_{\mathbb{P}^2}$ 
has codimension at least $1$ in $\partial \mathbb{H}_{\mathbb{P}^2}$. 
Now, pick a loxodromic $h$ such that $\textup{Fix}(h) \cap \cup_{i = 1}^r  \textup{Fix}(g_i) = \emptyset$. 
Then the group $\Gamma := \langle f, g, h\rangle$ is primitive. 
\end{remark}

\subsection{The Poisson boundary} \label{S:Poiss}

The well-known \emph{Poisson representation formula} expresses a duality between bounded harmonic functions on the unit disk and 
bounded functions on its boundary circle. Indeed, bounded harmonic functions admit radial limit values almost surely, while integrating a boundary function against the Poisson kernel gives a harmonic function on the interior of the disk. 

This picture is intimately connected with the geometry of $SL_2(\mathbb{R})$; then in the 1960's Furstenberg and others extended this duality 
to more general groups. In particular, let $G$ be a countable group of isometries of a Riemannian manifold $X$, and let us consider a 
probability measure $\mu$ on $G$. One defines $\mu$-harmonic functions as functions on $G$ which satisfy the mean value property 
with respect to averaging using $\mu$; in formulas $f : G \to \mathbb{R}$ is $\mu$-harmonic if 
$$f(g) = \sum_{h \in G} f(gh) \ \mu(h) \qquad \forall g \in G.$$
Following Furstenberg \cite{Furstenberg}, a measure space $(M, \nu)$ on which $G$ acts is then a boundary if there is a duality between bounded, $\mu$-harmonic functions on $G$ and $L^\infty$ functions on $M$. 

A related way to interpret this duality is by looking at random walks on $G$. In many situations, (e.g. when $X$ is hyperbolic) the space $X$ is equipped naturally with a topological boundary $\partial X$, and almost every sample path $(w_n x)$ converges to some point on the boundary of $X$. 
Hence, one can define the \emph{hitting measure} of the random walk 
as the measure $\nu$ on $\partial X$ given on a subset $A \subseteq \partial X$ by  
$$\nu(A) := \mathbb{P}\left(\lim_{n \to \infty} w_n x \in A \right).$$
A fundamental question in the field is then whether the pair $(\partial X, \nu)$ equals indeed the Poisson boundary of the random walk $(G, \mu)$, 
i.e. if all harmonic functions on $G$ can be obtained by integrating a bounded, measurable function on $\partial X$. 

In the proper case, the classical criteria in order to identify the Poisson boundary
can be applied and one gets that the Gromov boundary $(\partial X, \nu)$ with the hitting measure is a model for the Poisson boundary. 
In the non-proper case, the classical entropy criterion is not expected to work, as there may be infinitely many group 
elements contained in a ball of fixed diameter. 

We prove, however, that as long as $\Gamma_\mu$ contains a WPD element, the Poisson boundary indeed coincides with the Gromov boundary. 

\begin{theorem}(Poisson boundary for WPD actions.) \label{T:Poiss-WPD}
Let $G$ be a countable group which acts by isometries on a $\delta$-hyperbolic metric space $(X, d)$, and let $\mu$ be a non-elementary probability measure on $G$ with finite logarithmic moment and finite entropy.  Suppose that there exists at least one $WPD$ element $h$ in the semigroup generated by the support of $\mu$.  Then the Gromov boundary of $X$ with the hitting measure is a model for the Poisson boundary of the random walk $(G, \mu)$.
\end{theorem}

The result extends our earlier result in \cite{MT} for acylindrical actions. 


\subsection{Mapping class groups}

Let $S_{g, n}$ be a topological surface with genus $g$ and $n$ punctures, and let $Mod(S_{g,n})$ be its mapping class group, 
i.e. the group of homeomorphisms of $S_{g, n}$, up to isotopy. 
The mapping class group acts on a locally infinite, $\delta$-hyperbolic graph, known as the \emph{curve complex} \cite{mm1}.
Loxodromic elements for this action are the pseudo-Anosov mapping classes, and as they are all WPD elements, all results in our paper apply. 

As an application of Theorem \ref{T:normal-intro}, we prove that the normal closure of random 
mapping classes is a free group, answering a question of Margalit \cite{margalit}*{Problem 10.11}.

\begin{theorem} \label{T:mcg}
Let $G = Mod(S_{g,n})$ be the mapping class group of a surface of finite type, and suppose that $G$ is infinite. 
Let $\mu$ be a probability measure on $G$ with bounded support in the curve complex and such that $\Gamma_\mu = G$, 
and let $w_n$ be the $n^{th}$ step of the random walk generated by $\mu$.
Then the probability that the normal closure $\langle \langle w_n \rangle \rangle$ is free tends to $1$ as $n \to \infty$, 
with exponential decay.
\end{theorem}

The result follows from Theorem \ref{T:normal-intro} and the fact that, by the Nielsen realization theorem, the maximal normal subgroup 
of $Mod(S_{g, n})$ always equals its center (which is trivial unless the mapping class group contains a central hyperelliptic involution). See Section \ref{S:mcg} for details. 
Note that in fact the action is acylindrical \cite{bow}, hence some applications such as the Poisson boundary already 
follow from \cite{MT}.

\subsection{Outer automorphisms of the free group}

Another application of our setup is to the group $Out(F_n)$ of outer automorphisms of a finitely generated free group $F_n$
of rank $n \geq 2$. 

There are several hyperbolic graphs on which $Out(F_n)$ acts: the main two are the \emph{free factor complex} and the \emph{free splitting complex}. In particular, the free factor complex $\mathcal{FF}(F_n)$ is hyperbolic by work of Bestvina and Feighn \cite{bestvina-feighn}. Moreover, an element is loxodromic on $\mathcal{FF}(F_n)$ if and only if it is \emph{fully irreducible}, and all fully irreducible elements satisfy the WPD property.  However, it is not known whether the action of $Out(F_n)$ on the free factor complex is acylindrical.

On the other hand, the free splitting complex is also hyperbolic, but
the action on the free splitting complex $\mathcal{FS}(F_n)$ is known
not to be acylindrical, by work of Handel and Mosher \cite{HM13}.
Moreover, an element is loxodromic if and only if it admits a filling
lamination pair. This is a weaker condition than being fully
irreducible, and the stabilizer of a quasiaxis of a loxodromic
element need not be virtually cyclic. 

Thus, this is an example of an action for which 
not every loxodromic element satisfies the WPD property. 
However, by Theorem \ref{T:generic-WPD} even for this action
WPD elements are generic for the random walk.
Note that genericity of fully irreducible elements was already known by \cite{rivin2}. 

We have the following identification for the Poisson boundary of $Out(F_n)$.

\begin{theorem} \label{T:boundary-Out}
Let $\mu$ be a measure on $Out(F_n)$ such that the semigroup generated by the support of $\mu$ contains 
at least two independent fully irreducible elements. Moreover, suppose that $\mu$ has finite entropy and finite logarithmic moment 
for the simplicial metric on the free factor complex. 
Then the Gromov boundary of the free factor complex is a model for the Poisson boundary of $(G, \mu)$. 
\end{theorem}

\begin{proof}
By \cite{bestvina-feighn}, the action of fully irreducible elements on the free factor complex is WPD. Hence, the claim follows by Theorem \ref{T:Poiss-WPD}. \end{proof}

Note that the identification of the Poisson boundary for $Out(F_n)$
has been obtained by Horbez \cite{horbez} using the action of
$Out(F_n)$ on the outer space $CV_n$.  This gives an identification of
the Poisson boundary with both $\partial CV_n$ and
$\partial \mathcal{FF}(F_n)$, as there is a coarsely defined Lipschitz
map $CV_n \to \mathcal{FF}(F_n)$. In our theorem above, the moment
condition required is a bit weaker, as we only need the logarithmic
moment condition to hold with respect to the metric on
$\mathcal{FF}(F_n)$ instead of the metric on $CV_n$.

\subsection{Tame automorphism groups} 

Other groups arising in algebraic geometry admit an action on a non-proper $\delta$-hyperbolic space with WPD elements. 

First of all, the group $\textup{Aut}(\mathbb{C}^2)$ of polynomial automorphisms of $\mathbb{C}^2$ 
 (see \cite{furter-lamy} and references therein, as well as \cite{minasyan-osin}) can be written as an amalgamated product of two of its subgroups, hence it acts 
 on the corresponding Bass-Serre tree, which is a Gromov hyperbolic space; in fact, for this action every loxodromic element is WPD, but the action is not acylindrical. 

Remarkably, Lamy and Przytycki recently extended this work to three variables. They considered the \emph{tame automorphism 
group} $\textup{Tame}(\mathbb{C}^3)$, which is the group generated by affine and elementary automorphisms of $\mathbb{C}^3$ (see \cite{pry-lamy} for a precise definition), and showed that this group also acts on a Gromov hyperbolic complex and there are WPD elements, so the methods of the present paper apply.
  
Let us finally remark that much less is known about the structure of the Cremona group in three variables, 
and these methods do not easily apply since there is no immediate analog of the hyperboloid, as the Cremona group no longer 
preserves a quadratic form. 

\subsection{Genericity of WPD elements}

Maher \cite{Mah} and Rivin \cite{rivin} considered random walks on the
mapping class group acting on the curve complex, and showed that
pseudo-Anosov mapping classes are typical for random walks.  More
generally, in \cite{MT}, we showed that for a group $G$ acting
non-elementarily on a Gromov hyperbolic space $X$, loxodromic elements
are typical for the random walk: i.e., the probability that the random
product of $n$ elements is loxodromic tends to one as $n$ tends to
infinity. 

One of the ingredients in our proofs is that, as long as there is one WPD
element in the support of the measure generating the random walk, then
WPD elements are generic.

We say that a measure $\mu$ is \emph{non-elementary} if $\Gamma_\mu$
contains at least two independent loxodromic elements, and is
\emph{bounded} if for some $x \in X$ the set
$(gx)_{g \in \textup{supp }\mu}$ is bounded in $X$. Finally, $\mu$ is
\emph{WPD} if $\Gamma_\mu$ contains an element $h$ which is WPD in
$G$.

We will show that generic elements are WPD with an explicit bound on
the rate of convergence: we say that a sequence of numbers $(p_n)$
tends to $1$ with \emph{exponential decay} if there are constants
$B > 0$ and $c < 1$ such that $p_n \ge 1 - B c^{n}$.

\begin{theorem}(Genericity of WPD elements.) \label{T:generic-WPD} %
Let $G$ be a group acting on a  Gromov hyperbolic space $X$,
and let $\mu$ be a countable, non-elementary, bounded, WPD probability measure on
$G$.  Then
$$\mathbb{P}(w_n \textup{ is WPD}) \to 1$$ %
as $n \to \infty$, with exponential decay.
\end{theorem}

In fact, we obtain that most random elements have bounded coarse
stabilizer, where the bound does not depend on the point chosen. We
call this property \emph{asymptotic acylindricality}.  We prove the
following estimate on the joint coarse stabilizer.

\begin{theorem}(Asymptotic acylindricality.) \label{T:aa} %
Let $G$ be a group acting on a Gromov hyperbolic space $X$.  Let $\mu$
be a countable, non-elementary, bounded, WPD probability measure on
$G$, and let $x \in X$.  Then for any $K \geq 0$ there is an $N > 0$
such that
$$\mathbb{P}( \setnorm{ \textup{Stab}_K(x, w_n x) } \leq N ) \to 1,$$
with exponential decay. 
\end{theorem}

\subsection{Matching estimates and rates}

In order to obtain our results, we need to show that a random element has finite joint coarse stabilizer, 
and to do so we recur to what we call \emph{matching estimates}. 

Following \cite{calegari-maher}, we say that two geodesics $\gamma$ and $\gamma'$ in $X$ have 
a \emph{match} if there is a subsegment of $\gamma$ close to a $G$-translate of a subsegment of $\gamma'$ (see Definition \ref{D:match}). 
Let $x \in X$ be a basepoint and $(w_n)$ be a sample path. The two key estimates we will prove and use are the following. 

\begin{enumerate}

\item \emph{Matching estimate} (Proposition \ref{P:match-use}):
given a loxodromic element $g$, we show that the probability that the geodesic $[x, w_n x]$ has a match 
with a translate of the axis of $g$ is at least $1 - B c^{n}$. 

\item \emph{Non-matching estimate} (Proposition \ref{P:non-match}):
given a geodesic segment $\eta$ in $X$ of length $s$, the probability that there is a match between $[x, w_n x]$ 
and a $G$-translate of $\eta$ is at most $B c^{s}$. 
\end{enumerate}

\subsection{Asymmetric elements}

Another important tool in our proofs is the notion of asymmetric element, which was introduced in \cite{maher-sisto}. 
We call a loxodromic element $g \in G$ \emph{asymmetric} if any element which coarsely stabilizes a segment 
of the axis of $g$ actually coarsely stabilizes the set $\{ g^i x\}_{i \in \mathbb{Z}}$ (see Definition \ref{D:asym} for the precise statement). 
In \cite{maher-sisto} it is proven that if the action of $G$ is acylindrical, then asymmetric elements are generic. In this paper, we generalize this result 
to WPD actions, and use it to prove the other results.

Let $G_{WPD}$ be the set of WPD elements in $G$. 
For a loxodromic $g \in G$, let us denote as $\Lambda(g) := \{\lambda^+_g, \lambda_g^-\}$ the two fixed points of $g$ on $\partial X$. 
We denote as $E_G(g)$ the stabilizer of $\Lambda(g)$ as a set, and as $E_G^+(g)$ the pointwise stabilizer of $\Lambda(g)$.
Moreover, for a subgroup $H < G$ we denote as 
$$E_G(H) := \bigcap_{H \cap G_{WPD}} E_G(h)$$ 
the intersection of all $E_G(h)$ as $h$ lies in $H \cap G_{WPD}$ (a priori, this set may be smaller than the set 
of $WPD$ elements for the action of $H$ on $X$). Note that $E_G(G)$ is the maximal finite normal subgroup of $G$. 

We have the following characterization of $E_G(w_n)$ for generic elements $w_n$. Let $E_\mu := E_G^+(\Gamma_\mu)$. 

\begin{theorem}\label{T:Eplus}
Given $\delta \ge 0$ there are constants $K$ and $L$ with the
following properties.  Let $G$ be a group acting by isometries on a
$\delta-$hyperbolic space $X$, and let $\mu$ be a countable,
non-elementary, reversible, bounded, WPD probability distribution on
$G$.  Then there are constants $B > 0$ and $c < 1$ such that the
probability that $w_n$ is loxodromic, $(1, L, K)$-asymmetric, and WPD with
$$E_G(w_n) = E_G^+(w_n) =  \langle w_n \rangle \ltimes E_\mu$$ 
is at least $1 - Bc^{n}$.  
\end{theorem}

Note that the action of $E_\mu$ on $E_G(w_n)$ is precisely responsible for the value of $k$ in Theorems 
\ref{T:inj-intro} and \ref{T:normal-intro}. Indeed, one obtains that the cyclic group $\langle w_n \rangle$ is 
normal in $E_G(w_n)$ if and only if the image of $w_n$ in $\textup{Aut }E_\mu$ is trivial. 
Now, the random walk on $\Gamma_\mu$ pushes forward to a random walk on the finite group 
$\textup{Aut }E_\mu$, and this random walk equidistributes on the image of $\Gamma_\mu$ inside $\textup{Aut }E_\mu$, 
which we denote as $H_\mu$. This explains the asymptotic probability of $\frac{1}{\#H_\mu}$ in Theorem \ref{T:normal-intro}. 

\subsection{Further questions}
We conclude with a few questions for further exploration. 
\begin{enumerate}
\item
Can one drop ``reversible" as an hypothesis in Theorem \ref{T:normal-intro}? 
\item
Do our results still hold for measures $\mu$ with finite exponential moment, rather than bounded measures?
\item
Does the radius of injectivity $\textup{inj}(N_n)$ typically goes to
infinity as $n \to \infty$, and at what rate? 
\end{enumerate}
We believe that the answers to all these questions should be positive, but we do not attempt to solve them here. 

\subsection{Acknowledgements}

We would like to thank Mladen Bestvina for pointing out that the
Poisson boundary result from \cite{MT} holds in the WPD case.
We also thank Carolyn Abbott, Jeffrey Diller, Igor Dolgachev, Mattias
Jonsson, Stephane Lamy, Piotr Przytycki and Samuel Taylor 
for useful discussions and comments.

We would particularly like to thank the referee for a number of
insightful comments, which amongst many simplifications and
improvements, also enabled us to improve the rates from square root
exponential in the original version to exponential in the current
version.

The first named author acknowledges support from the Simons Foundation and PSC-CUNY.
The second named author is partially supported by NSERC and the Alfred P. Sloan Foundation.

\section{Background material}

Let $X$ be a $\delta$-hyperbolic metric space, and let $G$ be a group of isometries of $X$.
Let $\mu$ be a probability measure on $G$. This defines a \emph{random walk} by choosing for each $n$ an element $g_n$ of $G$ 
with distribution $\mu$ independently of the previous ones, and considering the product 
$$w_n := g_1 \dots g_n.$$
The sequence $(w_n)_{n \geq 0}$ is called a \emph{sample path} of the random walk, and we are interested in the asymptotic behavior 
of typical sample paths. 

\subsection{Isometries of hyperbolic spaces}

Recall that isometries of a $\delta$-hyperbolic space (even if it is
not proper) can be classified into three types (see \cite{gromov}, \cite{DSU}).  
In particular, $g \in \textup{Isom}(X)$ is:
\begin{itemize}
\item \emph{elliptic} if $g$ has bounded orbits; 
\item \emph{parabolic} if it has unbounded orbits, but zero translation length; 
\item \emph{loxodromic} (or \emph{hyperbolic}) if it has positive translation length. 
\end{itemize}
Here, the \emph{translation length} of $g \in \textup{Isom}(X)$ is the quantity 
\begin{equation} \label{eq:translation-length} %
\tau(g) := \lim_{n \to \infty} \frac{d(g^n x, x)}{n},
\end{equation}
where the limit always exists and is independent of the choice of $x$. 
Moreover, a loxodromic element has two fixed points on the Gromov boundary of $X$, one attracting and one repelling. 

A semigroup inside $\textup{Isom}(X)$ is \emph{non-elementary} if it contains two loxodromic elements which have disjoint fixed point sets 
on $\partial X$.  

We will use the following elementary properties of $\delta$-hyperbolic
spaces, which we state without proof.  A \emph{quasiaxis} for a
loxodromic isometry $g$ of $X$ is a quasigeodesic which is invariant
under $g$.  In fact, the constants may be chosen to depend only on
$\delta$, see for example Bonk and Schramm
\cite{bonk-schramm}*{Proposition 5.2} or Kapovich and Benakli
\cite{kapovich-benakli}*{Remark 2.16}.

\begin{proposition}\label{P:quasi-axis}
Given a constant $\delta \ge 0$, there is a constant $K_1$ such that
every loxodromic isometry of a $\delta$-hyperbolic space has a
$(1, K_1)$-quasiaxis.
\end{proposition}

To simplify notation, we will refer to a 
$(1, K_1)$-quasiaxis as a \emph{quasiaxis} for $g$.  

We will use the following fellow travelling properties of quasigeodesics 
in Gromov hyperbolic spaces.

The Morse lemma states that a quasigeodesic in a $\delta$-hyperbolic
space is contained in an $L$-neighborhood of a geodesic connecting its
endpoints, where $L$ depends only on $\delta$ and the quasigeodesic
constants.  The following result is a mild generalization of the Morse
lemma, and is exactly the Morse lemma if $K$ equals zero.  Given a
finite quasigeodesic $\gamma$ with endpoints $x$ and $y$, let
$\gamma^-_K := \gamma \setminus ( B_K(x) \cup B_K(y) )$.

\begin{proposition}[\cite{delzant}*{Proposition 1.3.3}] \label{P:fellow
  travel2} %
Given $\delta \ge 0$ and $K_1 \ge 0$, there is a constant $L$ such
that for any $K \geq 0$ and for any two $(1, K_1)$-quasigeodesics
$\gamma$ and $\eta$ in a $\delta$-hyperbolic space, with endpoints
distance at most $K$ apart, any point on $\gamma_{K}^-$ lies within
distance at most $L$ from a point on $\eta$.
\end{proposition}

We say a set $\gamma$ is $Q$\emph{-quasiconvex} if for any points $x$ and $y$
in $\gamma$, any geodesic $[x, y]$ is contained in a $Q$-neighbourhood
of $\gamma$.  Given a $Q$-quasiconvex set $\gamma$ in a hyperbolic
space $X$ and a point $x \in X$, let $\pi_\gamma(x)$ be a nearest
point on $\gamma$ to $x$. In a $\delta$-hyperbolic space, the nearest
point projection is not unique, but any two projections have uniformly
bounded distance, where the bound only depends on $\delta$ and $K$,
hence we shall always pick one and denote it $\pi_\gamma(x)$.

If two points $x$ and $y$ have nearest point projections
$\pi_\g(x)$ and $\pi_\g(y)$ which are sufficiently far apart, then the
piecewise geodesic running through $x$, $\pi_\g(x)$, $\pi_\g(y)$ and
then $y$, which we shall call a \emph{nearest point projection path}, 
is a quasigeodesic:

\begin{proposition}[\cite{cdp}*{Proposition 10.2.1}] \label{P:npp-qg} %
Given $\delta$ and $Q$, there are constants $L$ and $K$ such that for
any $\delta$-hyperbolic space $X$, for any $Q$-quasiconvex set
$\gamma$ in $X$, and any pair of points $x$ and $y$ in $X$, if
$d(\pi_\g(x), \pi_\g(y)) \ge L$, then the nearest point projection
path
$$[x, \pi_\g(x)] \cup [\pi_\g(x), \pi_\g(y)] \cup [\pi_\g(y), y]$$
is a $(1, K)$-quasigeodesic.
\end{proposition}

Let us recall that given $x, y \in X$ and $R \geq 0$, we define the
\emph{shadow} $S_x(y, R)$ as
$$S_x(y, R) := \{  z \in X \ : \ \gp{x}{z}{y} \geq d(x, y) - R \}.$$
The number $r = d(x, y) - R$ is called the \emph{distance parameter}
of the shadow.

\begin{proposition}\label{P:fellow travel3}
Given constants $\delta \ge 0, K_1 \ge 0$ and $R \ge 0$, there are
constants $D$ and $L$ with the following properties.  Let $x$ and $y$
be two points in a $\delta$-hyperbolic space $X$ with $d(x, y) \ge D$.
Let $A = S_x(y, R)$ and $B = S_y(x, R)$ be the corresponding shadows,
and let $\gamma$ be a $(1, K_1)$-quasigeodesic with one endpoint
in $A$ and the other endpoint in $B$.  Then any geodesic $[x, y]$ is
contained in an $L$-neighborhood of $\gamma$.
\end{proposition}

\begin{proof}
Let $p, q$ be the endpoints of $\gamma$, with $p \in A$ and $q \in B$,
and let $p'$, $q'$ be, respectively, a nearest point projection of $p$
to $[x, y]$ and of $q$ to $[x, z]$.  Then, by \cite[Proposition
2.4]{MT}, $d(p',y) \leq R + O(\delta)$ and
$d(q', z) \leq R + O(\delta)$.  We shall assume that we have chosen
$D \ge 2R + L_1 + O(\delta)$, where $L_1$ is the constant from
Proposition \ref{P:npp-qg}.  Then by Proposition \ref{P:npp-qg}, the
piecewise geodesic through $p, p', q'$ and $q$ is a quasigeodesic,
with quasigeodesic constants depending only on $\delta$.  As
quasigeodesics fellow travel, $[p, q]$ is contained in an
$L$-neighborhood of $\gamma$, where $L$ depends only on $\delta, K_1$
and $R$, as the quasigeodesic fellow travelling constants depend
only on $\delta$ and $K_1$.
\end{proof}

\subsection{Random walks on weakly hyperbolic groups}

In \cite{MT}, we established many properties of typical sample paths for random walks on general groups of isometries of $\delta$-hyperbolic spaces. 
Namely: 

\begin{theorem}[\cite{MT}] \label{T:MT}
Let $\mu$ be a countable, non-elementary measure on a group
of isometries of a $\delta$-hyperbolic metric space $X$, and let
$x \in X$.  Then
\begin{enumerate}
\item almost every sample path $(w_n x)$ converges to some point $\xi$ in the Gromov boundary of $X$; 
\item if $\mu$ has finite first moment in $X$, there exists $L > 0$ such that for almost all sample paths we have 
$$\lim_{n \to \infty} \frac{d(w_n x, x)}{n} = L;$$
\item moreover, if $\mu$ is bounded, there exists $L > 0, B \ge 0$ and $c < 1$ such that
the translation length grows linearly with exponential decay:
$$\mathbb{P}(\tau(w_n) \geq nL) \ge 1 - Bc^n.$$ 
\end{enumerate}
\end{theorem}

Note that in \cite{MT} the previous result is proven under the assumption that $X$ is \emph{separable}, i.e. it contains a countable dense set. 
However, since the measure $\mu$ is countable one can drop the separability assumption, as remarked in \cite[Remark 4]{GST}. 
In fact, the only point where separability is used is to prove convergence to the boundary, and one can prove it for general metric spaces 
from the separable case and the following fact. 

\begin{lemma}[\cite{GST}*{Remark 4}]
Let $\Gamma$ be a countable group of isometries of a $\delta$-hyperbolic metric space $X$. Then there exists a separable metric space $X'$ 
(in fact, a simplicial graph with countably many vertices) and a $\Gamma$-equivariant quasi-isometric embedding $i \colon X' \to X$. 
As a consequence, $i$ extends continuously to a $\Gamma$-equivariant inclusion $\partial X' \to \partial X$ between the Gromov boundaries.
\end{lemma}

By the theorem in the separable case, given $x' \in X'$ almost every sample path $(w_n x')$ converges to a point $\xi' \in \partial X'$, 
hence if $x = i(x')$ then almost every sample path $(w_n x)$ converges to $i(\xi') \in \partial X$, hence the random walk on $X$ 
converges almost surely to the boundary. 

Another ingredient in the proof of the previous theorem is the following lemma about the measure of shadows \cite{MT}*{Proposition 5.4}, which we will use in the later sections.

\begin{proposition}\label{P:positive}
Let $G$ be a non-elementary, countable group acting by isometries on a
Gromov hyperbolic space $X$, and let $\mu$ be a
non-elementary probability distribution on $G$. Then there is a number
$R_0$ such that if $g, h \in G$ are group elements such that $h$ and $h^{-1}g$ lie in the
semigroup generated by the support of $\mu$, then 
\[ \nu(\overline{S_{h x}( g x , R_0)}) > 0, \]
where $\overline{A}$ denotes the closure in $X \cup \partial X$. 
\end{proposition}

We will also use the well-known fact that in a Gromov hyperbolic space the complement of a shadow is approximately a shadow, 
as in the following proposition (see \cite{MT}, Proposition 2.4 and Corollary 2.5).

\begin{proposition} \label{P:nested}
Given non-negative constants $\delta, K$ and $L$, there are constants $C$ and $D$, such that
in any $\delta$-hyperbolic space $X$ we have: 
\begin{enumerate}
\item for any pair of points $x, y$ in $X$ and any $R \geq 0$ we have
$$X \setminus S_x(y, R) \subseteq S_y(x, d(x, y) - R + C);$$
\item
for any $R \ge 0$, and  any bi-infinite $(K, L)$-quasigeodesic $\gamma$, parameterized such that $\gamma(0)$ is a nearest point on $\gamma$
to the basepoint $x$, then for any shadow set
$V = S_{x}(\gamma(t), R)$ which does not contain $x$, with
$t \ge 0$, and for any point $y \in U = S_{x}(\gamma(t + D), R)$, we
have the inclusion
\[ X \setminus V \subseteq S_{y}(x, d(x, \gamma(t)) - R +
C). \]
\end{enumerate}
\end{proposition}

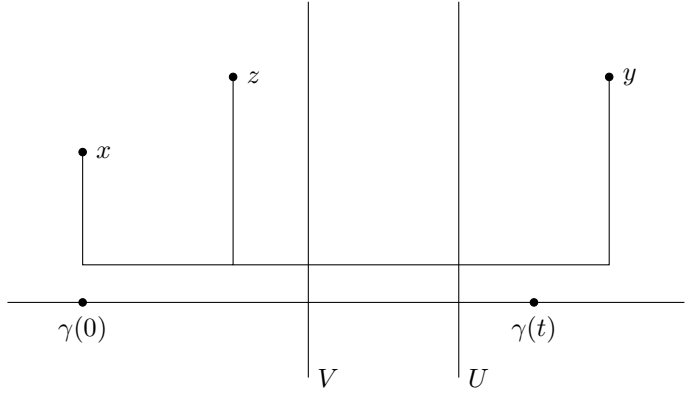
\begin{figure}[H]
\begin{center}
\begin{tikzpicture}

\tikzstyle{point}=[circle, draw, fill=black, inner sep=1pt]

\draw (-1, 0) -- (8, 0);

\draw (0, 2) -- (0, 0.5) -- (2, 0.5) -- (2, 3);
\draw (2, 0.5) -- (7, 0.5) -- (7, 3);

\draw (0, 2) node [point, label=right:{$x$}] {};
\draw (7, 3) node [point, label=right:{$y$}] {};
\draw (2, 3) node [point, label=right:{$z$}] {};
\draw (0, 0) node [point, label=below:{$\gamma(0)$}] {};
\draw (6, 0) node [point, label=below:{$\gamma(t)$}] {};

\draw (3, -1) node [right] {$V$} -- (3, 4);

\draw (5, -1) node [right] {$U$} -- (5, 4);

\end{tikzpicture}
\end{center}
\caption{The complement of a shadow is contained in a shadow.}\label{pic:atree}
\end{figure}

We will also use the following exponential decay estimates.
For $Y \subset X$ let $H^+(Y)$ denote the probability that the random
walk ever hits $Y$, i.e. that there is at least one index $n \in \N$ such
that $w_n x \in Y$.

\begin{lemma}[Exponential decay of shadows, \cite{maher-exp}*{Lemma
  2.10}] \label{L:exp} %
Let $G$ be a group which acts by isometries on a Gromov
hyperbolic space $X$, and let $\mu$ be a countable, non-elementary,
bounded probability measure on $G$.  Then there exist constants
$B > 0$ and $c < 1$ such that for any shadow $S_x(y, R)$ with distance parameter $r = d(x, y)  - R$,
we have the estimates
\begin{equation} \label{eq:exp nu}
\nu( \overline{S_x(y, R)} ) \le B c^r, 
\end{equation}
and
\begin{equation} \label{eq:exp hit}
H^+( S_x(y, R) ) \le B c^r.
\end{equation}
In particular, for all $n$:
\begin{equation} \label{eq:exp mu}
\mathbb{P}(w_n x \in S_x(y, R)) \leq B c^r.
\end{equation}
\end{lemma}

\noindent Indeed, equation \eqref{eq:exp nu} is \cite{maher-linear}*{Lemma 5.4}, and equation
\eqref{eq:exp hit} follows from  \eqref{eq:exp nu}  as in \cite{MT}*{Equation
  (5.3)}.  Equation \eqref{eq:exp mu} is an immediate consequence of
\eqref{eq:exp hit}.

Finally, we will also use the following positive drift, or linear progress, result.

\begin{proposition}[Exponential decay of linear progress,
\cite{maher-exp}] \label{L:progress-decay} %
Let $G$ be a group acting on a hyperbolic space $X$.  Let $\mu$ be a
countable, non-elementary measure on $G$ which has bounded support in
$X$. Then there exist constants $B > 0$, $L > 0$ and $0 < c < 1$ such
that for all $n$:
$$\mathbb{P}(d(x, w_n x) \leq Ln ) \leq B c^{n}.$$
\end{proposition}

\subsection{The Poisson boundary}

Given a countable group $G$ and a probability measure $\mu$ on $G$, one defines the space of bounded $\mu$-harmonic functions as 
$$H^\infty(G, \mu) := \left\{ f : G \to \mathbb{R} \textup{ bounded }\ : \ f(g) = \sum_{h \in G} f(gh) \mu(h) \ \forall g \in G \right\}.$$
Suppose now that $G$ acts by homeomorphisms on a measure space $(M, \nu)$. Then the measure $\nu$ is $\mu$-stationary 
if 
$$\nu = \sum_{h \in G} \mu(h)\ h_\star \nu.$$
A $G$-space $M$ with a $\mu$-stationary measure $\nu$ is called a \emph{$\mu$-boundary} if for almost every sample path $(w_n)$
the measure $w_n \nu$ converges to a $\delta$-measure. Given a $\mu$-boundary, one has the \emph{Poisson transform} 
$\Phi : L^\infty(M, \nu) \to H^\infty(G, \mu)$ defined as 
$$\Phi(f)(g) := \int_M f(g x) \ d\nu(x).$$

\begin{definition}
The space $(M, \nu)$ is the \emph{Furstenberg-Poisson boundary} of $(G, \mu)$ if the Poisson transform $\Phi$ is a
bijection between $L^\infty(M, \nu)$ and $H^\infty(G, \mu)$.
\end{definition}

It turns out that the Furstenberg-Poisson boundary is well-defined up
to $G$-equivariant measurable isomorphisms. 
Moreover, it is the maximal $\mu$-boundary in following sense: if $(B_{FP}, \nu_{FP})$ is the Furstenberg-Poisson boundary  and $(B, \nu)$
is another $\mu$-boundary, then there exists a $G$-equivariant measurable map $(B_{FP}, \nu_{FP}) \to (B, \nu)$.
Finally, such a boundary can be defined as the measurable quotient of the sample space of the random walk $(G, \mu)$ by identifying 
two sample paths if they eventually coincide (to be precise, one should cast this definition in the context of measurable partitions, 
as defined by Rokhlin \cite{Rok}).

\subsection{The strip criterion} 

In order to obtain the Poisson boundary for WPD actions, we will use Kaimanovich's \emph{strip criterion}. 
This basically says that if bi-infinite paths for the random walks can be approximated by subsets of $G$, called \emph{strips}, 
then one can conclude that the relative entropies of the conditional random walks vanish, hence the proposed geometric boundary 
is indeed the Poisson boundary.

Given a measure $\mu$ on $G$, its \emph{reflected measure} if $\rmu(g) := \mu(g^{-1})$. 
Moreover, we denote as $\rnu$ the hitting measure for the random walk associated to the reflected measure $\rmu$. 
We say that the measure $\mu$ has \emph{finite entropy} if 
$$H(\mu) := - \sum_{g \in G} \mu(g) \log \mu(g) < \infty.$$
Let $x \in X$ be a basepoint. The measure $\mu$ has \emph{finite logarithmic moment} if $\int_G \log^+ d(x, gx) \ d\mu(g) < \infty$.
Let us denote as 
$$B_G(g) := \{ h \in G \ : \ d(x, hx) \leq d(x, gx) \}.$$
We shall use the following \emph{strip criterion} by Kaimanovich.

\begin{theorem}[\cite{kaimanovich}] \label{T:strip}
Let $\mu$ be a probability measure with finite entropy on $G$, and let
$(\partial X , \nu )$ and $(\partial X , \rnu )$ be $\mu$- and
$\rmu$-boundaries, respectively. If there exists a measurable
$G$-equivariant map $S$ assigning to almost every pair of points $(\a
, \b ) \in \partial X \cross \partial X$ a non-empty ``strip'' $S(\a ,
\b ) \subset G$, such that for all $g$
\[ \frac{1}{n} \log \left| S(\a , \b ) g \cap B_G( w_n )\right| \to 0 \qquad \text{ as } n \to \infty, \]
for $\nu \cross \rnu $-almost every $(\a , \b ) \in \partial X
\cross \partial X$, then $(\partial X , \nu )$ and $(\partial X ,
\rnu )$ are the Poisson boundaries of the random walks $(G, \mu)$ and
$(G, \rmu)$, respectively.
\end{theorem}

\section{Background on the Cremona group} \label{S:Cremona}

We will start by recalling some fundamental facts about the Cremona group, and especially 
its action on the Picard-Manin space. For more details, see \cite{CL}, \cite{DF}, \cite{Fa} and references therein.

\subsection{The Picard-Manin space}

If $X$ is a smooth, projective, rational surface the group 
$$N^1(X) := H^2(X, \mathbb{Z}) \cap H^{1, 1}(X, \mathbb{R})$$
is called the \emph{N\'eron-Severi group}. Its elements are Cartier divisors on $X$ modulo numerical equivalence.
The intersection form defines an integral quadratic form on $N^1(X)$. 
We denote $N^1(X)_\mathbb{R} := N^1(X) \otimes \mathbb{R}$.

If $f : X \rightarrow Y$ is a birational morphism, then the pullback map $f^\star : N^1(Y) \to N^1(X)$ is injective
and preserves the intersection form, so $N^1(Y)_\mathbb{R}$ can be thought of as a subspace of $N^1(X)_\mathbb{R}$. 

A \emph{model} for $\CP$ is a smooth projective surface $X$ with a birational morphism $X \rightarrow \CP$. 
We say that a model $\pi' : X' \rightarrow \CP$ dominates the model $\pi : X \rightarrow \CP$ if the 
induced birational map $\pi^{-1} \circ \pi' : X' \dashrightarrow X$ is a morphism.
By considering the set $\mathcal{B}_X$ of all models which dominate $X$, 
one defines the space of \emph{finite Picard-Manin classes} as the injective limit
$$\mathcal{Z}(X) := \lim_{X' \in B_X} N^1(X')_\mathbb{R}.$$
In order to find a basis for $\mathcal{Z}(X)$, one defines an equivalence relation on the set of pairs $(p, Y)$ where $Y$ is a model of $X$ and $p$ a point in $Y$, 
as follows. One declares $(p, Y) \sim (p', Y')$ if the induced birational map $Y \dashrightarrow Y'$ maps $p$ to $p'$ 
and is an isomorphism in a neighbourhood of $p$. We denote the quotient space as $\mathcal{V}_X$. 
Finally, the \emph{Picard-Manin space} of $X$ is the $L^2$-completion 
$$\mathcal{Z}(X) := \left\{ [D] + \sum_{p \in \mathcal{V}_X} a_p [E_p] \ : \ [D] \in N^1(X)_\mathbb{R}, a_p \in \mathbb{R}, \sum_{p \in \mathcal{V}_X} a_p^2 < + \infty \right\}.$$
In this paper, we will only focus on the case $X = \mathbb{P}^2(\mathbb{C})$. Then the N\'eron-Severi group of $\CP$ 
is generated by the class $[H]$ of a line, with self-intersection $+1$. Thus, the Picard-Manin space is 
$$\overline{\mathcal{Z}}(\mathbb{P}^2) := \left\{ a_0 [H] + \sum_{p \in \mathcal{V}_{\CP}} a_p [E_p], \quad \sum_p a_p^2 < +\infty \right\}.$$
It is well-known that if one blows up a point in the plane, then the corresponding exceptional divisor has self-intersection $-1$, and 
intersection zero with divisors on the original surface. 

Thus, the classes $[E_p]$ have self-intersection $-1$, are mutually orthogonal, and are orthogonal to $N^1(X)$. Hence, 
the space $\overline{\mathcal{Z}}(\mathbb{P}^2)$ is naturally equipped with a quadratic form of signature $(1, \infty)$, 
thus making it a Minkowski space of uncountably infinite dimension. 
Thus, just as classical hyperbolic space can be realized as one sheet of a hyperboloid inside 
a Minkowski space, inside the Picard-Manin space one defines 
$$\mathbb{H}_{\mathbb{P}^2} := \{ [D] \in \overline{\mathcal{Z}}(\mathbb{P}^2) \ : \ [D]^2 = 1, [H] \cdot [D] > 0 \}$$
which is one sheet of a two-sheeted hyperboloid. The restriction of the quadratic intersection form to 
$\mathbb{H}_{\mathbb{P}^2}$ defines a Riemannian metric of constant curvature $-1$, thus making $\mathbb{H}_{\mathbb{P}^2}$
into an infinite-dimensional hyperbolic space. More precisely, the induced distance $\textup{dist}$ satisfies the 
formula 
$$\cosh \textup{dist}([D_1], [D_2]) = [D_1] \cdot [D_2].$$

Each birational map $f$ acts on $\overline{\mathcal{Z}}$ by orthogonal transformations. 
To define the action, recall that for any rational map $f : \CP \dashrightarrow \CP$ there exist a surface $X$ and morphisms $\pi, \sigma : X \to \CP$ 
such that $f = \sigma \circ \pi^{-1}$. Then we define $f^\star = (\pi^\star)^{-1} \circ \sigma^\star$, and $f_\star = (f^{-1})^\star$. 
Moreover, $f_\star$ preserves the intersection form, hence it acts as an isometry of $\mathbb{H}_{\mathbb{P}^2}$: in other words, the map 
$f \mapsto f_\star$ is a group homomorphism 
$$\textup{Bir} \ \CP \to \textup{Isom}(\mathbb{H}_{\mathbb{P}^2})$$
hence one can apply to the Cremona group the theory of random walks on groups acting 
on non-proper $\delta$-hyperbolic spaces.

The space $\mathbb{H}_{\mathbb{P}^2}$ is not separable; however, any countable subgroup of the Cremona group 
preserves a closed, totally geodesic, separable, subset of $\mathbb{H}_{\mathbb{P}^2}$ (see also \cite{DelzPy}, Remark 1).

\begin{definition} The \emph{dynamical degree} of a birational transformation $f : X \dashrightarrow X$ is defined as
$$\lambda(f) := \lim_{n \to \infty} \Vert (f^n)^\star \Vert^{1/n}$$
where $\Vert \cdot \Vert$ is any operator norm on the space of endomorphisms of $H^\star(X, \mathbb{R})$. 
\end{definition}

Note that $\lambda(f) = \lambda(gfg^{-1})$ is invariant by conjugacy. Moreover, if $f$ is represented by three homogeneous polynomials 
of degree $d$ without common factors, then the action of $f^\star$ on the class $[H]$ of a line is $f^\star([H]) = d [H]$, hence 
$$\lambda(f)  = \lim_{n \to \infty} \textup{deg}(f^n)^{1/n}.$$
Moreover, the degree is related to the displacement in the hyperbolic space $ \mathbb{H}_{\mathbb{P}^2}$: in fact, (see \cite{Fa}, page 17)
$$\textup{deg}(f) = f^\star [H] \cdot [H] = [H] \cdot f_\star[H] = \cosh d(x, fx)$$
if $x = [H] \in \mathbb{H}_{\mathbb{P}^2}$.
As a consequence, the dynamical degree $\lambda(f)$ of a transformation $f$ is related to its translation length $\tau(f)$ by the equation
(\cite{CL}, Remark 4.5): 
$$\tau(f) = \lim_{n \to \infty} \frac{\textup{dist}(x, f^n x)}{n} = \lim_{n \to \infty} \frac{\cosh^{-1} \textup{deg}(f^n)}{n} = \log \lambda(f).$$
Hence, a Cremona transformation $f$ is loxodromic if and only if $\lambda(f) > 1$. 

\section{Growth of translation length}

Let us now start by proving that for bounded probability measures translation length grows linearly along almost every sample path. 
This is a variation of \cite[Theorem 1.2]{MT} and \cite[Theorem 1.2]{dahmani-horbez}.

\begin{theorem}\label{T:growth}
Let $G$ be a group acting on a Gromov hyperbolic space $X$. 
Let $\mu$ be a countable non-elementary measure on $G$ whose support is bounded in $X$. Then for almost every sample path 
we have 
$$\lim_{n \to \infty} \frac{\tau(w_n)}{n} = L$$
where $L > 0$ is the drift of the random walk.
\end{theorem}

\begin{proof}[Proof of Theorem \ref{T:growth}]
Since the support is bounded in $X$, by Theorem \ref{T:MT}
there exists $L > 0$ such that almost surely
$$\lim_{n \to \infty} \frac{d(x, w_n x)}{n} = L.$$
Moreover, proceeding as in \cite[Section 5.8]{MT} and  using the exponential decay of shadows \cite[eq. (16)]{MT}, 
(see also \cite[proof of Prop. 2.6]{dahmani-horbez}),
there exist $B >0$ and $0 < c < 1$ such that for any $\epsilon >0$ we have  
\begin{equation}
\label{eq:exp-gproduct}
\mathbb{P}( \gp{x}{w_n x}{ w_n^{-1}x} \geq \epsilon n) \leq B c^{\epsilon n}.
\end{equation}
Now, by Borel-Cantelli, we obtain almost surely 
$$\lim_{n \to \infty} \frac{ \gp{x}{w_n x }{ w_n^{-1} x} }{n} = 0.$$
The claim then follows by using the well-known formula (see \cite{MT}, Appendix A)
$$\tau(g) = d(x, gx) - 2 \gp{x}{gx}{ g^{-1}x} + O(\delta).$$
\end{proof}

\section{WPD actions}

\subsection{The WPD condition}

Let $G$ be a group acting by isometries on a metric space $X$.  Recall that the action of $G$ on $X$ is \emph{proper} if the map $G \times X \to X \times X$ given by $(g, x) \mapsto (x, gx)$ is proper, i.e. the preimages of compact sets are compact.  A related notion is that the action is \emph{properly discontinuous} if for every $x \in X$ there exists an open neighbourhood $U$ of $x$ such that $gU \cap U \neq \emptyset$ holds for at most finitely many elements $g$.  If the space $X$ is not proper, it is very restrictive to ask for the action to be proper (for instance, point stabilizers for a proper action must be finite).  However, Bestvina-Fujiwara \cite{bestvina-fujiwara} defined the notion of \emph{weak proper discontinuity}, or \emph{WPD}; essentially, a loxodromic isometry $g$ is a WPD element if its action is proper in the direction of its axis.

\begin{definition}
Let $G$ be a group acting on a hyperbolic space $X$, and $h$ a loxodromic element of
$G$. One says that $h$ satisfies the \emph{weak proper discontinuity
  condition} (or $h$ is a WPD element) if for every $K > 0$ and every
$x \in X$, there exists $M \in \N$ such that 
\[ \setnorm{ \{g \in G \ : \  d(x, gx) < K, d(h^M x, gh^M x) < K \}
} < \infty. \]
\end{definition}
If we define the \emph{joint coarse stabilizer} of two points $x, y \in X$ as
$$\stab_K(x,y) := \{ g \in G \ : \ d(x, gx) \leq K \textup{ and } d(y,gy) \leq K \}$$
then the WPD condition says that for any $K$ and any $x$ there exists an integer $M$ such that 
$\stab_K(x, h^M x)$
is a finite set. A trivial consequence of the definition of WPD is the following.

\begin{lemma} \label{L:WPD} 
Let $G$ be a group acting on a Gromov hyperbolic space $X$, and
let $h$ be a WPD element in $G$. Then there are functions
$M_{W} \colon \R_{\ge 0} \to \N$ and
$N_{W} \colon \R_{\ge 0} \to \N$ such that for any $x \in X$, any
$K \geq 0$, and for any $f \in G$ one has
$$ \setnorm{ \stab_K(fx, fh^{M_{W}(K)} x) } \le N_{W}(K). $$
\end{lemma}

\begin{proof}
By definition, note that 
$$\stab_K(fx, fy) = f \stab_K(x, y) f^{-1}$$
hence the cardinality 
$$\setnorm{ \stab_K(fx, f h^M x) } = \setnorm{ f (\stab_K(x, h^M x) ) f^{-1} } =  \setnorm{ \stab_K(x, h^M x)} $$
is finite and independent of $f$, proving the claim. 
\end{proof}

Given a loxodromic element $g$, its associated \emph{maximal elementary subgroup} $E_G(g)$ is defined as the stabilizer of the
two endpoints of a quasiaxis of $g$, i.e.
$$E_G(g) = \stab^G(\{ \lambda_g^+, \lambda_g^-  \})$$  
(note that elements of $E_G(g)$ may permute the two fixed points).
We will use the following result due to
Bestvina and Fujiwara \cite{bestvina-fujiwara}*{Proposition 6}.

\begin{theorem}\label{T:bf}
Let $G$ act on $X$ with a WPD element $h$, with a 
quasiaxis $\alpha_h$.  Then $E_G(h)$ is the unique maximal virtually cyclic
subgroup containing $h$.  Furthermore, for any constant $K \ge 0$
there is a number $L$, depending on $h, \delta, K_1$ and $K$, such
that if $g \in G$ is an element which $K$-coarsely stabilizes a
subsegment of $\alpha_h$ of length $L$, then $g$ lies in $E_G(h)$.
\end{theorem}

That is, if $\alpha_h$ is a quasiaxis of a WPD element $h$, then
$$E_G(h) = \{ g \in G \ : \ d_{Haus}(g \alpha_h, \alpha_h) < \infty \}.$$ 
This is stated in \cite{bestvina-fujiwara} for a group action in which
all loxodromic elements are WPD, but the proof works for any group
acting non-elementarily on a Gromov hyperbolic space as long as $h$ is a WPD element. 

\section{The Poisson boundary}

Let us now use the WPD property to prove that the Poisson boundary coincides with the Gromov boundary, proving Theorem \ref{T:Poiss-WPD} in the Introduction. 

Similarly to \cite[Section 6]{MT}, the idea is to define appropriately the strips for Kaimanovich's criterion using ``elements of bounded geometry'' as below, and using 
the WPD condition to show that the number of elements in such strips grows at most linearly. 

The main difference is that we do not obtain a bound on the growth of all strips, but only on strips between \emph{almost all} pairs of boundary points. 
In fact, if $h$ is a WPD element, then one can use the WPD condition to obtain a bound of the number of bounded geometry elements in a ball
(see Lemma \ref{L:bound-geom} below). Moreover, by ergodicity, for almost every pair of boundary points, any $(1, K_1)$-quasigeodesic between them 
will fellow travel a translate of a quasiaxis of $h$, hence we can use the previous claim to bound the number of elements in any strip 
between almost every pair of boundary points. 

\subsection{Elements of bounded geometry}

Let $R \geq 0$ and $v \in G$. Then for any pair $(\alpha, \beta) \in \partial X \times \partial X$, with $\alpha \neq \beta$, define the set of \emph{bounded geometry elements} as
$$\mathcal{O}_{R, v}(\alpha, \beta) := \{ g \in G \ : \  \alpha \in \overline{ S_{gvx}(gx, R)} \textup{ and } \beta \in \overline{ S_{gx}(gvx, R) } \}.$$
An example of a bounded geometry element is illustrated below in Figure \ref{pic:bounded geometry}.
Note that for any $g \in G$ we have $\mathcal{O}_{R, v}(g \alpha, g \beta) = g \mathcal{O}_{R, v}(\alpha, \beta)$. Moreover, we define the ball in the group with respect to the metric on $X$ as  
$$B_G(y, r) := \{g \in  G \ : \ d(y, gx) \leq r \}$$
where $y \in X$ and $r \geq 0$.

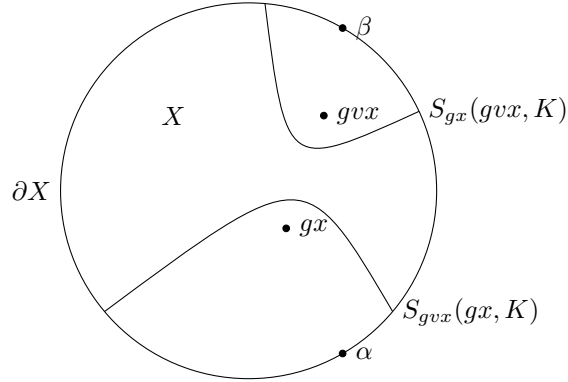
\begin{figure}[H] \begin{center}
\begin{tikzpicture}[scale=0.5]

\tikzstyle{point}=[circle, draw, fill=black, inner sep=0pt, minimum width=2.5pt]

\draw (0,0) circle (5cm);

\draw (60:5) node [point, label=right:$\beta$] {}; 
\draw (-60:5) node [point, label=right:$\alpha$] {};

\draw (85:5) .. controls (1, 0.5) .. (25:5) node [right] {$S_{g x}(g v
  x, R)$};

\draw (-140:5) .. controls (1.5, 0.75) .. (-40:5) node [right] {$S_{g v x}(g
  x, R)$};

\draw (1, -1) node [point, label=right:$g x$] {};

\draw (2, 2) node [point, label=right:$g v x$] {};

\draw (-2, 2) node {$X$};

\draw (-5, 0) node [left] {$\partial X$};

\end{tikzpicture}
\end{center} 
\caption{A bounded geometry element $g$ in $\mathcal{O}_{R, v}(\alpha,
  \beta)$.} \label{pic:bounded geometry}
\end{figure}

The most crucial property of bounded geometry elements is that their number in a ball grows linearly with the radius of the ball. 

\begin{proposition} \label{P:bound-geom} %
Let $G$ be a group acting on a hyperbolic space $X$, let $x \in X$,
and let $h$ be a WPD element. Then for any $R > 0$, there is a 
positive power $v = h^M$ of the WPD element and a constant $C$
such that for any radius $r > 0$ and any pair of distinct boundary points $\alpha, \beta \in \partial X$ one has 
$$\setnorm{ B_G(x, r) \cap \mathcal{O}_{R, v}(\alpha, \beta) } \leq C r.$$
\end{proposition}
This fact follows from the next lemma, which uses the WPD property in a crucial way. 

\begin{lemma} \label{L:bound-geom} %
Let $G$ be a group acting on a hyperbolic space $X$, let $x \in X$, and 
let $h$ be a WPD element. Then for any $R > 0$, there are positive
constants $L$, $M$ and $N$, such that if  $v = h^M$ then
$$\setnorm{  B_G(z, L) \cap \mathcal{O}_{R, v}(\alpha, \beta) } \leq N$$
for any $z \in X$ and any pair of distinct boundary points $\alpha, \beta$.
\end{lemma}

\begin{proof} 

We choose $L$ to be the maximum of the two fellow travelling constants
(both called $L$) in Propositions \ref{P:fellow travel2} and
\ref{P:fellow travel3}.  

We choose $M$ to be sufficiently large such
that $d(x, h^M x)$ is at least the separation distance $D$ from
Proposition \ref{P:fellow travel3} and $M$ is also at least
$M_W(10 L)$, where $M_W$ is the function arising from the WPD
condition in Lemma \ref{L:WPD}.  Finally, we choose $N$ to be
$N_W(10 L)$, where again $N_W$ is the function described in Lemma
\ref{L:WPD}.

Let us consider two elements $g, g'$ which belong to $\mathcal{O}(\alpha, \beta) \cap B_G(z, 4K)$.
Then if we let $f =  g' g^{-1}$, then 
\begin{equation} \label{E:d1}
d(gx, fgx) \leq 2 L.
\end{equation}
Let $\gamma$ be a $(1, K_1)$-quasigeodesic which joins $\alpha$ and $\beta$, and denote $S_1 := \overline{S_{gvx}(gx, R)}$, $S_2 := \overline{S_{gx}(gvx, R)}$. 
By construction, $\alpha$ belongs to both $S_1$ and $f S_1$ hence both $\alpha$ and $f \alpha$ belong to $f S_1$;
similarly, $\beta$ and $f \beta$ belong to $f S_2$. Hence, the two quasigeodesics $\gamma$ and $f \gamma $ have endpoints in $f S_1$ and $f S_2$, 
hence they must fellow travel in their middle: more precisely, by Proposition \ref{P:fellow travel3} they must pass within distance $L$ from both $fgx$ and $y:= fgvx$. 
Hence, if we call $q$ a nearest point to $fgx$ on $f \gamma$, we have $d(fgx, q) \leq L$. Moreover, if we call $p$ a nearest point on $\gamma$ to $y$, and $p'$ a nearest point on $f\gamma$ to $y$, we have 
$$d(p, p') \leq d(p, y) + d(y, p') \leq 2L.$$
Combining this with eq. \eqref{E:d1} we get 
$$|d(gx, p) - d(fgx, p')| \leq 4 L.$$
Moreover, since $f$ is an isometry we have $d(fgx, fp) = d(gx, p)$, hence 
\begin{equation} \label{E:dist}
|d(fgx, fp) - d(fgx, p')| \leq 4 L.
\end{equation}
Now, the points $q$, $p'$ and $fp$ both lie on the quasigeodesic
$f\gamma$; let us assume that $fp$ lies in between $q$ and $p'$, and
draw a geodesic segment $\gamma'$ between $q$ and $p'$, and let $p''$
be a nearest point projection of $fp$ to $\gamma'$ (the case where
$p'$ lies between $q$ and $fp$ is completely analogous).  By fellow
travelling (Proposition \ref{P:fellow travel2}), we have $d(fp, p'') \leq L$.  
Then, since $p', p''$ and $q$ lie on a geodesic,
we have
$$d(p', p'') = |d(q, p') - d(q, p'')| \leq $$
and by using eq. \eqref{E:dist}
$$\leq |d(fgx, p') - d(fgx, fp)| + d(fgx, q) + d(fgx, q) + d(fp, p'') \leq 4L + L + L + L$$
hence 
$$d(fp, p') \leq d(fp, p'') + d(p', p'') \leq 7L + L = 8 L  $$
and finally
$$d(y, fy) \leq d(y, p') + d(p', fp) + d(fp, fy) \leq L + 8 L + L = 10 L.$$
Thus 
$d(gvx, fgvx) = d(fgvx, f^2gvx) \leq 10 L$
hence 
$$f \in \stab_{10 L}( g x, gvx)$$
so by Lemma \ref{L:WPD} there are only $N = N_W(10 L)$ possible choices of $f$, as claimed.
\end{proof}

\begin{proof}[Proof of Proposition \ref{P:bound-geom}]
Let $\gamma$ be a $(1, K_1)$-quasigeodesic in $X$ which joins $\alpha$ and $\beta$. By definition, if $g$ belongs to $\mathcal{O}_{R, v}(\alpha, \beta)$, then $gx$ lies within distance  $\leq L$ of $\gamma$. Then one can pick points $(z_n)_{n \in \mathbb{Z}}$ along $\gamma$ such that any point of $\gamma$ is within distance $\leq L$ of some $z_n$. Then, any ball of radius $r$ contains at most $c r$ of such $z_n$, where $c$ depends only on $L$ and the quasigeodesic constant 
of $\gamma$. The claim then follows from Lemma \ref{L:bound-geom}.
\end{proof}

We now turn to the proof of Theorem \ref{T:Poiss-WPD}.
By Theorem \ref{T:MT}, we know that since both $\mu$ and its reflected measure $\check{\mu}$ are non-elementary, 
both the forward random walk and the backward random walk converge almost surely to points on the boundary of $X$. 
Thus, one defines the two boundary maps $\partial_\pm : (G^\mathbb{Z}, \mathbb{\mu}^\mathbb{Z}) \to \partial X$ as follows.
Let $\omega = (g_n)_{n \in \mathbb{Z}}$ be a bi-infinite sequence of increments, and define 
$$\partial_+(\omega) := \lim_{n \to \infty} g_1 \dots g_n x, \qquad \partial_-(\omega) := \lim_{n \to \infty} g_0^{-1} g_{-1}^{-1} \dots g_{-n}^{-1} x$$
the two endpoints of, respectively, the forward random walk and the backward random walk. 
Then choose $R \geq R_0$ as in Proposition \ref{P:positive} and $v = h^M$ as in Proposition \ref{P:bound-geom}. Define 
$$\mathcal{O}(\omega) := \mathcal{O}_{R, v}(\partial_+(\omega), \partial_-(\omega))$$
the set of bounded geometry elements along the $(1, K_1)$-quasigeodesic which joins $\partial_+(\omega)$ and $\partial_-(\omega)$.
Note that, if $T : G^\mathbb{Z} \to G^\mathbb{Z}$ is the shift in the space of increments, we have 
$$\mathcal{O}(T^n \omega) = \mathcal{O}(w_n^{-1} \partial_+(\omega), w_n^{-1} \partial_-(\omega)) = w_n^{-1} \mathcal{O}(\omega).$$
Now we will show that for almost every bi-infinite sample path $\omega$ the set $\mathcal{O}(\omega)$ is non-empty and has 
at most linear growth. 
In fact, by definition of bounded geometry,
$$p := \mathbb{P}(1 \in \mathcal{O}(\omega)) = \nu(\overline{S}) \rnu(\overline{S'}) > 0,$$
where $S = S_{vx}(x, R)$ and $S' = S_x(vx, R)$, and their measures are positive by Proposition \ref{P:positive}.
Moreover, since the shift map $T$ preserves the measure in the space of increments, we also have for any $n$
$$\mathbb{P}(w_n \in \mathcal{O}(\omega)) = \mathbb{P}(1 \in \mathcal{O}(T^n \omega)) = p > 0.$$
Thus, by the ergodic theorem, the number of times $w_n$ belongs to $\mathcal{O}(\omega)$ grows almost surely linearly with $n$: namely, 
for a.e. $\omega$ 
$$\lim_{n \to \infty} \frac{\setnorm{ \{1 \leq i \leq n \ : \ w_i \in \mathcal{O}(\omega)\} } }{n} = p > 0.$$
Hence the set $\mathcal{O}(\omega)$ is almost surely non-empty (in fact, it contains infinitely many elements).
On the other hand, by Proposition \ref{P:bound-geom} the set $\mathcal{O}(\omega)$ has at most linear growth, i.e.
there exists $C > 0$ such that for any $z \in X$ we have
\begin{equation} \label{E:lin-g}
\setnorm{ \mathcal{O}(\omega) \cap B_G(z, r) } \leq Cr \qquad \forall r > 0.
\end{equation}
The Poisson boundary result now follows from the strip criterion (Theorem \ref{T:strip}). 
Let $P(G)$ denote the set of subsets of $G$. Then, we define the strip map 
$S : \partial X \times \partial X \to P(G)$ as $S(\alpha, \beta) := \mathcal{O}_{R, v}(\alpha, \beta)$; 
hence, applying equation \eqref{E:lin-g} with $z = x$, $r = d(w_n x, x)$ we obtain
$$ \setnorm{ S(\alpha, \beta)g \cap B_G(w_n) } \leq C d(w_n x, x).$$
Then, since $\mu$ has finite logarithmic moment, one has almost surely 
$$\lim_{n \to \infty} \frac{1}{n} \log d(w_n x, x) \to 0,$$
which verifies the criterion of Theorem \ref{T:strip}, establishing that the Gromov boundary of $X$ is a model 
for the Poisson boundary of the random walk. 

\begin{remark} \label{remark:cobounded} %
We would like to thank the referee for pointing out an alternative
approach if the action of $G$ on $X$ is cobounded.  In this case,
Osin's \cite{osin} construction of the projection complex $Y$ may be
realized as a quotient of $X$, and the action of $G$ on $Y$ is
acylindrical.  Hence, one can use \cite{MT} to identify the Poisson
boundary of $(G, \mu)$ with the hitting measure on $\partial Y$.  One
may verify that the Lipschitz map from $X$ to $Y$ is alignment
preserving as defined by Dowdall and Taylor \cite{dowdall-taylor},
whose work then shows that the subset of $\partial X$ consisting of
quasigeodesic rays with infinite diameter image in $Y$ maps
injectively into $\partial Y$.  A quasigeodesic ray in $X$ has
infinite image in $Y$ if it fellow travels with infinitely many
distinct translates of a quasiaxis for the chosen WPD element, and this
happens for a full measure subset of $\partial X$ with respect to the
hitting measure.  Therefore $\partial X$ with the hitting measure is a
model for the Poisson boundary.  It is not clear to the authors how to
extend this argument to the non-cobounded case.
\end{remark}

\section{Genericity of WPD elements}\label{S:wpd}

Let $G$ be a group acting by isometries on a hyperbolic space $X$, and
let $\mu$ be a countable, non-elementary, bounded, WPD probability
distribution on $G$.  Let $h$ be a WPD element in $\Gamma_\mu$.  We
start by showing that the probability that a random walk gives a
WPD element tends to one exponentially quickly, and furthermore that
the probability that a quasiaxis of the WPD element fellow travels with a
translate of a quasiaxis of $h$ also tends to one exponentially quickly.

Before proceeding, we give a brief overview of the argument.  We wish
to show that a random walk on $G$ gives rise to a WPD element with
probability tending to one exponentially quickly.  Given a WPD element
$h$ with quasiaxis $\alpha_h$, one may construct a \emph{projection
  complex} $P$ on which $G$ also acts.  The projection complex is a
quasi-tree, hence hyperbolic, and has the property that any
element $g$ of $G$ which acts loxodromically on $P$ acts as a WPD
element on $X$, and furthermore, any quasiaxis $\alpha_g$ of $g$ has a
large subsegment which fellow travels with a translate of a quasiaxis
$\alpha_h$ of $h$.  Furthermore, this gives control over the size of
the joint stabilizer $\stab_K(x, gx)$, a property we call \emph{asymptotic acylindricality} (see Section \ref{S:aa}).
If $G$ acts non-elementarily on $X$, then it also acts non-elementarily
on $P$, so the fact that random walks on groups acting on hyperbolic
spaces give loxodromic elements with probability tending to one,
applied to the action on $P$, gives the required result.

The property that two axes have subsets that fellow travel each other
will be useful, and so we will use the following definition from
\cite{calegari-maher}, see also \cite{maher-sisto}.

\begin{definition}\label{D:match}
Let $G$ be a group acting on a Gromov hyperbolic space $X$.  Given
constants $K$ and $L$ we say that two geodesics $\gamma$ and $\gamma'$
in $X$ have an \emph{$(L, K)-$match} if there exist geodesic
subsegments $\alpha \subseteq \gamma$ and $\alpha' \subseteq \gamma'$
of length $\geq L$ and some $g \in G$ such that $g\alpha$ and
$\alpha'$ have Hausdorff distance $\leq K$.
\end{definition}

The main result of this section, from which we will derive Theorem \ref{T:generic-WPD}, is the following. 

\begin{theorem} \label{T:lox-implies-WPD} %
Let $G$ be a group with a non-elementary action by isometries on a
hyperbolic space $X$, and let $h$ be a WPD element for this action.
Then there is a constant $K$ with the following properties.  For any
$L$ there exists a non-elementary acylindrical action of $G$ on a quasi-tree $Y$ 
such that if $g \in G$ acts loxodromically on $Y$, then the action of
$g$ on $X$ is loxodromic and WPD.  Furthermore, a quasiaxis 
for $g$ in $X$ has an $(L, K)$-match with a quasiaxis for $h$ in $X$. 
\end{theorem}

This result is implicit in the constructions of projection complexes
in \cites{bbf, osin, DGO, balasubramanya, bbfs} and thus is likely
well-known. As we are unable to find a reference in the literature, in
the next two subsections we provide a proof using published results of
\cite{bbfs}. 

\subsection{Projection complexes}

We now review the projection complex construction from \cite{bbf},
\cite{bbfs}.  We do not give complete details, but we state precisely the properties we use.

Let $h$ be a WPD element. In general, $h$ may
only have an invariant quasiaxis in $X$, which is coarsely preserved
by $E_G(h)$.  However, $X$ embeds quasi-isometrically inside a
hyperbolic space $X'$ such that $h$ has an invariant \emph{geodesic} axis
$A_h$ which is preserved by $E_G(h)$, see for example
\cite{bbf-arxiv}*{Lemma 4.9}\footnote{This result appears in the
  initial \url{arXiv} version but was omitted from the published
  version \cite{bbf2}.}. 
In the rest of this section, we will assume that the action of $G$ on
$X$ has the property that $E_G(h)$ preserves a 
geodesic axis. At the end, we will remark that we can obtain the same results
for general actions $X$ with different constants by using the
quasi-isometry between $X$ and $X'$.

Note that, as in this section we need to consider distances in several metric spaces, 
we will write $d_X$ instead of $d$ for the distance in $X$.
We say that a collection of geodesics in $X$ has \emph{$D$-bounded
  projections} if for any two distinct geodesics $A$ and $B$ in the
collection, the nearest point projection $\pi_A(B)$ has diameter at
most $D$.  Let $\A$ be the set of distinct translates of $A_h$ under
$G$.  As $h$ is a WPD element, there is a constant $D$ such that $\A$
has $D$-bounded projections, see for example \cite{bbf}*{Theorem H}.
Given three distinct elements $A, B$ and $C$ of $\A$, define
\[ d_C(A, B) := \textup{diam} \{ \pi_C(A) \cup
\pi_C(B) \}.  \]
Bestvina, Bromberg and Fujiwara \cite{bbf} define a \emph{projection
  complex} $P_L(\A)$, which is a graph whose vertices are elements of
$\A$, and in which two distinct vertices $A$ and $B$ are connected by
an edge if $d_C(A, B) \le L$ for all $C \in \A \setminus \{ A, B\}$.
In fact, their construction is more general, 
but we shall restrict attention to a version that applies 
in the context of WPD actions. We shall give $P_L(\A)$ the natural path
metric in which every edge has length one, and we shall denote this
metric by $d_{P_L}$.

Bestvina, Bromberg and Fujiwara \cite{bbf} showed that $P_L(\A)$ is a
quasi-tree for all $L$ sufficiently large.  Osin \cite{osin} defined a
slightly different space on which the action of $G$ is acylindrically
hyperbolic, and Balasubramanya \cite{balasubramanya} showed that this
construction could be modified to guarantee that the space is a
quasi-tree.  In fact, we shall use the version from Bestvina,
Bromberg, Fujiwara and Sisto \cite{bbfs}, which also construct a
projection complex which is a quasi-tree and on which $G$ acts
acylindrically.

The result below summarizes the properties of the construction that we
use.

\begin{theorem} \cite{bbfs}*{Theorem 4.1, Theorem 3.10} %
Let $G$ be a group acting by isometries on a $\delta$-hyperbolic space
$X$.  Let $h$ be a WPD element such that $E_G(h)$ preserves a
geodesic axis $A_h$.  Let $\mathcal{A}$ be the set of
distinct translates of $A_h$ under $G$, and suppose that $\A$ has
$D$-bounded projections.  The nearest point projection maps $\pi_A$
may be replaced with maps $\pi'_A$ such that for all $A$ and $B$ in
$\A$, $\pi'_A(B) \subseteq N_D(\pi_A(B))$, and for all $L$
sufficiently large, the projection complex $P_L(\A)$, constructed
using the modified projection maps $\pi'_A$, is a quasi-tree on which
$G$ acts acylindrically.  Furthermore, if $G$ acts non-elementarily on
$X$, then it acts non-elementarily on $P_L(\A)$.
\end{theorem}

We will also use the following result from \cite{bbfs}, which shows
that the distance in $P_L(\A)$ between two axes $A$ and $B$ is
coarsely equivalent to the number of other axes to which $A$ and $B$
have large diameter projections.

\begin{theorem} \cite{bbfs}*{Corollary 3.7} %
Let $G$ be a group acting by isometries on a $\delta$-hyperbolic space
$X$.  Let $h$ be a WPD element such that $E_G(h)$ preserves a
geodesic axis $A_h$.  Let $\mathcal{A}$ be the set of
distinct translates of $A_h$ under $G$.  For $A, B \in \mathcal{A}$,
denote
\[ Y_L(A, B) := \{ C \in \A \setminus \{ A, B \} \mid \pi'_C(A, B) \ge L\}. \]
Then there is a constant $L_0$ such that for all $L \ge L_0$, the
metric on $P_L(\A)$ is coarsely equivalent to the number of elements
in $Y_L(A, B)$.  In fact, for $A \not = B$,
\[ \lfloor \tfrac{1}{2} (\# \norm{ Y_L(A, B) } + 1) \rfloor + 1 \le
d_{P_L}(A, B) \le \# \norm{ Y_L(A, B) } + 1. \]
\end{theorem}

We next show that loxodromic isometries of the projection complex act
as WPD elements on $X$.

\subsection{Loxodromic isometries of the projection complex}

We will make use of the following elementary result.

\begin{lemma} \label{L:large projection} %
Let $\alpha$ and $\beta$ be two geodesic segments in a
$\delta$-hyperbolic space $X$, of length at least $L$, contained in
$K$-neighbourhoods of each other.  Then
$$\textup{diam } \pi_\alpha(\beta)  \ge L - 4 K.$$
\end{lemma}

\begin{proof}
Let $a$ be an endpoint of $\alpha$.  As $\alpha \subset N_K(\beta)$,
there is a point $b \in \beta$ such that $d_X(a, b) \le K$.  Let $a'$
be a nearest point on $\alpha$ to $b$, so $d_X(a', b) \le K$.  By the
triangle inequality, $d_X(a, a') \le 2K$.  Applying the same argument
to the other endpoint of $\alpha$ implies that the diameter of
$\pi_\a(\b)$ is at least $L - 4K$.
\end{proof}

We now show that distance in $P_L(\A)$ is a coarse lower bound for the
distance between elements of $\A$ in $X$.

\begin{proposition} \label{P:pc distance} %
Let $G$ be a group acting on a $\delta$-hyperbolic space $X$, and let
$h$ be a WPD isometry such that $E_G(h)$ preserves a geodesic axis
$A_h$.  Let $\A$ be the collection of distinct translates of $A_h$
under $G$, with $D$-bounded projections.  Then there are constants $K$
and $Q > 0$ with the following properties.

There exists $L_0$ such that for all $L \ge L_0$,
and for any $A$ and $B$ in $\A$, the distance $d_{P_L}$ in the
projection complex $P_L(\A)$ is a coarse lower bound for distance in
$X$, i.e.
\begin{equation}\label{eq:pc-lower-bound}
d_X(A, B) \ge Q d_{P_L}(A, B) - Q.
\end{equation}
Furthermore, any shortest geodesic $[a, b]$ from $A$ to $B$ in $X$ has
an $(L, K)$-match with the axis of $h$.
\end{proposition}

\begin{proof}
We give a brief outline of the argument.  Let $A$ and $B$ be two
elements of $\A$, and let $\gamma_1$ be a shortest path from $A$ to
$B$ in $X$.  Their distance $d_{P_L}(A, B)$ in the projection complex
is coarsely equal to the number of elements in $Y_L(A, B)$, the
collection of $C \in \A$ to which the projections of $A$ and $B$ are
distance at least $L$ apart.  This means that the nearest point
projection path from $A$ to $B$ via $C$ in $X$ is a quasigeodesic, and
so the shortest path $\gamma_1$ from $A$ to $B$ fellow travels with
$C$ distance roughly $L$.  However, as the collection of geodesics in
$\A$ has $D$-bounded projections, the fellow travelling segments of
translates of $A_h$ can't overlap too much along $\gamma_1$, so this
gives a lower bound on the length of $\gamma_1$, which is linear in
the number of elements in $Y_L(A, B)$, and hence linear in
$d_{P_L}(A, B)$.

We now give the details of this argument.  Recall that Proposition
\ref{P:npp-qg} says that if two points have nearest point projections
to a geodesic that are distance at least $L_1$ apart, then the nearest
point projection path is a $(1, K_1)$-quasigeodesic, where $K_1$ and
$L_1$ depend only on $\delta$.  Furthermore, by Proposition \ref{P:fellow
  travel2}, there is a constant $K_2$ such that if two
$(1, K_1)$-quasigeodesics have common endpoints, then their Hausdorff
distance is at most $K_2$.  Here $K_2$ depends on $\delta$ and $K_1$,
but as $K_1$ only depends on $\delta$, $K_2$ only depends on $\delta$.

Choose $L_0 = 9D + 8K_2 + L_1$, and let $L \geq L_0$. 
Let $\gamma_1 = [a, b]$ be a shortest
path from $A$ to $B$ in $X$.  We may assume that $A \not = B$ and so
$d_{P_L}(A, B) \ge 1$, thus by the definition of $P_L(\A)$ there is at
least one $C \in \A$ such that $d_C(A, B) \ge L$.  This implies that
$d_X(\pi_C(a), \pi_C(b)) \ge L - 2D \ge L_1$, so by Proposition
\ref{P:npp-qg}, the nearest point projection path
$\gamma_2 = [a, \pi_C(a)] \cup [\pi_C(a), \pi_C(b)] \cup [\pi_C(b),
b]$ is a $(1, K_1)$-quasigeodesic.

By our choice of $K_2$, $\gamma_1$ and $\gamma_2$ are contained in
$K_2$-neighbourhoods of each other.  The segment
$[\pi'_C(a), \pi'_C(b)]$ has length at least $L$, and so
$[\pi_C(a), \pi_C(b)]$ has length at least $L - 2D > L_1$.  As the
nearest point projection path is a $(1, K_1)$-quasigeodesic it is
contained in a $K_2$-neighbourhood of $\gamma_1$, and so
$[\pi'_C(a), \pi'_C(b)]$ is contained in a $(K_2 + D)$-neighbourhood
of $\gamma_1$.  As $C$ is a translate of the axis $A_h$, this implies
that the geodesic $\gamma_1 = [a, b]$ has an $(L, K)$-match with
$A_h$, giving the final statement of the result with $K = K_2 + D$.

The choice of $C$ in $Y_L(A, B)$ was arbitrary, so for every $C$ in
$Y_L(A, B)$, the geodesic $\gamma_1 = [a, b]$ $K$-fellow travels with
$C$ distance at least $L$.  If the number of elements of $Y_L(A, B)$
is at least $2 D d_X(A, B)/L + 1$, then there are at least two
distinct translates $C$ and $C'$ of $A_h$ which have subsegments of
length at least $L/2$ which $K$-fellow travel.  By Lemma
\ref{L:large projection}, the nearest point projection of $C$ to $C'$
has diameter at least $L/2 - 4K$.  Our choice of $L_0$ ensures that
$L/2 - 4 K > D$, which contradicts the fact that elements of $\A$ have
$D$-bounded projections.  Therefore
\[ d_X(A, B) \ge \frac{L}{2 D} ( \# \norm{ Y_L(A, B) } - 1 ), \]
and so the result follows by choosing $Q$ equal to $L/2D$.
\end{proof}

We now show that if an isometry acts loxodromically on the
projection complex $P_L(\A)$, then it acts loxodromically on $X$.

\begin{corollary}\label{C:lox-is-lox}
Let $G$ be a group acting by isometries on a hyperbolic space $X$,
with a WPD element $h$ such that $E_G(h)$ preserves a geodesic axis.
Let $P_L(\A)$ be the corresponding projection complex determined by
$h$.  Then for all $L$ sufficiently large, if $g$ acts loxodromically
on $P_L(\A)$, then $g$ acts loxodromically on $X$.
\end{corollary}

\begin{proof}
Recall that if $g$ is a loxodromic isometry of $P_L(\A)$, then the
translation length of $g$ is positive, i.e. $\tau_{P_L}(g) > 0$.  
Let $A \in \A$ and $a$ be a point on the axis $A$.
We observe that $d_X(a, g^n a) \ge d_X(A, g^n A)$, as $a$
lies in $A$.
Choosing $L \ge L_0$, where $L_0$ is the constant from
Proposition \ref{P:pc distance}, we may apply
\eqref{eq:pc-lower-bound} to the pair $A$ and $g^n A$ and obtain
$d_X(A, g^n A) \ge Q d_{P_L}(A, g^n A) - Q$.  
Moreover, by
definition of translation length \eqref{eq:translation-length}, 
$d_{P_L}(A, g^n A) \ge n \tau_{P_L}(g)$
for any $A \in \A$ and any $n \geq 0$.  Hence
$$d_X(a, g^n a) \ge Q n \tau_{P_L}(g) - Q.$$
Dividing by $n$ and taking the limit as $n \to \infty$
shows that $\tau_X(g) \ge Q \tau_{P_L}(g) > 0$, and
so the action of $g$ on $X$ is loxodromic, as required.
\end{proof}

\begin{corollary}\label{cor:axis-match}
Let $G$ be a group acting by isometries on a hyperbolic space $X$,
with a WPD element $h$ such that $E_G(h)$ preserves a geodesic axis,
and let $P_L(\A)$ be the corresponding projection complex determined
by $h$.  Then there are constants $K$ and $L_0$, such that for all
$L \ge L_0$, if $g$ acts loxodromically on $P_L(\A)$, then $g$ acts
loxodromically on $X$.  Furthermore, a quasiaxis of $g$ has an
$(L, K)$-match with the axis of $h$.
\end{corollary}


\begin{figure}[H]
\begin{center}
\begin{tikzpicture}

\tikzstyle{point}=[circle, draw, fill=black, inner sep=1pt]

\begin{scope}[yscale=-1, xscale=0.8, xshift=-5cm]
\draw (2, -3.5) .. controls (2, -2.5) and (2, -2.5) .. (3, -2.5) --
(5, -2.5) node (a) [point, midway, label=above:$a$] {}
.. controls (6, -2.5) and (6, -2.5) .. (6, -3.5) node [right] {$A$};
\end{scope}

\begin{scope}[yscale=-1, xscale=0.8, xshift=+7cm]
\draw (2, -3.5) .. controls (2, -2.5) and (2, -2.5) .. (3, -2.5) --
(5, -2.5) node (ga) [point, midway, label=above:$g^n a$] {}
.. controls (6, -2.5) and (6, -2.5) .. (6, -3.5) node [right] {$g^n A$};
\end{scope}

\begin{scope}[yscale=-1, xscale=0.8, xshift=+1cm]
\draw (2, -3.5) .. controls (2, -2.5) and (2, -2.5) .. (3, -2.5) node
(pa) [point, label=above:$\pi_C(a)$] {} --
(5, -2.5) node (pga) [point, label=above:$\pi_C(g^n a)$] {}
.. controls (6, -2.5) and (6, -2.5) .. (6, -3.5) node [right] {$C$};
\end{scope}

\draw [thick] (a) .. controls (0, 1.5) and (2.5, 1.5) .. (pa) -- (pga)
 .. controls (5, 1.5) and (8.5, 1.5) .. (ga);

\draw [thick] (a) .. controls (-0.5, 0.5) and (-0.5, 0.5) .. (0, 0.5)
-- (8, 0.5) .. controls (8.5, 0.5) and (8.5, 0.5) .. (ga);

\draw (-2, 0) -- (10, 0) node [below] {$\alpha_g$};
\draw (-0.5, 0) node [point, label=below:{$p$}] {};
\draw (8.5, 0) node [point, label=below:{$q$}] {};

\end{tikzpicture}
\end{center}
\caption{The geodesic from $a$ to $g^n a$, and the nearest point
  projection path via $C$.}\label{pic:axis}
\end{figure}

\begin{proof}
Let $L_1$ be the maximum
of the fellow travelling constant $L$ from Proposition \ref{P:fellow
  travel2} for $(1, K_1)$-quasigeodesics, and the constant $L$ from
Proposition \ref{P:npp-qg}, such that if two points have nearest point
projections to a geodesic distance at least $L$ apart, then the
nearest point projection path is a $(1, K_1)$-quasigeodesic.  Let $D$
be a constant such that the geodesics in $\A$ have $D$-bounded
projections.  Finally, choose $L_0$ sufficiently large such that
Corollary \ref{C:lox-is-lox} holds, and furthermore, choose $L_0 \ge
2D + L_1$.

Let $A$ be an axis for $h$, and let $a$ be a point on $A$.  As $g$
acts loxodromically on $P_L(\A)$, the distance $d_{P_L}(A, g^n A)$
tends to infinity as $n$ tends to infinity.  By
\eqref{eq:pc-lower-bound}, there is an $n$ sufficiently large such
that $Y_L(A, g^n A)$ is non-empty.  Let $C$ be an element of
$Y_L(A, g^n A)$.  Recall that by the definition of the projection
complex, the diameter of $\pi_C'(A) \cup \pi'_C(g^n C)$ is at least
$L$.  The image of the modified projection maps $\pi'_C$ is contained
within a $D$-neighbourhood of the nearest point projection maps
$\pi_C$, so $d_X( \pi_C(a), \pi_C(g^n a)) \ge L - 2D$.  By our choice
of $L_0$, $L - 2D \ge L_1$, so by Proposition \ref{P:npp-qg}, the
nearest point projection path
$\eta = [a, \pi_C(a)] \cup [\pi_C(a), \pi_C(g^n a)] \cup [\pi_C(g^n
a), g^n a]$ is a $(1, K_1)$-quasigeodesic.  In particular, there is a
segment $[\pi_C(a), \pi_C(g^n a)]$ of length at least $L - 2D$
contained in an $L_1$-neighbourhood of any geodesic $[a, g^n a]$.

By Proposition \ref{P:quasi-axis}, there is a $(1, K_1)$-quasigeodesic
$\alpha_g$ in $X$, which is a quasiaxis 
for $g$ acting on $X$.  Let
$p$ be a nearest point on 
$\alpha_g$ to $a$, and let $q$
be a nearest point on $\alpha_g$ to $g^n a$.  As $g$ is an isometry,
$d_X(a, p) = d_X(g^n a, g^n p)$.  The point $g^n p$ lies on
$g^n \alpha_g$, which by Proposition \ref{P:fellow travel2}, is
contained in an $L_1$-neighbourhood of $\alpha_g$, and so
$d_X(g^n a, q) \le d_X(a, p) + L_1$, which in particular is
independent of $n$.

Therefore, by Proposition \ref{P:fellow travel2}, outside an
$(d_X(a, p) + L_1)$-neighbourhood of its endpoints, the geodesic
$[a, g^n a]$ is contained in an $L_1$-neighbourhood of $\alpha_g$.  By
\eqref{eq:pc-lower-bound} the number of geodesics in $\A$ which may
have segments of length $L$ which $K$-fellow travel a geodesic of
length $(d_X(a, p) + L_1)$ is at most $(d_X(a, p) + L_1)/Q$.  In
particular, for $n$ sufficiently large, there is an element $C$ in
$Y_L(A, g^n A)$ which has a subsegment of length at least $L - 2D$
contained in an $L_1$-neighbourhood of $[a, g^n a]$, distance at least
$d_X(a, p) + L_1$ from its endpoints, and hence contained in an
$2L_1$-neighbourhood of $\alpha_g$.  The translate $C$ of $A_h$ then
has an $(L, K)$-match with $\alpha_g$ for $K = 2L_1 + 2D$.
\end{proof}

Recall that the following (a priori weaker) definition, which we shall
refer to as \emph{axial WPD}, is equivalent to WPD.

\begin{definition} \label{D:axial} %
Let $G$ be a group acting on a $\delta$-hyperbolic space $X$, and let
$h$ be a loxodromic isometry with a quasiaxis 
$\alpha_h$.  Then $h$ is
an \emph{axial WPD} if there exists $p \in \alpha_h$ such that for any
constant $K \ge 0$, there is an $M > 0$, such that
\[ \setnorm{ \stab_K(p) \cap \stab_K(h^M p)  } < \infty. \]
\end{definition}

\begin{lemma} \label{L:axial} %
Let $G$ be a group acting on a $\delta$-hyperbolic space $X$, and let
$h$ be a loxodromic isometry.  Then $h$ is an axial WPD if and only if
$h$ is WPD.
\end{lemma}

\begin{proof}
If $h$ is WPD, then it is an axial WPD.  We now show the other
direction.  By the triangle inequality, for any $x, y \in X$,
$g \in G$, and $K \geq 0$
$$\stab_K(y) \cap \stab_K(h^M y) \subseteq \stab_{K'}(x) \cap \stab_{K'}(h^M x)$$
where $K' = K + 2 d(x, y)$. 
\end{proof}

We now show that for $L$ sufficiently large, loxodromics on $P_L(\A)$
act as WPD elements on $X$.

\begin{proposition} \label{P:LW2}
Let $G$ be a group acting by isometries on a hyperbolic space $X$,
with a WPD element $h$ so that $E_G(h)$ preserves a geodesic axis, 
and let $P_L(\A)$ be the corresponding projection complex determined by $h$.  Then there is a constant $L_0$
such that for all $L \ge L_0$, if $g$ acts loxodromically on $P_L(\A)$,
then $g$ acts as a WPD element on $X$.
\end{proposition}

\begin{proof}
Let $A_h$ be the geodesic axis of $h$ in $X$, 
and let $\alpha_g$ be a
$(1, K_1)$-quasiaxis for $g$ in $X$.  Let $p$ be a nearest point on
$\alpha_g$ to $A_h$, and let $K$ be a constant.

The group $G$ acts on both $X$ and $P_L(\A)$.  We will write
$\stab^X_K(x)$ for the coarse stabilizer of a point $x \in X$ and
$\stab^P_K(A)$ for the coarse stabilizer of a point $A \in P_L(\A)$.

Let $f$ be an isometry such that
$f \in \stab^X_K(p) \cap \stab^X_K(g^m p)$.
In particular, by the triangle inequality and the fact that $p$ is a nearest point projection,
$d_X(A_h, f A_h) \le 2 d_X(A_h, \alpha_g) + K$, and similarly,
$d_X(g^m A_h, f g^m A_h) \le 2 d_X(A_h, \alpha_g) + K$. Using
\eqref{eq:pc-lower-bound} implies that for
$K' = (2 d_X(A_h, \alpha_g) + K)/Q$,
\[ f \in \stab^{P}_{K'}(A_h) \cap \stab^P_{K'}(g^m A_h).  \]
The isometry $g$ acts as a WPD element on $P_L(\A)$, and let
$M_{W}$ and $N_{W}$ be the corresponding functions from Lemma
\ref{L:WPD}.  For all $m \ge M_{W}(K')$ there are at most
$N_{W}(K')$ elements $f$.  Therefore $g$ acts as an axial WPD
element on $X$, hence by Lemma \ref{L:axial} as a WPD element, as required.
\end{proof}

Theorem \ref{T:lox-implies-WPD} now follows immediately from Corollary \ref{cor:axis-match} and Proposition \ref{P:LW2}
in the case that $E_G(h)$ preserves a geodesic axis; the general case follows as discussed 
by replacing $X$ with a quasi-isometric space $X'$ on which $E_G(h)$ preserves a geodesic axis. 

\subsection{WPD isometries are generic}

We may now prove the following slightly stronger form of Theorem
\ref{T:generic-WPD}.

\begin{theorem} \label{T:generic match} %
Let $G$ be a group acting on a Gromov hyperbolic space $X$, and let
$\mu$ be a countable, non-elementary, bounded, WPD probability measure
on $G$. Then there exist constants $B > 0$, $c < 1$ such that the probability that 
$w_n$ is WPD satisfies
$$\mathbb{P}(w_n \text{ is WPD}) \geq 1 - B c^n $$ 
for any $n$. 

Furthermore, for any WPD element $h \in \Gamma_\mu$, there is a constant $K$,
such that for all $L \geq 0$, the probability that a quasiaxis 
for $w_n$ has an $(L, K)$-match with a quasiaxis 
for $h$ tends to one as $n \to \infty$, with exponential decay.
\end{theorem}

\begin{proof}
Let $h \in \Gamma_\mu$ be a WPD element and let $K$ be given by Theorem \ref{T:lox-implies-WPD}. 
For any $L \geq 0$, let $Y$ be the quasi-tree given by Theorem \ref{T:lox-implies-WPD}.
As $\mu$ is bounded in $X$, it is also bounded in $Y$.  As
$\Gamma_\mu$ contains $h$ and acts non-elementarily on $X$, it also
acts non-elementarily on $Y$.  A bounded non-elementary random walk on
a group acting on a Gromov hyperbolic space gives rise to a loxodromic
element with probability tending to one with exponential decay, by
\cite{MT}.  If $w_n$ is loxodromic on $Y$, then it is WPD on $X$, as
required.  The final statement follows immediately from the final
statement in Theorem \ref{T:lox-implies-WPD}.
\end{proof}

\section{Asymptotic acylindricality} \label{S:aa}

We say that a group $G$ acting by isometries on a Gromov hyperbolic
space $X$ is \emph{acylindrical} if for all $K \ge 0$, there are constants
$R \ge 0$ and $N \ge 0$, such that for all points $x$ and $y$ in $X$,
with $d(x, y) \ge R$, one has the bound
\[ \setnorm{ \stab_K(x) \cap \stab_K(y)  } \le N. \]

\begin{definition}
Let $\mu$ be a probability measure on a group $G$ acting by isometries
on a metric space $X$, and let $x \in X$.  
We say that the random walk generated
by $\mu$ is \emph{asymptotically acylindrical} if there is a function
$N_{ac} \colon \R_{\ge 0} \to \R_{\ge 0}$ such that for all $K \ge 0$,
the probability that
\[ \setnorm{ \stab_K(x) \cap \stab_K(w_n x)  } \le N_{ac}(K) \]
tends to one as $n$ tends to infinity.
\end{definition}

\subsection{More matching estimates}

We now show that for any WPD element $h$ in $\Gamma_\mu$, the
probability that $[x, w_n x ]$ has an $(L, K)$-match with a translate
of a quasiaxis $\alpha_h$ of $h$ tends to one as $n$ tends to infinity.

The following results are analogous to \cite[Propositions 3.2]{maher-sisto},
where the action is assumed to be acylindrical.

\begin{proposition}\label{P:match-use}
Let $G$ be a group acting by isometries on a Gromov hyperbolic space
$X$ with a WPD element $h$, with quasiaxis
$\alpha_h$.  Let
$x$ be a basepoint in $X$.  Then there is a constant $K_0$ such that
for any countable, non-elementary, WPD probability distribution $\mu$
on $G$, which is bounded in $X$, the following properties hold.

\begin{enumerate}

\item If $w_n$ is loxodromic, then let $\alpha_{w_n}$ be a quasiaxis
for $w_n$, and let $p$ be a nearest point on
$\alpha_{w_n}$ to the basepoint $x$. Then for any $K \ge K_0$ and any
$L \ge 0$, there are constants $B_1 > 0 $ and $c_1 < 1$ such that the
probability that $w_n$ is loxodromic and that $[p, w_n p]$ has an
$(L, K)$-match with $\alpha_h$ is at least $1 - B_1 c_1^{n}$.

\item There is a constant $K$ such that for any $L \ge 0$, there are
constants $B_2 > 0 $ and $c_2 <1$ such that the probability that
$\gamma_n = [x, w_n x]$ has an $(L, K)$-match with $\alpha_h$ is at
least $1 - B_2 c_2^{n}$.

\end{enumerate}

\end{proposition}

\begin{proof}


Let $K$ be the constant from Theorem \ref{T:generic match}. Then, for
any $L \ge 0$, Theorem \ref{T:generic match} implies that the
probability that $\alpha_{w_n}$ and $\alpha_h$ have a $(L, K)$-match
tends to one with exponential decay.  This is illustrated in Figure
\ref{pic:axis2} below.

\begin{figure}[H]
\begin{center}
\begin{tikzpicture}

\tikzstyle{point}=[circle, draw, fill=black, inner sep=1pt]

\begin{scope}[yshift=+2cm, xshift=-4cm]
\draw [thick] (2, -3.5) .. controls (2, -2.5) and (2, -2.5) .. (3,
-2.5) -- (5, -2.5) .. controls (6, -2.5) and (6, -2.5) .. (6, -3.5)
node [right] {$\alpha_h$};
\end{scope}

\draw (0, 2) node (x) [point, label=above:$x$] {};
\draw (8, 2) node (wx) [point, label=above:$w_n x$] {};

\draw [thick] (x) .. controls (0, 0.5) and (0, 0.5) .. (1, 0.5)
-- (7, 0.5) .. controls (8, 0.5) and (8, 0.5) .. (wx);

\draw [thick] (-2, 0) -- (10, 0) node [below] {$\alpha_{w_n}$};
\draw (0, 0) node (p) [point, label=below:{$p$}] {};
\draw (8.5, 0.5) node (wp) [point, label=below:{$w_n p$}] {};

\draw (-2, 0.25) -- (p) -- (wp) -- (10, 0.25) node [above] {$\gamma$};

\end{tikzpicture}
\end{center}
\caption{The quasiaxis $\alpha_{w_n}$ has an $(L, K)$-match with the
  quasiaxis $\alpha_h$.}\label{pic:axis2}
\end{figure}

By Proposition \ref{P:fellow travel2}, there exists $L_1$, which only
depends on $\delta$, such that any two $(1, K_1)$-quasigeodesics with
common endpoints are contained in $L_1$-neighbourhoods of each other.
The point $w_n p$ lies on the quasiaxis 
$w_n \alpha_{w_n}$,
which is contained in an $L_1$-neighbourhood of $\alpha_{w_n}$.  In
particular, the distance from $w_n p$ to $\alpha_{w_n}$ is at most
$L_1$, and so again, by Proposition \ref{P:fellow travel2}, the
geodesic $[p, w_n p]$ is contained in a $2 L_1$-neighbourhood of
$\alpha_{w_n}$.  Let $\gamma$ be the orbit of $[p, w_n p]$ under
powers of $w_n$. Then $\gamma$ is a connected bi-infinite quasiaxis
for $w_n$, contained in an $2 L_1$-neighbourhood of $\alpha_{w_n}$.
Let $q$ be a nearest point projection of $x$ to $\gamma$, and let
$q'$ be a nearest point projection of $w_n x$ to $\gamma$.  As
$\alpha_{w_n}, w_n \alpha_{w_n}$ and $\gamma$ are all contained in
$2 L_1$-neighbourhoods of each other, $d_X(p, q) \le 2 L_1$ and
$d_X(w_n p, q') \le 2 L_1$.

By Proposition \ref{P:npp-qg}, there are $L_2$ and $K_2$,
which only depend on $\delta$, such that if $d_X(q, q') \ge L_2$ ,
then the nearest point projection path
$[x, q] \cup [q, q'] \cup [q', w_n x]$ is a $(1, K_2)$-quasigeodesic.
As $[q, q']$ and $[p, w_n p]$ are Hausdorff distance $2 L_1$ apart,
there are constants $K_3$ and $L_3$, which only depend on $\delta$,
such that if $d_X(p, w_n p) \ge L_3$, then the path
$[x, p] \cup [p, w_n p] \cup [w_n p, w_n x]$ is a
$(1, K_3)$-quasigeodesic.

The distance $d_X(p, w_n p)$ is at least the translation length
$\tau(w_n)$.  By Theorem \ref{T:MT}, the translation length grows
linearly with exponential decay, so the probability that
$d_X(p, w_n p) \ge L_3$ tends to one with exponential decay.
Therefore the probability that the path
$[x, p] \cup [p, w_n p] \cup [w_n p, w_n x]$ is a
$(1, K_3)$-quasigeodesic tends to one with exponential decay.

If $\alpha_h$ has an $(L, K)$-match with $\alpha_{w_n}$, then it has
an $(L, K + 2 L_1)$-match with $\gamma$.  If this match is disjoint
from the orbit of $p$ under powers of $w_n$, then we are done.  If the
orbit of $p$ is contained in the match, then, at worst, $p$ divides the
subsegment of $\gamma$ realizing the match in two equal parts, so the
probability that $[p, w_n p]$ has an $(L/2, K + 2 L_1)$-match with
$\alpha_h$ tends to one exponentially quickly.  This gives the first
statement of the result, for appropriate choices of constants.

For the second statement, the path $[p, w_n p]$ is a subsegment of the
$(1, K_3)$-quasigeodesic $[x, p] \cup [p, w_n p] \cup [w_n p, w_n x]$.
By Proposition \ref{P:fellow travel2}, there is a constant $L_4$,
which only depends on $\delta$, such that $[p, w_n p]$ is contained in
an $L_4$-neighbourhood of $[x, w_n x]$.  Therefore, the
$(L/2, K + 2 L_1)$-match with $[p, w_n p]$ gives an
$(L/2, K + 2 L_1 + L_4 )$-match with $[x, w_n x]$, as required.
\end{proof}

Finally, we show:

\begin{lemma} \label{L:match2} %
Let $G$ be a group acting on a Gromov hyperbolic space $X$.  Let $\mu$
be a countable, non-elementary, bounded, WPD probability distribution
on $G$, and let $h$ be a WPD element in $G$ which lies in
$\Gamma_\mu$. Then there is a constant $K_0$ such that for any
$\epsilon > 0$, any $K \geq K_0$, and any $L > 0$ there are constants
$B > 0$ and $c < 1$ such that the probability that every segment
$[w_i x, w_{i + \epsilon n} x]$ for $0 \leq i \leq n(1-\epsilon)$ has
a $(L, K)$-match with a translate of a 
quasiaxis of $h$ is at least $1 - B c^{n}$.
\end{lemma}

\begin{proof}
By Proposition \ref{P:match-use}, for each $i$ the probability that
$[w_i x, w_{i + \epsilon n} x]$ does not have a $(L, K)$-match with a
translate of a
quasiaxis of $h$ is at most $B_1 c_1^{\epsilon n}$ for
some $c_1 < 1$, and there are at most $n (1-\epsilon)$ possible values
of $i$, hence the total probability is at most
$B_1 (1-\epsilon) n c_1^{\epsilon n}$.  The result then follows for
suitable choices of $B$ and $c$.
\end{proof}

\subsection{Proof of asymptotic acylindricality}

We now show that if $\Gamma_\mu$ contains a WPD element, then the
random walk determined by $\mu$ is asymptotically acylindrical with
exponential decay, which is Theorem \ref{T:aa} in the Introduction.

\begin{theorem} \label{T:asymp-acyl} %
Let $G$ be a group acting by isometries on a Gromov hyperbolic space
$X$, let $x \in X$, and let $\mu$ be countable, non-elementary,
bounded, WPD probability distribution on $G$.  Then for any
$K \geq 0$, there are constants $N > 0$, $B >0 $ and $c < 1$ such that
\[ \P \left( \setnorm{ \textup{Stab}_K(x, w_n x)  } \le N \right) \ge
1 -  B c^{n}. \]
\end{theorem}

\begin{proof}
Without loss of generality we assume that $E_G(h)$ preserves a geodesic axis; 
the general case follows as before by replacing the space $X$ by a quasi-isometric space $X'$ 
and changing constants. 
Recall that distance between elements of $A$ in $X$ is a coarse upper bound for the
distance in $P_L(\A)$.  So if an isometry coarsely stabilizes $x$ in
$X$, then it coarsely stabilizes $A_h$ in $P_L(\A)$.  By linear
progress with exponential decay, the distance $d_{P_L}(A_h, w_n A_h)$
grows linearly with exponential decay.  As the action of $G$ on
$P_L(\A)$ is acylindrical, the probability that the coarse stabilizer
of $A_h$ and $w_n A_h$ is bounded tends to one exponentially quickly,
so this also holds for the coarse stabilizer of $x$ and $w_n x$.

We now make this precise.  As the action of $G$ on the projection
complex $P_L(\A)$ is acylindrical, there are functions $R_{ac}$ and
$N_{ac}$ such that for all $K \ge 0$, and all $A$ and $B$ in $P_L(\A)$
with $d_{P_L}(A, B) \ge R_{ac}(K)$, we have
\[ \# \norm{ \stab^P_K(A, B) } \le N_{ac}(K). \]
Let $A_h$ be the geodesic axis of $h$. 
Then if $d_X(x, f x) \le K$, then by the triangle inequality
$d_X(A_h, f A_h) \le K + 2 d_X(x, A_h)$.  Recall that by Proposition
\ref{P:pc distance}, distance between elements of $A$ in $X$ is a
coarse upper bound for the distance in $P_L(\A)$.  In particular, if we set
$K' = (K + 2 d_X(x, A_h))/Q$, where $Q$ is from
Proposition \ref{P:pc distance}, then $d_{P_L}(A_h, f A_h) \le K'$.

This implies that if $f \in \stab^X_K(x, w_n x)$, then
$f \in \stab^P_{K'}(A_h, w_n A_h)$.  By linear progress with
exponential decay (Proposition \ref{L:progress-decay}), the probability
that $d_{P_L}(A_h, w_n A_h) \ge R_{ac}(K')$ tends to one exponentially
quickly.  Therefore the probability that
$\# \norm{ \stab^P_{K'}(A_h, w_n A_h) } \le N_{ac}(K')$ tends to one
with exponentially decay, and so the probability that
$\# \norm{ \stab^X_{K}(x, w_n x) } \le N_{ac}(K')$ also tends to one
exponentially quickly, as required.
\end{proof}

\section{Non-matching estimates}

So far, we have established generic properties of our random walks by
proving \emph{matching estimates}, i.e. by showing that with high
probability there is a subsegment of the sample path that fellow
travels some given element.  However, in order to establish our
results on the normal closure, we need to prove that the probability
of such a matching to occur too often is not so high: we call this a
\emph{non-matching estimate}. Note that, while matching happens for
random walks on any group of isometries of a hyperbolic space, to
prove non-matching one uses crucially the WPD property (and in fact,
non-matching may not hold in the non-WPD case, for example, for a
dense subgroup of $SL(2, \R)$ acting on $\mathbb{H}^2$).

We now define notation for the nearest point projection of a location
$w_m x$ of the random walk to a geodesic $\gamma_n$ from $x$ to
$w_n x$.

\begin{definition}
Given integers $0 \le m \le n$, let $\gamma_n$ be a geodesic from $x$
to $w_n x$, and let $\gamma_n(t_m)$ be a nearest point on $\gamma_n$
to $w_m x$.
\end{definition}

The main non-matching estimate is the following proposition, which
says that the probability that $\gamma_n$ contains in its
neighbourhood a translate of a given geodesic segment $\eta$ starting
at $\gamma_n(t_m)$ is bounded above by an exponential function of
$\norm{ \eta }$.  We will prove it by using the asymptotic
acylindricality property established in the previous section.

\begin{proposition} \label{P:non-match} %
Given a constant $\delta \ge 0$ there is a constant $K_0 \ge 0$ with
the following properties.  Let $G$ be a group which acts by isometries
on the  $\delta$-hyperbolic space $X$, and let $\mu$ be a
countable, bounded probability distribution on $G$, such that 
the random walk generated by $\mu$ is asymptotically acylindrical 
with exponential decay. 

Then for any constant $K \ge K_0$ there are constants $B > 0$ and
$c < 1$, such that for any geodesic segment $\eta$ and any integers $m \geq 0$, $n \geq 0$, 
the probability that a $G$-translate of $\eta$ is contained in a
$K$-neighbourhood of $[ \gamma_n(t_m), \gamma_n(t_m + \norm{\eta}) ]$
is at most $B c^{ \norm{\eta} }$.
\end{proposition}

Before embarking on the details, we give a brief overview of the
contents of this section.  Fix a geodesic segment $\eta$ of length
$2s$.  We wish to estimate the probability that some translate of
$\eta$ is contained in a neighbourhood of
$[ \gamma_n(t_m), \gamma_n(t_m + 2s) ]$.  Let $U \subset (G, \mu)^\Z$
be the event that some translate of $\eta$ is contained in a
neighbourhood of $[ \gamma_n(t_m), \gamma_n(t_m + 2s) ]$, and let $V$
be the event that some translate of the first half of $\eta$ is
contained in a neighbourhood of
$[ \gamma_n(t_m), \gamma_n(t_m + s) ]$.  Since $U \subseteq V$, 
the conditional probability of $U$ given $V$ satisfies $\P(U) = \mathbb{P}(U \cap V) \le \P( U \mid V)$.  Let
$U_g$ be the event that a specific translate $g \eta$ is contained in
a neighbourhood of $[ \gamma_n(t_m), \gamma_n(t_m + 2s) ]$, and let
$V_g$ be the event that the first half of $g \eta$ is contained in a
neighbourhood of $[ \gamma_n(t_m), \gamma_n(t_m + s) ]$.  The event
$U$ is the union of the events $U_g$, and the event $V$ is the union
of the events $V_g$.  It follows from exponential decay of shadows
that $\P(U_g \mid V_g)$ decays exponentially in $s$.  In order to use
this fact to estimate $\P(U \mid V)$ we need the following extra
information: it follows from asymptotic acylindricality that with high
probability any point of $V$ is contained in a bounded number of sets
$V_g$, and this is enough for the exponential decay in $s$ of
$\P(U_g \mid V_g)$ to imply exponential decay in $s$ of
$\P(U \mid V)$.

We now give the details of the results discussed above.  We will need
information about the distribution of the nearest point projections of
the locations $w_m x_0$ of the random walk to the geodesic $\gamma_n$,
and we start with the following estimate on Gromov products, which
follows directly from exponential decay of shadows.

\begin{proposition}\label{P:gp}
Let $G$ be a group acting by isometries on a  Gromov
hyperbolic space $X$, and let $\mu$ be a countable, non-elementary,
bounded probability distribution on $G$.  Then there are constants $B$
and $c < 1$ such that for all $0 \le i \le n$ and for any $r \geq 0$,
\[ \P ( \gp{w_i x}{x}{w_n x} \ge r ) \le B c^r.  \]
\end{proposition}

\begin{proof}
If $\gp{w_i x}{x}{w_n x} \ge r$, then $x$ lies in a shadow
$S_{w_i}(w_n x, R)$, with $d(w_i x, w_n x) - R \ge r + O(\delta)$.
The random variables $w_i$ and $w_i^{-1} w_n$ are independent, so by
exponential decay of shadows \cite[eq. (16)]{MT}, this occurs with probability at most
$B c^{r + O(\delta)}$.
\end{proof}

Linear progress for the locations of the sample path $w_m x_0$ in $X$,
and exponential decay for the distribution of the Gromov products
$\gp{w_m x_0}{x_0}{w_n x_0}$ imply that the points $\gamma_n(t_m)$ are
reasonably evenly distributed along $\gamma_n = [x_0, w_n x_0]$.  We
now make this precise.  As $\mu$ has bounded support in $X$, there is
a constant $D$ such that any point in $\gamma_n$ lies within distance
at most $D$ from a nearest point projection $\gamma_n(t_i)$ of one of
the locations of the random walk $w_i x$, for $0 \le i \le n$, and
furthermore, we may choose $D$ to be an upper bound for the diameter
of the support of $\mu$ in $X$.  For any constant $s \ge 0$, let
$P_{s}$ be the collection of indices $0 \le i \le n$ such that
$t_i \in [s, s + D]$.  This collection is non-empty if
$s \le \norm{\gamma_n}$.  We emphasize that $P_s$ only contains
indices between $0$ and $n$, there may be other locations of the
bi-infinite random walk which have nearest point projections to
$\gamma_n$ contained in $[\gamma(s), \gamma(s+D)]$, and we consider
this separately in Proposition \ref{P:tail indices} below.

\begin{proposition}\label{P:npp} 
Let $G$ be a group which acts by isometries on the 
hyperbolic space $X$, and let $\mu$ be a countable, non-elementary,
bounded probability distribution on $G$.  Then there are constants
$0 < L_1 \le L_2$, $B \ge 0$ and $c < 1$ such that for any $s > 0$ and
any $n \geq 0$,
\[ \P( P_s \subseteq [ L_1 s, L_2 s ] ) \ge 1 - B c^s.  \]
\end{proposition}

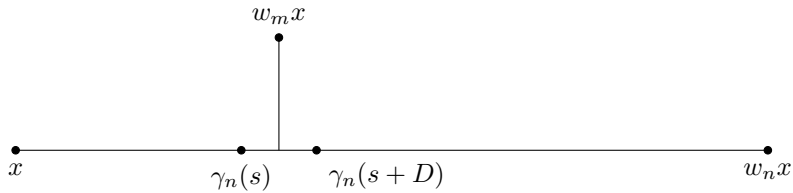
\begin{figure}[H]
\begin{center}
\begin{tikzpicture}

\tikzstyle{point}=[circle, draw, fill=black, inner sep=1pt]

\draw (0, 0) node [point, label=below:{$x$}] {} -- (10, 0) node
[point, label=below:{$w_n x$}] {};

\draw (3.5, 1.5) node [point, label=above:{$w_m x$}] {} -- (3.5, 0);
\draw (3, 0) node [point, label=below:{$\gamma_n(s)$}] {};
\draw (4, 0) node [point, label=below right:{$\gamma_n(s + D)$}] {};

\end{tikzpicture}
\end{center}

\caption{The set $P_s$ defined right before Proposition \ref{P:npp}. The index $m$ belongs to $P_s$ as its projection to $[x, w_n x]$ lies 
within distance $s$ and $s+D$ from the basepoint.}\label{fig:npp}
\end{figure}

\begin{proof}
If $s > d(x, w_n x)$, then $P_s = \varnothing$, and the statement
follows immediately, so we may assume that $\gamma_n(s)$ determines a
point in $\gamma_n$.

By linear progress with exponential decay (Proposition
\ref{L:progress-decay}), there are constants $L > 0, B_1 \ge 0$ and
$c_1 < 1$ such that for any $m \geq 0$
\[ \P( d( x, w_m x) \le L m ) \le B_1 c_1^m. \]
Therefore, by summing the geometric series we get 
\[ \P( d( x, w_m x ) \le L m \text{ for any } m \ge N ) \le
\frac{B_1}{1-c_1} c_1^N.  \]
In particular, there are constants $B_2$ and $c_2 < 1$ such that 
\begin{equation}\label{eq:distance}
\P( d( x, w_m x ) \ge L m \text{ for all } m \ge 2s/L ) \ge 1 -
B_2 c_2^s.
\end{equation}
If \eqref{eq:distance} holds, and if $m \ge 2s/L$, then
$d( x, w_m x) \ge Lm \ge 2s$, so by thin triangles 
and the definition of the Gromov product, if the nearest point projection
$\gamma(t_m)$ of $w_m x$ lies in $[ \gamma_n(s), \gamma_n(s + D)]$,
then
\begin{equation}\label{eq:npp}
\gp{w_m x}{x}{w_n x} \ge d(x, w_m x) - s - D - O(\delta).
\end{equation}

By exponential decay for Gromov products (Proposition \ref{P:gp}),
there are constants $B_3$ and $c_3$ such that
$\P( \gp{w_m x}{x}{w_n x} \ge r ) \le B_3 c_3^r$. In particular,
\[ \P ( \gp{w_m x}{x}{w_n x} \ge Lm - s - D - O(\delta) ) \le B_3 c_3^{Lm
  - s - D - O(\delta)}.  \]
This implies that there are constants $B_4$ and $c_4 < 1$ such that for any $n$
\begin{equation}\label{eq:gp}
\P ( \gp{w_m x}{x}{w_n x} \ge Lm - s - D - O(\delta) \text{ for any } m
\ge 2s/L) \le B_4 c_4^s.
\end{equation}

Except for a set of probability at most $B_2 c_2^s + B_4 c_4^s $, we
may assume that \eqref{eq:distance} holds, and \eqref{eq:gp} does not
hold.  Equation \eqref{eq:npp} then implies that $\gamma(t_m)$ does
not lie in $[\gamma_n(s) , \gamma_n(s + D)]$ for all $m \ge 2s/L$.
This gives the required upper bound, with $L_2 = 2/L$, and suitable
choices of $B$ and $c$.  As $\mu$ has bounded support in $X$, the
lower bound may be chosen to be $L_1 = 1/D$.
\end{proof}

We now obtain estimates for the nearest point projections of the
remaining locations of the random walk $w_m x$ to a geodesic $\gamma_n
= [x, w_n x]$, i.e. for those indices $m \le 0$ and $m \ge n$.

\begin{proposition} \label{P:tail indices} %
Let $G$ be a group which acts by isometries on the 
hyperbolic space $X$, and let $\mu$ be a countable, non-elementary,
bounded probability distribution on $G$.  Then there are constants $B$
and $c$ such that for all $s \ge 0$ the probability that all of the
nearest point projections of $\{ w_m x \ : \ m \le 0 \}$ to
$\gamma_n = [x, w_n x]$ are contained within distance $s$ of the
initial point $x$, and all of the nearest point projections of
$\{ w_m x \ : \ m \ge n \}$ to $\gamma_n$ are contained within
distance $s$ of the terminal point $w_n x$, is at least $1 - Bc^s$.
\end{proposition}

\begin{proof}
By the Markov property, the backward random walk
$(w_{-n} x)_{n \in \N}$ is independent of $\gamma_n$.  Similarly, the
forward random walk starting at $w_n x$ is also independent of
$\gamma_n$.  More precisely, applying the isometry $w_n^{-1}$, the
random walk $w_n^{-1}(w_m x)_{m \ge n}$ starting at $x$, is
independent of $w_n^{-1} \gamma_n$.  Therefore, it suffices to show
that for any geodesic ray $\gamma$ starting at $x$, a random walk has
nearest point projection to an initial segment of $\gamma$ with high
probability.

Let $\gamma$ be a geodesic ray starting at $x$, with unit speed
parameterization, and consider the forward locations of the random
walk $(w_n x)_{n \in \N}$. Let $\gamma(t_n)$ be the nearest point
projection of a location $w_n x$ to $\gamma$.  If $t_n \ge s$, then
$w_n x$ lies in the shadow $S_x(\gamma(s), R)$, for some $R$ which
only depends on $\delta$.  By \eqref{eq:exp hit} the probability that
$(w_n)_{n \in \Z}$ ever hits $S_x(\gamma(s), R)$ is at most $B c^s$.
Therefore the probability that this does not occur for any index $n$
is at least $1 - B c^s$.
\end{proof}

We now consider the following situation: we have chosen an index
$0 \le m \le n$, and a constant $s \ge 0$.  We wish to estimate the
probability that there is a translate of a geodesic $\eta$ of length
$2s$ close to $\gamma_n$ starting at $\gamma_n(t_m)$.  In order to do
this, it will be convenient to have information about the distribution
of the nearest point projections of $w_k x_0$ to $\gamma_n$, and in
particular, the sets $P_{t_m + s}$ and $P_{t_m + 2s}$.  Proposition
\ref{P:m+s} below assembles the geometric information we need from all
of the results above, and in particular shows that with high
probability, there are linear bounds on the sizes of $P_{t_m + s}$ and
$P_{t_m + 2s}$, and that these sets are disjoint.

\begin{proposition}\label{P:m+s}
Let $G$ be a group which acts by isometries on the 
hyperbolic space $X$, and let $\mu$ be a countable, non-elementary,
bounded probability distribution on $G$.  Then there are constants
$0 < L_1 \le L_2$, such that for any $0 < \e < 1$, there are constants
$B \ge 0$ and $c < 1$ such that for any $0 \le m \le n$ and $s >
0$,
the probability that all of the following events occur is at least
$1 - B c^s$:
\begin{align}
&  \gp{w_m x}{x}{w_n x} \le \e s \tag{\ref{P:m+s}.1} \label{eq:m+s.1} \\
&   L_1 s \le \min P_{t_m + s} \le \max P_{t_m + s} \le L_2 s
  \tag{\ref{P:m+s}.2} \label{eq:m+s.2} \\ 
&   2 L_1  s \le \min P_{t_m + 2s} \le \max P_{t_m + 2s} \le 2 L_2 s
  \tag{\ref{P:m+s}.3} \label{eq:m+s.3} \\ 
&   \gp{w_i x}{x}{w_n x} \le \e s \textup{ for all } i \in P_{t_m + s} \cup
  P_{t_m + 2s} \tag{\ref{P:m+s}.4} \label{eq:m+s.4} \\
&   \max P_{t_m + s} \le \min P_{t_m + 2s}
  \tag{\ref{P:m+s}.5} \label{eq:m+s.5}
\end{align}
\end{proposition}

The proposition is illustrated in Figure \ref{fig:m+s} below, where
the index $m + a$ belongs to $P_{t_m + s}$, and $m + b$ belongs to
$P_{t_m + 2s}$.

\begin{figure}[H]
\begin{center}
\begin{tikzpicture}

\tikzstyle{point}=[circle, draw, fill=black, inner sep=1pt]

\draw (0, 0) node [point, label=below:{$x$}] {} -- (10, 0) node
[point, label=below:{$w_n x$}] {};

\draw (2, 1.5) node [point, label=above:{$w_m x$}] {} --
      (2, 0) node [point, label=below:{$\gamma_n(t_m)$}] {};
\draw (4.5, 1.5) node [point, label=above:{$w_{m+a} x$}] {} --
      (4.5, 0) node [point] {};
\draw (6.5, 1.5) node [point, label=above:{$w_{m+b} x$}] {} --
      (6.5, 0) node [point] {};

\draw (2, 1.5) .. controls (2, 0) and (9, 0) ..
      (10, 0) node [midway, above] {$\gamma$};
      
\draw (4, 0) node [point, label=below:{$\gamma_n(t_m + s)$}] {};
\draw (6, 0) node [point, label=below:{$\gamma_n(t_m + 2s)$}] {};

\end{tikzpicture}
\end{center}
\caption{Nearest point projections relative to $\gamma_n(t_m)$.}\label{fig:m+s}
\end{figure}
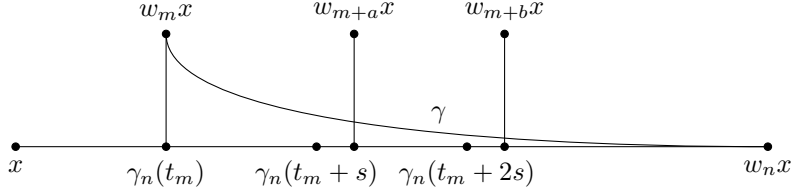

\begin{proof}
We say that a function $\cE(s) \colon \R \to \R$ is \emph{exponential
  in $s$} if there are constants $B \ge 0$ and $c < 1$ such that
$\cE(s) \le B c^s$ for all $s \ge 0$.  We observe that the sum of any
two functions which are exponential in $s$ is exponential in $s$, and
if $p(s)$ is a polynomial in $s$, and $\cE(s)$ is exponential in $s$,
then $p(s) \cE(s)$ is also exponential in $s$.

By exponential decay for Gromov products (Proposition \ref{P:gp}),
eq. \eqref{eq:m+s.1} holds with probability at least $1 - \cE_1(s)$,
where $\cE_1(s) = B c^s$.

Let $\gamma$ be a geodesic from $w_m x$ to $w_n x$, with unit speed
parameterization, and write $\gamma(t_k)$ for a nearest point
projection of $w_k x$ to $\gamma$.  By the Markov property, we may
apply Proposition \ref{P:tail indices} to $\gamma$, and so there are
constants $B \ge 0$ and $c < 1$ such that the probability that
\begin{equation}\label{eq:tail indices}
\{ \gamma(t_k) \ : \ k \in \Z, k \le m \} \subset [w_m x, \gamma(s/2)]
\end{equation}
holds with probability at least $1 - \cE_2(s)$, where
$\cE_2(s) = B c^s$.

By thin triangles and assuming that
$\gp{w_m x}{x}{w_n} \leq \epsilon s$, if the nearest point projection
to $\gamma_n$ of a location $w_{m+a} x$ lies in
$[\gamma_n(t_m + s), \gamma_n(t_m + s) + D]$, then the nearest point
projection of $w_{m+a} x$ to $\gamma$ lies in
$[\gamma(s), \gamma(s + \e s + D + \delta)] $.  Proposition
\ref{P:npp} applied to each of the $(\e s + \delta) / D$ subsegments
of $[s, s + \e s + D + \delta]$ of length $D$ implies that
$L_1 s \le a \le L_2( s + \e s + D + \delta )$ with probability at
least $1 - \cE_3(s)$, where
$\cE_3(s) = ( (\e s + \delta) / D ) B c^s$.  Therefore
\eqref{eq:m+s.2} holds (with a slightly larger value of $L_2$).
Furthermore, by \eqref{eq:tail indices} there are no locations $w_k x$
with $k \le m$ or $k \ge n$ which have nearest point projections in
$[\gamma(s), \gamma(s + \e s + D + \delta)] $.

The exact same argument works for \eqref{eq:m+s.3}, as long as
$t_m + 5s/2 \le \norm{\gamma}$.

Exponential decay for Gromov products then implies \eqref{eq:m+s.4}
with probability at least $1 - \cE_4(s) $, where
$\cE_4(s) = 3 (L_2 - L_1) s B c^{\e s}$. The constant
$3 (L_2 - L_1) s$ here derives from the cardinality of
$P_{t_m + s} \cup P_{t_m + 2s}$ when \eqref{eq:m+s.2} and
\eqref{eq:m+s.3} hold.

Finally, if there is some $b < a$, then
$\gp{w_{m+b} x}{x}{ w_{m+a} x} \ge s - D + O(\delta)$, and so the
probability that this does not occur for any $a$ and $b$
(i.e. \eqref{eq:m+s.5} holds) is at least $1 - \cE_5(s)$, where
$\cE_5(s) = 3 (L_2 - L_1) s B c^{s - D + O(\delta)}$.

Therefore all equations \eqref{eq:m+s.1}--\eqref{eq:m+s.5} hold with
probability at least $1 - \cE(s)$, where $\cE(s)$ is the sum of the
functions $\cE_1(s)$--$\cE_5(s)$ above.  All of these functions are
exponential in $s$, so $\cE(s)$ is also exponential in $s$, as
required.
\end{proof}

We now show that for any fixed translate $g \eta$ of a geodesic $\eta$
of length $2s$, if the first half of $\eta$ is contained in a
neighbourhood of $[\gamma_n(t_m), \gamma_n(t_m + s)]$, then the
probability that $\eta$ is contained in a neighbourhood of
$[\gamma_n(t_m), \gamma_n(t_m + 2s)]$ decays exponentially in $s$.

\begin{proposition}\label{P:g match} 
Let $G$ be a group which acts by isometries on the 
hyperbolic space $X$, and let $\mu$ be a countable, non-elementary,
bounded probability distribution on $G$.  Then there are constants
$B \ge 0$ and $c < 1$ such that for any geodesic segment $\eta$ of
length $2s$ with initial half-segment $\eta_1$ of length $s$, if there
is an isometry $g \in G$ such that $g \eta_1$ is contained in a
$K$-neighbourhood of $[\gamma_n(t_m), \gamma_n(t_m + s)]$, then the
probability
that
$g \eta$ is contained in a $K$-neighbourhood of
$[\gamma_n(t_m), \gamma_n(t_m + 2 s)]$ is at most $B c^s$.
\end{proposition}

\begin{proof}
By Proposition \ref{P:m+s}, there are constants $B_1$ and $c_1 < 1$
such that \eqref{eq:m+s.1}--\eqref{eq:m+s.5} hold, with probability at
least $1 - B_1 c_1^s$.

If $g \eta_1$ is contained in a $K$-neighbourhood of
$[\gamma_n(t_m), \gamma_n(t_m + s)]$, then in order for $\eta$ to be
contained in a $K$-neighbourhood of
$[\gamma_n(t_m), \gamma_n(t_m + 2 s)]$, for any index $m + b \in P_{t_m + 2s}$ 
the point $w_{m + b} x$ must lie in a shadow $S_{w_{m+a} x}(g \eta(2s), R)$,
where $R$ depends only on $K$ and $\delta$.  As $w_{m+a}$ and
$w_{m+a}^{-1} w_{m+b}$ are independent, and there are at most
$2(L_2 - L_1)s$ elements of $P_{t_m + 2s}$, this happens with
probability at most $2(L_2 - L_1)s B_2 c_2^s$, by exponential decay
for shadows.  The result then follows for suitable choices of $B$ and
$c$.
\end{proof}

Proposition \ref{P:g match} above only holds for a \emph{fixed}
translate $g \eta$.  We will use asymptotic acylindricality to extend
this result to hold for \emph{some} translate $g \eta$, where $g$ runs
over all elements of $G$.  We start with a result from Calegari and
Maher \cite{calegari-maher}, which says that every point in $\gamma_n$
is close to some location $w_k x_0$.  We say that a point
$\gamma(t) \in \gamma_n$ is \emph{$K$-close} if
$d(\gamma(t), w_i x) \le K$ for some $0 \le i \le n$.  We shall denote
the set of $K$-close points by $\gamma_{n , K}$.

\begin{lemma}\cite{calegari-maher}*{Lemma 5.13}\label{L:proximal}
Given $\delta \ge 0$ and positive constants $D, L$ and $\e$, there is
a constant $K \ge 0$ such that for any sequence of points
$x_0, x_1, \ldots x_n$ in a $\delta$-hyperbolic space $X$, with
$d(x_i, x_{i+1}) \le D$, and $d(x_0, x_n) \ge L n$, and for any
geodesic $\gamma_n$ from $x_0$ to $x_n$, the total length of
$\gamma_{n, K}$ is at least
\[ \norm{\gamma_{n, K}} \ge (1 - \e) \norm{\gamma_n}.  \]
\end{lemma}

Let $U$ be the event that some translate of $\eta$ is contained in a
neighbourhood of $[\gamma_n(t_m), \gamma_n(t_m + 2s) ]$, and let $V$
be the event that the first half of some translate of $\eta$ is
contained in a neighbourhood of $[\gamma_n(t_m), \gamma_n(t_m + s) ]$.
We wish to estimate $\P(U)$.  However, as $U \subset V$, the formula
for conditional probability implies that $\P(U) \le \P( U \mid V)$, so
it suffices to estimate $\P(U \mid V)$.

Let $U_g$ be the event that the translate $g \eta$ is contained in a
neighbourhood of $[\gamma_n(t_m), \gamma_n(t_m + 2s)]$, and let $V_g$
be the event that the first half of the translate $g \eta$ is
contained in a neighbourhood of $[\gamma_n(t_m), \gamma_n(t_m + s) ]$.
The set $U$ is equal to the union of the $U_g$, and similarly $V$ is
equal to the union of the $V_g$.  For each $g$, we have
$\P(U_g \mid V_g) \le B c^s$, by Proposition \ref{P:g match}.  We wish
to use this information to estimate $\P(U \mid V)$.  The key property
is that asymptotic acylindricality implies that with high probability
each point of $V$ is contained in a bounded number of sets $V_g$, and
so exponential decay for the individual conditional probabilities
$\P(U_g \mid V_g)$ gives exponential decay for $\P(U \mid V)$.  We now
give the details of this argument.

Let $V$ and $\{ V_i \}_{i \in I}$ be a collection of subsets of a
probability space.  We say that the collection of sets
$\{ V_i \}_{i \in I}$ \emph{covers} the set $V$ if
$V \subset \bigcup_{i \in I} V_i$.  We say that the \emph{covering
  depth} of the $\{ V_i \}_{i \in I}$ is
$\sup_{v \in V} \setnorm{ \{ i \in I \ : \ v \in V_i \}}$.  If the
covering depth of $\{ V_i \}_{i \in I}$ is $N$, and all sets are
measurable, then $\P(V) \le \sum_{i \in I} \P(V_i) \le N \P(V)$.

We will also make use of the following definition:

\begin{definition}
We say that a pair of points $x$ and $y$ are \emph{$(K, N)$-stable} if
$$\setnorm{\stab_K(x) \cap \stab_K(y)} \le N.$$  
We say that a geodesic segment $\eta$ is \emph{$(K, N)$-stable} if its endpoints are $(K, N)$-stable.
\end{definition}

\begin{proof}[Proof (of Proposition \ref{P:non-match})]
Let $s := \norm{\eta}/2$. 
We wish to estimate the probability that a translate of $\eta$ is contained in a $K$-neighbourhood of
$[\gamma(t_m), \gamma(t_m + 2s)]$.  Let $\eta_1$ be the
initial subsegment of $\eta$ with length
$\norm{\eta_1} = \norm{\eta}/2 = s$.  
By Proposition \ref{P:m+s}, we may assume that
\eqref{eq:m+s.1}--\eqref{eq:m+s.5} hold, with probability at least
$1 - B c^s$.

Let us suppose now that a translate $g \eta$ is contained in a $K$-neighbourhood of
$[\gamma(t_m), \gamma(t_m + 2s)]$. 
By thin triangles, the geodesic $g \eta_1$ is contained in a
$(K + 2 \delta)$-neighbourhood of the geodesic $[w_m x, w_{m+a} x]$.
By Lemma \ref{L:proximal}, choosing $\e = 1/8$, there is a constant
$K_1$ such that there are indices $i$ and $j$, with $w_{i} x$
within distance $K_2 = K_1 + K + 2 \delta$ of
$[g \eta_1(0),  g \eta_1(s/4)]$ and $w_{j} x$ within distance $K_2$ of
$[g \eta_1(3 s / 4), g \eta_1(s)]$.  In particular
$d(w_{i} x, w_{j} x) \ge s / 2 - 2 K_2$, and so
\begin{equation} \label{eq:ij}
\norm{i - j} \ge ( s / 2 - 2 K_2 ) / D.
\end{equation}
Set $K_3 = \max \{ K_2, 5 K\}$ and $K_4 = K_3 + 2K + 2 \delta$.

Let $U \subseteq (G, \mu)^\N$ be
the set of sample paths for which a translate of $\eta$ is contained
in a $K$-neighbourhood of $[\gamma(t_m), \gamma(t_m + 2s)]$, and let
$U_g$ be the set of sample paths for which $g \eta$ is contained in a
$K$-neighbourhood of $[\gamma(t_m), \gamma(t_m + 2 s)]$.  Let
$V \subseteq (G, \mu)^\N$ be the set of sample paths for which a
translate of $\eta_1$ is contained in a $K$-neighbourhood of
$[\gamma(t_m), \gamma(t_m + s)]$, and let $V_g$ be the set of sample
paths for which $g \eta_1$ is contained in a $K$-neighbourhood of
$[\gamma(t_m), \gamma(t_m + s)]$.  As $U \subseteq V$, the conditional
probability $\P(U \vert V)$ satisfies $\P(U) \leq \P(U \vert V)$. 

Proposition \ref{P:g match} shows that for any $g$ the conditional probability $\P( U_g \vert V_g)$ decays
exponentially in $n$.  The sets $\{ U_g \}_{g \in G} $ cover $U$, in
fact $U = \bigcup_{g \in G} U_g$, and similarly
$V = \bigcup_{g \in G} V_g$.  The covering depth of $\{ V_g \}$ is an upper
bound on the covering depth of $\{ U_g \}$.  We now show that with high
probability the covering depth of $\{ V_g \}$ is bounded, i.e. there exists a set $S$ 
of large measure such that the covering depth of $\{ V_g \cap S \}$ is bounded.

We now have two cases. If $\eta_1$ is not $(K_3, N_{ac}(K_4))$-stable, then $w_{i} x$ and
$w_{j} x$ are not $(K_4, N_{ac}(K_4))$-stable, where $N_{ac}(K)$ is
the function from asymptotic acylindricality.  Then by Theorem \ref{T:asymp-acyl} the probability that, given 
$i$ and $j$, the points $w_i x$ and $w_j x$ are not $(K_4, N_{ac}(K_4))$-stable is at most 
$B c^{|j-i|} \leq B_3 c_3^{s/ 2D}$
for some constants $B_3$ and
$c_3 < 1$, where we used eq. \eqref{eq:ij}. 
Recall that by construction $m \leq i \leq j \leq m +a$, and by \eqref{eq:m+s.2} we have $a \leq L_2 s$, hence there are at most $(L_2 s)^2$ such choices of $i, j$. Hence, 
the probability that there are such indices $i$ and $j$ is at most
$ 2 (L_2 s)^2 B_3 c_3^{s/ 2D}$. 

If $\eta_1$ is $(K_4, N_{ac}(K_4))$-stable, then by definition the covering depth of
$V_g$ is at most $N_{ac}(K_4)$.  By Proposition \ref{P:g match}, there
are constants $B_4$ and $c_4 < 1$ such that
$\P(U_g \vert V_g) \le B_4 c_4^s$.  As $U_g \subseteq V_g$, this implies
$\P(U_g) \le B_4 c_4^s \P(V_g)$.  Therefore
\[ \P(U) \le \sum_{g \in G} \P(U_g) \le B_4 c_4^s \sum_{g \in G}
\P(V_g) \le N_{ac}(K_4) B_4 c_4^s \P(V) \le N_{ac}(K_4) B_4 c_4^s. \]
Therefore, the probability that a translate of $\eta$ is contained in
a $K$-neighbourhood of $[\gamma_n(t_m), \gamma_n(t_m + s) ]$ is at
most
$B c^s + 2 (L_2s)^2 B_3 c_3^{s/2D} + N_{ac}(K_4) B_4 c_4^s$,
which has exponential decay in $s$, as required.
\end{proof}

We are now interested in the particular case of matching between two subsegments of a given geodesic segment. 
We call this phenomenon a \emph{self-match}. Here is the precise definition.

\begin{definition} \label{D:self-match}
We say that a geodesic segment $\gamma$ has an $(L, K)$\emph{-self match} if there exist two disjoint subsegment $\eta, \eta' \subseteq \gamma$ 
of length $L$ and an element $g \in G \setminus \{1\}$ such that the Hausdorff distance between $g \eta$ and $\eta'$ is at most $K$.
\end{definition}
 
\begin{proposition} \label{P:self-match} %
Let $G$ be a group acting by isometries on a Gromov hyperbolic space
$X$, and let $\mu$ be a countable, non-elementary, bounded probability
distribution on $G$, such that the
random walk generated by $\mu$ is asymptotically acylindrical
with  exponential decay. 
Then there is a constant $K_0$, depending only on
$\delta$, such that for any $K \ge K_0$, there exists $B > 0$ such that for any $L \ge 0$ and any $n \geq 0$
the probability that $\gamma_n$ has an $(L , K )$-self match is at most
$n^3 B c^{L}$.
\end{proposition}

\begin{proof}
Suppose that $\gamma_n$ has an $(L, K)$-self-match.  Then there is a
subgeodesic $\eta = [\gamma_n(t), \gamma_n(t + L)]$ such that a
translate $g \eta$ is contained in a $K$-neighbourhood of $\gamma_n$,
and the nearest point projection of $g \eta$ to $\gamma_n$ is
disjoint from $\eta$.  Without loss of generality, we may assume that
the translate of $\eta$ is contained in a $K$-neighbourhood of
$[\gamma_n(t+L), \gamma_n(\norm{\gamma_n})]$.

There is a constant $D$ such that the nearest point projection of the
sample path $\{ w_m x \colon 0 \le m \le n \}$ to $\gamma_n$ is
$D$-coarsely onto, and the diameter of the support of $\mu$ in $X$ is
at most $D$.  Let $w_m x$ be a location of the random walk such that
the nearest point projection $\gamma_n(t_m)$ lies within distance $D$
of the interval of $\gamma_n$ between $\eta$ and the nearest point
projection of $g \eta$.

Then $\eta$ is contained in a $(K + D + \delta)$-neighbourhood of
$[x, w_m x]$, and $g \eta$ is contained in a
$(K + D + \delta)$-neighbourhood of $[w_m x , w_n x]$.  We do not need
to consider all possible subsegments of $[x, w_m x]$, as it suffices
to consider those whose endpoints are integer distances from $x$.
More precisely, there is a subsegment
$\eta_- = [\gamma_n(a), \gamma_n(b)]$ of $\eta$, for integers
$a \le b$, with $\norm{\eta_-} \ge \norm{\eta} - 2$.  If we set
$K_1 := K + D + \delta + 1$, then the geodesic $\eta_-$ $K_1$-matches
$\gamma' = [w_m x, w_n x]$ at distance $\gamma'(c)$ from $w_m x$,
where $c$ is also an integer.

There are at most $n$ choices for $m$, at most
$d(x, w_m x) \le Dm \le Dn$ choices for $a$, and at most
$d(w_m x, w_n x) \le D (n-m) \le Dn$ choices for $c$, so in total at
most $D^2 n^3$ choices for the triple $(m, a, c)$.  Given a triple of
choices $m, a$ and $c$, and the constant $K_1$, Proposition
\ref{P:non-match} implies that there are constants $B_1$ and $c_1$
such that the probability that a translate of $\eta_-$ is contained in
a $K_1$-neighbourhood of $[w_m x, w_n x]$ is at most
$B_1 c_1^{L - 2D}$.  Therefore the probability that $\gamma_n$
has an $(L, K)$-self-match is at most
$D^2 n^3 B_1 c_1^{L - 2D}$, and the result follows by suitable
choices of $B$ and $c$ (since $D$ is a constant).
\end{proof}

We will use the following result due to Dahmani and Horbez
\cite[Proposition 2.5]{dahmani-horbez}: they do not explicitly state
the rate, but it follows immediately from the proof.

\begin{proposition} \label{P:match-axis} %
Given $\delta$ and $K_1$ there is a constant $K$ with the following
properties.  Let $G$ be a group acting on a $\delta$-hyperbolic space
$X$, and let $\mu$ be a countable, non-elementary, bounded probability
distribution on $G$.  Let $\ell > 0$ be the drift of the random walk
generated by $\mu$.  If $w_n$ is loxodromic, let $p$ denote a nearest
point projection of $x$ to a
quasiaxis for $w_n$. Then
there exist constants $B > 0$, $c < 1$ such that for any
$\epsilon > 0$ we have
$$\mathbb{P}\left( \gamma_n \textup{ has a }((\ell-\epsilon)n, K)\textup{-match with }[p, w_n p] \right) \geq 1 - B c^{\epsilon n}. $$
\end{proposition}

Finally, we record the following result, which is an immediate
consequence of Propositions \ref{P:self-match} and Proposition
\ref{P:match-axis} above.

\begin{corollary}\label{C:self-match}
For any $\delta \ge 0$, there is a constant $K_0$ with the following
properties.  Let $G$ be a group acting by isometries on a
$\delta$-hyperbolic space $X$, and let $\mu$ be a countable,
non-elementary, bounded probability distribution on $G$, such that
the random walk generated by $\mu$ is asymptotically acylindrical with 
exponential decay. 
Let $\ell > 0$ be the drift for $\mu$, and let $p$ be a point on a quasiaxis for $w_n$.  

Then for any
$K \ge K_0$ and $\e > 0$, there are constants $B$ and $c < 1$ such
that for any $n \ge 0$ the probability that either
$\gamma_n = [x, w_n x]$ or $[p, w_n p]$ has an
$(\e \ell n , K )$-self match is at most $B c^{n}$.
\end{corollary}

\section{Asymmetric elements}

We now use the non-matching results to show that a generic element is
asymmetric in the following sense. This definition is a variation of
the one used in \cite{maher-sisto}, where similar results are obtained in
the case that the action is acylindrical.

\begin{definition} \label{D:asym}
We say that a loxodromic isometry $g \in G$ is
\emph{$(\e, L, K)$-asymmetric} if for any subsegment
$[p, q] \subset \alpha_g$ of length at least $\e d(p, gp)$, and any
group element $h$, if $h [p, q]$ is contained in an $L$-neighbourhood
of $\alpha_g$, then there is an $i \in \Z$ such that
$d(hp, g^i p) \le K$ and $d(h q, g^{i} q) \le K$.
\end{definition}

\begin{proposition}\label{P:asymmetric}
Given a constant $\delta \ge 0$, for any constants $\e > 0$ and
$L \ge 0$, there is a constant $K$ such that if $G$ is a group acting
on a $\delta$-hyperbolic space $X$, and $\mu$ is a countable,
non-elementary, bounded probability distribution on $G$, such that the random walk
generated by $\mu$ is asymptotically acylindrical with exponential decay, 
then there are constants $B$ and $c < 1$ such that
the probability that $w_n$ is $(\e, L, K)$-asymmetric is at least
$1 - B c^{n}$.
\end{proposition}

We first recall the following useful fact about isometries of Gromov
hyperbolic spaces.

\begin{proposition} \label{P:concatenate} %
Given $\delta \ge 0$ there is a constant $K_0$ such that for any
$K \ge K_0$, if $X$ is a $\delta$-hyperbolic space, and $g$ is an
isometry for which there is a point $x \in X$ such that
$d(x, gx) \ge 3 K$ and $\gp{g x}{x}{g^2 x} \le K$, then $g$ is
loxodromic, and any quasiaxis $\alpha_g$ of $g$ passes within distance
$2K$ of $g x$.
\end{proposition}

\begin{proof}
This follows from the following estimate for the translation length of
an isometry:
\[ \tau(g) \ge d(x, g x) - 2 \gp{g x}{x}{g^2 x} - O(\delta), \]
see for example \cite{MT}*{Proposition 5.8}.  As long as
$\tau(g) \ge O(\delta)$, then any path $[x, gx]$ has a subsegment
which is contained in an $L_1$-neighbourhood of $\alpha_g$, and so by
thin triangles, the distance from $g x$ to $\alpha_g$ is at most
$\gp{g x}{x}{g^2 x} + L_1 + O(\delta)$.
\end{proof}

Let $\gamma_1$ and $\gamma_2 = [x, y]$ be two
$(1, K_1)$-quasigeodesics.  Parameterizations
$\gamma_1 \colon I_1 \to X$ and $\gamma_2 \colon I_2 \to X$ determine
orientations of $\gamma_1$ and $\gamma_2$.  Let $x' = \gamma_1(s)$ be
a nearest point on $\gamma_1$ to $x$, and let $y' = \gamma_1(t)$ be a
nearest point on $\gamma_1$ to $y$.  We say these orientations agree
if $s < t$ for any choice of nearest points $x' = \gamma_1(s)$ and
$y' = \gamma_1(t)$, and we say they disagree if $s > t$ for any choice
of nearest points $x' = \gamma_1(s)$ and $y' = \gamma_1(t)$.  In any
other case we say that the orientation of $\gamma_2$ is not
well-defined with respect to $\gamma_1$.  We omit the proof of the
following basic fact.

\begin{proposition}\label{P:orientation}
Given constants $\delta, K_1$ and $L$, there is a constant $L'$ with
the following properties.  Let $X$ be a $\delta$-hyperbolic space, and let $\gamma_1$ and $\gamma_2$
be $(1, K_1)$-quasigeodesics in $X$ such that $\gamma_2$ is contained in an $L$-neighbourhood of
$\gamma_1$.  If the length of $\gamma_2$ is at least $L'$, then
the orientation of $\gamma_2$ either agrees or disagrees with that of
$\gamma_1$.
\end{proposition}

Recall that we say a function $\cE(n) \colon \N \to \N$ is exponential
in $n$ if there are constants $B$ and $c < 1$ such that
$\cE(n) \le B c^{n}$ for all $s \ge 0$.  Clearly, if $\cE_1(n)$ is
exponential in $n$, and $\cE_2(n)$ is exponential in $n$, then the sum
of these two functions is exponential in $n$.

We may now complete the proof of Proposition \ref{P:asymmetric}.

\begin{proof}[Proof of Proposition \ref{P:asymmetric}]
If $L' \ge L$, then $N_L(\alpha_g) \subseteq N_{L'}(\alpha_g)$, so if
the result holds for some $K'$ and $L'$, it also holds for $K'$ and
$L$. Therefore, without loss of generality we may assume that
$L \ge 1 + \delta$.

Let $\alpha_{w_n}$ be a quasiaxis for $w_n$, and let $x'$ be the nearest
point projection of the basepoint $x$ to $\alpha_{w_n}$.  If the
result holds for some $\e > 0$, it holds for any larger value of $\e$,
so we may assume that $\e \le 1$.  Furthermore, as $\alpha_{w_n}$ is
$w_n$-invariant, after translating by a power of $w_n$, and possibly
replacing $\e$ by $\e/2$, we may assume that $w_n^i [p, q]$ is
contained in $[x', w_n x']$.  By abuse of notation, we will relabel
$w_n^i [p, q]$ as $[p, q]$.

If $h [p, q]$ is contained in a $L$-neighbourhood of $\alpha_{w_n}$,
then as $\alpha_{w_n}$ is $w_n$-invariant, then after replacing $h$ by
$w_n^k h$, we may assume that the nearest point projection of
$h [p, q]$ to $\alpha_{w_n}$ is contained in $[x', w_n^2 x']$.  By
abuse of notation, we will relabel $w_n^k h$ as $h$.

Given $L$, let $L'$ be the constant from Proposition
\ref{P:orientation}.  As $d(x', w_n x')$ tends to infinity almost
surely as $n$ tends to infinity, we may assume that
$d(x', w_n x') \ge L'/\e$, and so $d(p, q) \ge L'$.  In particular,
the orientation of $h [p, q]$ is well defined with respect to
$\alpha_{w_n}$, and either agrees, or disagrees with the orientation
of $\alpha_{w_n}$.

First consider the case in which $h$ reverses the orientation of
$[p, q]$ with respect to $\alpha_{w_n}$, as illustrated below in
Figure \ref{pic:reverse}.  We will show that if this occurs, it gives
a self-match for $\gamma_n$ which occurs with probability which is at
most exponential in $n$.

\begin{figure}[H]
\begin{center}
\begin{tikzpicture}

\tikzstyle{point}=[circle, draw, fill=black, inner sep=1pt]

\draw (3, 0) -- (15, 0) node [above right] {$\alpha_{w_n}$};

\draw (4, 0) node [point, label=above:$x'$] {};
\draw (9, 0) node [point, label=above:$w_n x'$] {};
\draw (14, 0) node [point, label=above:$w_n^2 x'$] {};

\draw [arrows=-triangle 45] (5, 0) -- (6.5, 0);
\draw [arrows=-triangle 45] (9.5, -0.5) -- (8, -0.5);

\draw (5, 0) node [point, label=above:$p$] {};
\draw (8, 0) node [point, label=above:$q$] {};

\draw (6.5, -0.5) node [point, label=below:$h q$] {} --
      (9.5, -0.5) node [point, label=below:$h p$] {};

\end{tikzpicture}
\end{center}\caption{An orientation reversing translate of $[p, q]$
  close to $\alpha_{w_n}$.}\label{pic:reverse}
\end{figure}
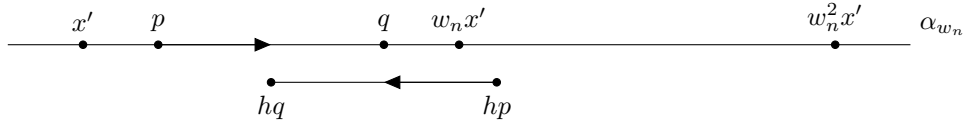

By replacing $[p, q]$ by either its initial half, or terminal half, we
may assume that either $[p, q]$ or $w_n^{-1}[p, q]$ has nearest point
projection to $\alpha_{w_n}$ contained in $[x', w_n x']$.  Again
replacing $[p, q]$ by either its initial half, or terminal half, we
may assume that $h[p, q]$ lies within distance $K$ of a disjoint
subsegment of $[x', w_n x']$ of length at least
$\e d(x', w_n x') / 4$.  This gives rise to an
$(\e d(x', w_n x') /4, K)$-self match for $[x', w_n x']$.

Let $\ell > 0$ be the linear progress constant for $\mu$, and fix some
$0 < \e' < \min \{ \ell, 1 \} /2$.

The subsegment $[x', w_n x']$ of $\alpha_{w_n}$ is contained in an
$L_1$-neighbourhood of $[x, w_n x]$, and by Proposition
\ref{P:match-axis}, given $\e' > 0$, there are constants $B_1$ and
$c_1 < 1$ such that the probability that the length of $[x', w_n x']$
is at least $(\ell - \e')n$ is at least $1 - \cE_1(n)$, where
$\cE_1(n) = B_1 c_1^n$, where $\ell$ is the linear progress constant
for $\mu$.

This gives an $(\e(\ell - \e')n/4, K)$-self match for
$[x, w_n x]$, and by Proposition \ref{P:self-match}, there are
constants $B_2$ and $c_2 < 1$ such that the probability that this
occurs is at most $\cE_2(n) = B_2 c_2^{n}$.

Therefore, the existence of an orientation reversing translate of
$[p, q]$ occurs with probability at most $\cE_1(n) + \cE_2(n)$, which
is exponential in $n$, as required.

We now consider the case in which the orientation of $h [p, q]$ agrees
with that of $\alpha_{w_n}$.  We may replace $[p, q]$ by either its
initial half or terminal half subinterval (in which case replace $\e$
by $\e/2$), and possibly replace $h$ by $w_n^{-1} h$, to ensure that
the nearest point projection of $h [p, q]$ to $\alpha_{w_n}$ is
contained in $[x', w_n x']$.  This is illustrated below in Figure
\ref{pic:asymmetric}.

\begin{figure}[H]
\begin{center}
\begin{tikzpicture}

\tikzstyle{point}=[circle, draw, fill=black, inner sep=1pt]

\draw (3, 0) -- (15, 0) node [above right] {$\alpha_{w_n}$};

\draw (4, 0) node [point, label=above:$x'$] {};
\draw (14, 0) node [point, label=above:$w_n x'$] {};

\draw (6, -0.5) node [point, label=below:$h p$] {} --
      (11, -0.5) node [point, label=below:$h q$] {};

\draw (6, 0) node [point, label=above:$p'$] {};
      
\draw (5, 0) node [point, label=above:$p$] {};
\draw (10, 0) node [point, label=above:$q$] {};

\begin{scope}[xshift=+1cm]
\draw (6, 0) node [point, label=above:$t$] {};
\draw (7, 0) node [point, label=above:$t'$] {};
\draw (7, -0.5) node [point, label=below:$h t$] {};
\draw (8, -0.5) node [point, label=below:$h t'$] {};
\draw (8, 0) node [point, label=above:$t''$] {};
\draw (8, -1.25) node [point, label=below:$h^2 t$] {};
\end{scope}

\end{tikzpicture}
\end{center}\caption{An orientation preserving translate of
  $[p, w_n p]$ close to $\alpha_{w_n}$.}\label{pic:asymmetric}
\end{figure}
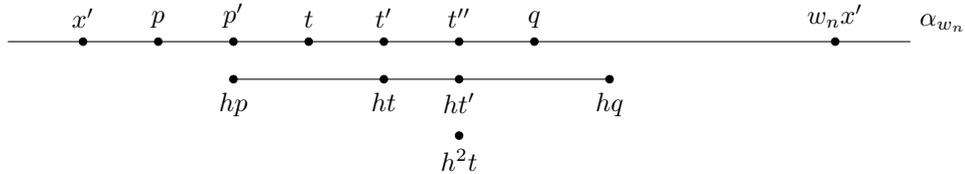

Let $p'$ be a nearest point on $\alpha_{w_n}$ to $hp$.
If $d(p, p') \ge \e \ell n/10$, then this gives a linear size
self-match of $[x, w_n x]$, and again by Proposition
\ref{P:self-match} there are constants $B_3$ and $c_3 < 1$ such that
the probability that this occurs is at most
$\cE_3(n) = B_3 c_3^{n}$.

We shall choose a constant $K = 4L + O(\delta)$, but in order to
guarantee that there is no circularity in our choice of constants, we
now recall some basic facts about Gromov hyperbolic spaces and give an
explicit choice of the $O(\delta)$ term in terms of geometric
constants which only depend on $\delta$.

Recall that every quasiaxis is a $(1, K_1)$-quasigeodesic, where $K_1$ only
depends on $\delta$.  Let $L_1$ be a Morse constant for
$(1, K_1)$-quasigeodesics, i.e. any geodesic $[x, y]$ with endpoints
in a $(1, K_1)$-quasigeodesic $\alpha$ is contained in an
$L_1$-neighbourhood of $\alpha$.  As $K_1$ only depends on $\delta$,
the Morse constant $L_1$ also only depends on $\delta$.

Given constants $\delta \ge 0$ and $K_1 \ge 0$ there are constants
$K_2$ and $K_3$, such that for any $(1, K_1)$-quasigeodesic $\alpha$,
and any two points $x$ and $y$ in $X$, if $x'$ is the nearest point
projection of $x$ to $\alpha$ and $y'$ is the nearest point projection
of $y$ to $\alpha$, then if $x'$ and $y'$ are distance at least $K_2$
apart, then the geodesic from $x$ to $y$ is Hausdorff distance at most
$K_3$ from the piecewise geodesic path
$[x, x'] \cup [x', y'] \cup [y', y]$.  Furthermore
\begin{equation}\label{eq:npp-bound}
d(x', y') \ge d(x, y) - d(x, x') - d(y, y') - K_3. 
\end{equation}
As $K_1$ only depends on $\delta$, the constants $K_2$ and $K_3$ also
only depend on $\delta$.  We may now set
$K = 4 L + 2 K_1 + 3 K_2 + 3 K_3 + 6 \delta$.

Now suppose that $p'$ is close to $p$ and the length of $[p, p']$ is
greater than $K$ but less than $\e \ell n/10$.  Let $t$ be any point
in $[p', q]$.  Let $t'$ be a nearest point on $[p, q]$ to $h t$, and
let $t''$ be a nearest point on $[p, q]$ to $h t'$.

\begin{claim}
We have chosen $K$ sufficiently large such that $d(t, t') \ge K_2$.
\end{claim}

\begin{proof}
By \eqref{eq:npp-bound},
\begin{align*}
d(p', t') & \ge d(hp, ht) - d(hp, p') - d(ht, t') - K_3. \\
\intertext{%
As $h$ is an isometry, and $d(hp, p')$ and $d(ht, t')$ are
  at most $L$, this gives%
}
d(p', t') & \ge d(p, t) - 2L - K_3. 
\intertext{%
The points $p, p', t$ and $t'$ all lie on the $(1,
           K_1)$-quasigeodesic $\alpha_{w_n}$, which implies $d(p', t) +
           d(t, t') \ge d(p', t') - K_1$, and $d(p, t) \ge d(p, p') + d(p', t) -
           K_1$.  This yields
           }           
d(t, t') & \ge d(p, p') - 2L - 2 K_1 - K_3.
\end{align*}
Our choice of $K$ therefore guarantees that $d(t, t') \ge K_2$, as
required.  In fact $d(t, t') \ge 2L + K_2 + K_3 \ge K_2$, and we will
now use this stronger bound to obtain a bound on $d(t', t'')$.
\end{proof}

\begin{claim}
We have chosen $K$ sufficiently large such that $d(t', t'') \ge K_2$.
\end{claim}

\begin{proof}
By \eqref{eq:npp-bound},
\begin{align*}
d(t', t'') & \ge d(ht, ht') - d(ht, t') - d(ht', t'') - K_3. \\
  \intertext{%
  as $h$ is an isometry, and $d(ht, t')$ and $d(ht', t'')$ are
  at most $L$, this gives
  }
d(t', t'') & \ge d(t, t') - 2L - K_3.
\end{align*}
Our choice of $K$ then implies that $d(t', t'') \ge K_2$, as required.
\end{proof}

As $d(t', t'') \ge K_2 + L$, the geodesic from $ht$ to $h^2 t$ passes
within distance $K_3$ of $[t', t'']$, the Gromov product
$\gp{ht}{t}{h^2 t}$ is at most $K_4 := L + K_2 + K_3 + 2 \delta$.  We
have chosen $K$ sufficiently large such that $d(t, ht) \ge 3 K_4$, and
so Proposition \ref{P:concatenate} implies that $h$ is loxodromic, and
any
quasiaxis of $h$ passes within distance $2 K_4$ of $\alpha_{w_n}$.

As we have assumed that $\tau(h) \le \e \ell n / 10$, this gives a
$(\e \ell n / 10, 2 K_4)$-self match of $[x', w_n x']$, and hence of
$\gamma_n = [x, w_n x]$, and so again by Proposition
\ref{P:self-match} there are constants $B_4$ and $c_4 < 1$ such that
the probability that this occurs is at most
$\cE_4(n) = B_4 c_4^{n}$.

Therefore, we have shown that the case of an orientation preserving
translate of $[p, q]$ occurs with probability at most $\cE_3(n) +
\cE_4(n)$, which is exponential in $n$, as required.
\end{proof}

\section{Small cancellation and normal closure}

We will now prove results on the normal closure (Theorems
\ref{T:inj-intro} and \ref{T:normal-intro} in the Introduction).  In
order to do so, we will use the following notions of \emph{small
  cancellation} from \cite{DGO}.  We remark that the small
cancellation results in this section were previously obtained in the
case of acylindrical actions by Maher and Sisto \cite{maher-sisto},
using work of Hull \cite{hull}, and we further extend their methods to
the case of WPD actions.  If $H \subseteq G$ is a subgroup, we define
its \emph{injectivity radius} as
$$\textup{inj}(H) := \inf \{ d(gx, x) \ : \ g \in H \setminus \{1\}, x \in X \}.$$
Let $\mathcal{R}$ be a family of loxodromic elements which is closed
under conjugation. We define its \emph{injectivity radius} as
$$\textup{inj}(\mathcal{R}) := \inf_{g \in \mathcal{R}} \inf \{ d(g^k x, x), k \in \mathbb{Z} \setminus \{0 \}, x \in X\}.$$
In particular, if $g$ is loxodromic and $\mathcal{R} := \{ h g h^{-1}, h \in G \}$ is the set of conjugates of $g$, then 
$$\textup{inj}(\mathcal{R}) \ge \tau(g).$$
Following \cite{DGO}, for a loxodromic element $g$, let
$\textup{Ax}(g)$ be the $20 \delta$-neighbourhood of set of points $x$ for
which $d(x, gx) \le \inf_{y \in X} d(y, gy) + \delta$.  If $\tau(g)$
is sufficiently large, then this set is contained in a bounded
neighbourhood of a  
quasiaxis $\alpha_g$ for $g$.

\begin{proposition}\label{P:ax}
Given $\delta \ge 0$, there are constants $A$ and $K$, such that if
$g$ is a loxodromic isometry of $\delta$-hyperbolic space $X$ with
quasiaxis $\alpha_g$ and $\tau(g) \ge A$, then
$\textup{Ax}(g) \subset N_K(\alpha_g)$.  Furthermore, $\textup{Ax}(g)$
is $10 \delta$-quasiconvex.
\end{proposition}

\begin{proof}
Let $x$ be a point in $X$, and let $p$ be a nearest point on
$\alpha_g$ to $x$.  As we may assume that $\alpha_g$ is $g$-invariant,
$gp$ is a nearest point on $\alpha_g$ to $gx$, and
$d(p, gp) \ge \tau(g)$.  Given $\delta$, there are constants $A_1$ and
$K_1$ such that if $d(p, gp) \ge A_1$, then the union of the three
geodesic segments $[x, p], [p, gp]$ and $[gp, gx]$ is contained in a
bounded neighbourhood of a geodesic $[x, gx]$, and in particular,
\[ d(x, gx) \ge d(x, p) + d(p, gp) + d(gp, gx) - K_1.  \]
This is an elementary application of thin triangles, see for example
\cite[Proposition 2.3]{MT} for the geodesic case.  As the
quasigeodesics constants for the quasiaxis $\alpha_g$ only depend on
$\delta$, $A_1$ and $K_1$ may also be chosen to only depend on
$\delta$.  Therefore, if $d(x, p) \ge B_1 + \delta$ then $x$ does not
lie in $\textup{Ax}(x)$, so we may choose $A = A_1$ and
$K = K_1 + \delta$.

For the final statement, see for example Coulon \cite[Proposition 3.10]{coulon}.
\end{proof}

We also define, for $g$ and $h$ loxodromic,
$$\Delta(g, h) := \textup{diam } \left( N_{20\delta}(\textup{Ax}(g)) \cap N_{20 \delta}(\textup{Ax}(h)) \right)$$
where $N_R(Y)$ denotes the $R$ neighbourhood of the set $Y$ in $X$.

Recall that $E_G(h)$ is the maximal virtually cyclic subgroup
containing $h$, which is equal to the stabilizer of the endpoints
$\{ \lambda_h^-, \lambda_h^+ \} $ of $h$ in $\partial X$.  We now
record the following elementary property of $E_G(h)$, that the image
of this group in $X$ under the orbit map intersects any bounded set in
only finitely many points.

\begin{lemma}\label{P:bounded}
Let $G$ be a group acting on a Gromov hyperbolic space $X$ which
contains a loxodromic isometry $h$, and let $H$ be a subgroup of
$G$ which contains $\langle h \rangle$ as a finite index subgroup.
Then for any $x \in X$ and $K \ge 0$, there is an $N$ such that
$\setnorm{ Hx \cap B_K(x) } \le N$.
\end{lemma}

\begin{proof}
As $\langle h \rangle$ is a finite index subgroup of $H$, there is a
finite set of group elements $F$ such that $H$ is a finite union of
right cosets $\langle h \rangle f$, for $f \in F$.  In particular, any
element $g \in H$ may be written as $g = h^k f$, for some $k \in \N$
and $f \in F$.  By the triangle inequality,
$d(x, g x) \ge d(x, h^k x) - d(x, f x)$.  The distances
$d(x, f x)$ have an upper bound depending on $F$ and $x$, and
$d(x, h^k x) \ge k \tau(h)$, so there are only finitely many group
elements $g \in H$ with $d(x, g x) \le K$.
\end{proof}

Let $g$ be a loxodromic element in $G$.  We shall write $E_G^+(g)$ for
the orientation preserving subgroup of $E_G(g)$, i.e. the subgroup
which stabilizes $\lambda^+_g$ and $\lambda^-_g$ pointwise.  This
group is either equal to $E_G(g)$ or has index two in $E_G(g)$.  There
are elements $g$ with $E_G(g) = E^+_G(g)$, and in fact they are
generic.

\begin{corollary}\label{C:orientable}
Let $G$ be a group acting by isometries on a $\delta$-hyperbolic space
$X$, and let $\mu$ be a countable, non-elementary, bounded probability
distribution on $G$.  Then there are constants $B$ and $c < 1$ such
that the probability that $w_n$ is loxodromic with
$E_G(w_n) = E_G^+(w_n)$ is at least $1 - Bc^{n}$.
\end{corollary}

\begin{proof}
If $E^+_G(w_n)$ is index two in $E_G(w_n)$, then there is an element
$f$ which reverses the orientation of $\alpha_{w_n}$.  This gives an
$(\ell n / 4, K)$-self match of $[p, w_n p]$, where $\ell > 0$ is the
positive drift constant for $\mu$, and $K$ is the fellow travelling
constant from Proposition \ref{P:fellow travel2}.  However by
Corollary \ref{C:self-match}, there are constants $B$ and $c < 1$ such
that the probability that this occurs is at most $Bc^{n}$.
\end{proof}

An essential feature of asymmetric elements is the following.

\begin{proposition}\label{P:splits}
Given $\delta \ge 0$, there are constants $K$ and $L$ such that if $g$
is a WPD element of $G$ which is $(1, L, K)$-asymmetric, with
translation length $\tau(g) > 3 L + 2 K$, then there is a surjective
homomorphism $\phi \colon E^+_G(g) \to \Z$ with $\phi(g) = 1$.  In
particular, 
$$E^+_G(g) = \langle g \rangle \ltimes \ker \phi,$$ 
where
$\ker \phi$ is finite and consists precisely of the
elliptic elements of $E^+_G(g)$.
\end{proposition}

Note that the proposition is not true if one replaces $E^+_G(g)$ by
$E_G(g)$, as the latter may contain infinitely many elliptic elements
(think of the action of the infinite dihedral group on $\mathbb{Z}$).

\begin{proof}
Let $p$ be a point on a quasiaxis $\alpha_g$.  Let $L$ be the fellow
travelling constant from Proposition \ref{P:fellow travel2}.  The quasiaxis
$\alpha_g$ is $L$-coarsely preserved by $E^+_G(g)$.  As $g$ is
$(1, L, K)$-asymmetric, the set $\{ g^i p \colon i \in \Z\}$ is
$K$-coarsely preserved by $E^+_G(g)$.  As elements act by isometries,
this gives an action of $E^+_G(g)$ on $\Z$, defined as follows.  If
$f \in E^+_G(g)$, $\phi(f)$ sends $g^i p$ to the nearest $g^j p$ to
$f g^i p$.
As $g$ is WPD, the group $E^+_G(g)$ is
virtually cyclic, so $\ker \phi$ is finite.  The element
$g \in E^+_G(g)$ maps to $1 \in \Z$ and gives a splitting, so
$E_G^+(g) = \langle g \rangle \ltimes \ker \phi$.

As $\ker \phi$ is a finite subgroup of $G$, all elements of
$\ker \phi$ are elliptic.  If $\phi(f) \not = 0$, then as
$\tau(g) \ge 3 L + 2 K$, the three points $p, f p$ and $f^2 p$ satisfy
$d(p, fp) \ge 3 L$, $d(fp, f^2 p) \ge 3 L$ and
$\gp{f p}{p}{f^2 p} \le L$, and so $f$ is loxodromic by Proposition
\ref{P:concatenate}.
\end{proof}

Let $G_{\textup{WPD}}$ denote the set of WPD elements of $G$, and let
$H \le G$ be a subgroup of $G$ which contains an element of
$G_{\textup{WPD}}$.
Define
$$E^+_G( H ) := \bigcap_{g \in H \cap G_{\textup{WPD}}} E^+_G(g). $$
and an equivalent definition holds for $E_G(H)$. We will also use the notation $E(G) := E_G(G)$
when $G$ and $H$ are equal.

Recall that two elements $h_1, h_2$ of $G$ are \emph{commensurable} if
some power of $h_1$ is conjugate to some power of $h_2$, and
\emph{non-commensurable} otherwise.  The result below follows from the
arguments in \cite{DGO}*{Lemma 6.17}, but we give the details for the
convenience of the reader.

\begin{proposition}\label{P:h1-h2}
Let $G$ be a group acting by isometries on a Gromov hyperbolic space
$X$, and let $H$ be a non-elementary subgroup of $G$ which contains an
element of $G_{\textup{WPD}}$.  Then there exist two independent, WPD
elements $h_1$, $h_2$ in $H$ such that
$$E^+_G(h_1) \cap E^+_G(h_2) = E^+_G(H).$$
Moreover, for any $K \geq 0$ there exists an element $f$ in $H$ such
that for any $z \in \alpha_f$ one has
$$\textup{Stab}_K(z, fz) \subseteq E_G^+(H).$$
\end{proposition}

\begin{proof}

By \cite[Corollary 6.12]{DGO}, there exist two non-commensurable,
loxodromic, WPD elements $h_1$, $h_2$ in $H$ (pick $h_1$ as one such
element, then apply Corollary 6.12 with the subgroup called $G$ in
Corollary 6.12 chosen to be $H$, the subgroup called $H$ in the
Corollary 6.12 chosen to be $E_G(h_1)$ and
$a \in H \setminus E_G(h_1)$).  Let $N$ be the normalizer of $H$ in
$G$, i.e.
$$N := \{ x \in G \ : \ x H x^{-1} = H \}$$
which contains the group $H$. 
Denote as $T(h_i)$ the set of finite order elements in $E_G^+(h_i)$.
In $E_G^+(h_i)$ every conjugacy class is finite (since all conjugate
elements have equal translation length), so a result of Neumann
\cite{neumann} then implies that the set $T(h_i)$ of finite order
elements is a finite group.
Let us suppose that for any $x \in N$ we have
$$E^+_G(x h_1 x^{-1}) \cap E^+_G(h_2) \neq E^+_G(H).$$
Note moreover that 
$$E^+_G(x h_1 x^{-1}) \cap E^+_G(h_2) = x T(h_1) x^{-1} \cap T(h_2).$$
Given $(s, t) \in P:= T(h_1) \times (T(h_2) \setminus E^+(H))$, we pick
$y \in N$ such $y s y^{-1} = t$, if it exists, and $y(s, t) = 1$
otherwise.  Let $C_N(t)$ be the centralizer of $t$ in $N$.  Now, we
claim that
$$N = \bigcup_{(s, t)\in P} y(s, t) C_N(t).$$
Indeed, let $x \in N$. Then since
$x T(h_1) x^{-1} \cap T(h_2) \neq E_G^+(H)$, then there exists
$s \in T(h_1)$ and $t \in T(h_2) \setminus E^+(H)$ such that
$s = x^{-1} t x \in T(h_1)$. Thus if $y = y(s, t)$ then
$s = x^{-1} t x = y^{-1} t y$, so $x y^{-1} \in C_N(t)$.  This means
that there is a finite collection of cosets of the subgroups $C_N(t)$,
with $t \in T(h_2) \setminus E^+(H)$, which covers $N$, and a theorem
of Neumann \cite{neumann2} then implies that at least one of these
subgroups has finite index in $N$.  Therefore, there is a
$t \in T(h_2) \setminus E_G^+(H)$ such that $C_N(t)$ has finite index
in $N$. Hence, if $h \in N$ is a WPD element, then there exists
$k > 0$ such that $h^k t = t h^k$, hence $t \in E^+_G(h)$. Thus,
$t \in E^+_G(N) \subseteq E^+_G(H)$, which is a contradiction.
Finally, let us note that the claim implies that $h_1$ and $h_2$ are
independent. In fact, as both $h_1$ and $h_2$ are WPD, the fixed point
sets of $h_1$ and $h_2$ cannot have a common point.  This is because
in this case both $h_1$ and $h_2$ would coarsely stabilize a large
segment of the quasiaxis of $h_1$, which by Theorem \ref{T:bf}, would imply
that $E^+_G(h_1) = E^+_G(h_2)$, contradicting the non-commensurability
of $h_1$ and $h_2$.

We now prove the second claim. As $h_1$ and $h_2$ are independent
loxodromic isometries, the ping-pong lemma implies that for any
$n > 0$ sufficiently large, the orbit map gives a quasi-isometric
embedding of the free group $\langle h_1^n, h_2^n \rangle$ in $X$.  In
particular, for all $m > 0$, the element
$f := h_1^{nm} h_2^{nm}$ is loxodromic.

Fix some $K \ge 0$, and let $L_1$ be the fellow travelling constant
for $(1, K_1)$-quasigeodesics from Proposition \ref{P:fellow
  travel3}. Let $L_2$ be the constant given by Theorem \ref{T:bf}
using the constant $K + 2 \delta + L_1$.  We may choose $m$
sufficiently large so that there are two segments
$\eta_1 \subseteq \alpha_{h_1}$ and $\eta_2 \subseteq \alpha_{h_2}$ of
length $\geq L_2$, and a segment $\eta \subseteq \alpha_f$ such that
$$\eta_1 \cup \eta_2 \subseteq N_{L_1}(\eta). $$
Thus, if $h$ belongs to $\stab_K(z, fz)$, then for some
$k \in \mathbb{Z}$ the isometry $f^k h f^{-k}$ $(K+2\delta)$-coarsely
stabilizes the segment $\eta$, hence it also
$(K + 2\delta + L_1)$-coarsely stabilizes both $\eta_1$ and $\eta_2$,
and preserves the orientation of the axes.  Then by Theorem \ref{T:bf}
it is contained in
$$E^+_G(h_1) \cap E^+_G(h_2) = E^+_G(H).$$
Thus, $h$ belongs to
$f^{-k} E^+_G(H) f^{k} = E^+_G(H)$, as required.
\end{proof}

From now on we shall assume that the probability distribution $\mu$ is
reversible, so $\Gamma_\mu$ is a group. We will use the notation $E_\mu :=  E_G^+(\Gamma_\mu)$.

\begin{corollary}\label{C:Eplus}
Given $\delta \ge 0$ there are constants $K$ and $L$ with the
following properties.  Let $G$ be a group acting by isometries on a
$\delta-$hyperbolic space $X$, and let $\mu$ be a countable,
non-elementary, reversible, bounded, WPD probability distribution on
$G$.  Then there are constants $B$ and $c < 1$ such that the
probability that $w_n$ is loxodromic, $(1, L, K)$-asymmetric, WPD with
$$E_G(w_n) = E_G^+(w_n) =  \langle w_n \rangle \ltimes E_\mu$$ 
is at least $1 - Bc^{n}$.  In particular, if
$E_\mu$ is trivial, then $E_G(w_n)$ is cyclic with
probability at least $1 - Bc^{n}$.
\end{corollary}

\begin{proof}
We are left with proving the last claim. By Proposition
\ref{P:asymmetric}, we know that there are constants $B_1$ and
$c_1 < 1$ such that the probability that $w_n$ is
$(1, L, K)$-asymmetric is at least $1 - B_1 c_1^{n}$, hence
$$E_G^+(w_n) = \langle w_n \rangle \ltimes \ker \phi$$
where $\phi : E_G^+ \to \mathbb{Z}$ is the homomorphism given in
Proposition \ref{P:splits}.  Now, since $w_n$ is asymmetric, we have
that $\ker \phi$ is the (finite) set of elliptic elements in
$E_G^+(w_n)$, hence it is contained in $\stab_K(p, w_n p)$ where $p$
is some point on the quasiaxis of $w_n$.  Let $f \in \Gamma_\mu$ be given
by Proposition \ref{P:h1-h2}.  By Proposition \ref{P:match-use}, there
are constants $B_2$ and $c_2 < 1$ such that the probability the quasiaxis
of $w_n$ has a $(L, K)$-match with a translate of the quasiaxis of $f$ is
at least $1 - B_2 c_2^{n}$.  Therefore, for
$K' = 2 K + 2 \delta$ we get for some $z \in \alpha_f$
$$\ker \phi \subseteq \stab_K(p, w_n p) \subseteq g \stab_{K'}(z, fz)
g^{-1}\subseteq E^+_G(\Gamma_\mu) = E_\mu. $$
The result then holds for suitable choices of $B$ and $c < 1$.
\end{proof}

Given $g \in G$ a loxodromic element, let us define the \emph{fellow travelling constant} for $g$ as
$$\Delta(g) := \sup_{h \in G \setminus E(g)} \Delta(g, hgh^{-1})$$
where $E(g)$ is the maximal elementary subgroup which contains $g$. 

\begin{definition}(\cite{DGO}*{Definition 6.25}) %
Let $X$ be a $\delta$-hyperbolic space with $\delta > 0$, and let
$\mathcal{R}$ be a family of loxodromic isometries of $X$ which is
closed under conjugation. Then we say that $\mathcal{R}$ satisfies the
$(A, \epsilon)$-small cancellation condition if the following holds:
\begin{enumerate}

\item $\textup{inj}(\mathcal{R}) \geq A \delta$

\item $\Delta(g, h) \leq \epsilon \cdot \textup{inj}(\mathcal{R})$ for
all $g \neq h^{\pm 1} \in \mathcal{R}.$

\end{enumerate}

\end{definition}

We will now prove that the cyclic subgroup generated by a power of
$w_n$ satisfies the small cancellation condition.  First of all, we
show that the fellow travelling constant between translates of the
quasiaxis is sublinear in $n$.

\begin{proposition} \label{P:Delta-small} %
Let $G$ be a group of isometries of a $\delta$-hyperbolic metric space
$X$, and $\mu$ a countable, non-elementary, reversible, bounded, WPD
probability measure on $G$.  Let $\ell > 0$ be the drift of the random
walk.  Then for any $0 < \epsilon < 1$, there are constants $B$ and
$c < 1$ such that for all $n$ the fellow travelling constant of $w_n$ satisfies 
$$\mathbb{P}(\Delta(w_n) \geq \epsilon \ell n) \le B c^{n}.$$
\end{proposition}

\begin{proof}
By Proposition \ref{P:ax}, there is an $L$ such that
$N_{20 \delta}(\textup{Ax}(w_n)) \subset N_{L/2}(\alpha_{w_n})$.
Therefore, if $\Delta(w_n) \ge \e \ell n$, there is a translate
$h \alpha_{w_n}$, with $h \not \in E(w_n)$, such that $\alpha_{w_n}$
and $h \alpha_{w_n}$ have a $(\e \ell n, L)$-match.
This by definition means that there is a segment
$\eta = [p, q] \subseteq \alpha_{w_n}$ with $\norm{\eta}$ equal to
$\e \ell n$, such that $h \eta$ is contained in an $L$-neighbourhood
of $\alpha_{w_n}$. By replacing $\eta$ with $w_n^i \eta$ for some $i \in \mathbb{Z}$ 
and replacing $\epsilon$ by $\epsilon /2$, we can assume that $\eta \subseteq [x', w_n x']$
where $x'$ is a nearest point projection of the basepoint $x$ to $\alpha_{w_n}$.

By Proposition \ref{P:asymmetric}, there are constants $B_1$ and
$c_1 < 1$ such that the element $w_n$ is $(\epsilon, L, K)$-asymmetric
with probability at least $1 - B_1 c_1^{n}$.  Thus there is a
$K$, depending on $\e$ and $L$, such that up to replacing $h$ by
$w_n^j h$ for some $j \in \mathbb{Z}$, we may assume that
$d(p, hp) \le K$ and $d(q, hq) \le K$.

Let $f$ be given as in the second part of Proposition \ref{P:h1-h2}.
As $[p, q]$ has length $\e \ell n$ and is contained in $[x', w_n x']$,
by Lemma \ref{L:match2} there are constants $B_2$ and $c_2 < 1$ such
that the probability that it contains a match with a large subsegment
of a translate $g \alpha_f$ of a quasiaxis $\alpha_f$ (where
$g \in \Gamma_\mu$)
is at least $1 - B_2 c_2^{n}$.

As $h$ $K$-coarsely stabilizes this subsegment,  
this implies that there exists $z \in \alpha_f$ such that by Proposition \ref{P:h1-h2},
$$h \in \stab_K(gz , g f z) =  g \stab_K(z , f z) g^{-1} \subseteq g E_G^+(\Gamma_\mu) g^{-1} = E_G^+(\Gamma_\mu), $$
hence, since by construction $E_G^+(\Gamma_\mu) \subseteq E^+_G(w_n)$
and, by Corollary \ref{C:orientable}, there are constants $B_3$ and
$c_3 < 1$ such that the probability that $E^+_G(w_n)= E_G(w_n)$ is at
least $1 - B_3 c_3^{n}$.  Therefore, by suitable choices of $B$
and $c < 1$, any such $h$ must lie in $ E_G(w_n)$ with probability at
least $1 - B c^{n}$.  However, this contradicts our initial
choice of $h$, and implies that $\Delta(w_n) \ge \e \ell n$ with
probability at most $B c^{n}$, as required.
\end{proof}

\subsection{The structure of the normal closure}

The last step we need to understand the structure of the normal closure $\langle \langle w_n \rangle \rangle$ of $w_n$ in $G$
is to take care of the fact that the elementary subgroup $E_G^+(w_n)$ need not be cyclic, so we may have to pass to a power of
$w_n$. However, the power may be chosen to be a constant which only depends on $G$ and $\mu$, as we now explain.

Let $\Gamma_\mu$ be the group generated by the support of $\mu$, and let $E_\mu := E^+_G(\Gamma_\mu)$. 
By definition, $E_\mu$ is a normal subgroup of $\Gamma_\mu$, hence one has the homomorphism 
\begin{equation} \label{E:conj}
\varphi: \Gamma_\mu \to \textup{Aut } E_\mu
\end{equation}
given by conjugation: $g \mapsto (k \mapsto g k g^{-1})$. We will denote as $H_\mu := \varphi(\Gamma_\mu)$ the image of $\varphi$. 

\begin{lemma} \label{L:EZ}
The image of $\varphi$ in $\textup{Aut } E_\mu$ is trivial if and only if $E_\mu = Z(\Gamma_\mu)$. 
\end{lemma}

\begin{proof}
First note that $Z(\Gamma_\mu) \subseteq E_\mu$. In fact, let $g \in Z(\Gamma_\mu)$ and let $h \in \Gamma_\mu$ be a loxodromic, WPD element. 
Then $g h g^{-1} = h$, hence $\textup{Fix} (g h g^{-1}) = g \textup{Fix}(h) = \textup{Fix}(h)$, hence $g \in E_G(h)$. Since this is true for any $h$ WPD, 
then $g \in E_\mu$. 

Moreover, the kernel of $\varphi$ is the set of $g$ which commute with every element of $E_\mu$, hence the image is trivial if and only if every element 
of $E_\mu$ commutes with every element of $\Gamma_\mu$, which means that $E_\mu \subseteq Z(\Gamma_\mu)$. 
\end{proof}

Now, by Corollary \ref{C:Eplus}, with probability which tends to $1$, $E_G(w_n)$ is the semidirect product
$$E_G(w_n) = \langle w_n \rangle \ltimes E_\mu$$
and the group structure of $E_G(w_n)$ is determined by the map $\langle w_n \rangle \to \textup{Aut }E_\mu$, hence
by the image $\varphi(w_n)$ in $\textup{Aut }E_\mu$.

\begin{lemma} \label{L:semidirect}
Let $K$ be a finite group, let $\psi \in \textup{Aut }K$, and consider the semidirect product 
$$H = \mathbb{Z} \ltimes_{\psi} K$$
where we denote as $t$ a generator for $\mathbb{Z}$, so that $tkt^{-1} = \psi(k)$ for any $k \in K$. Then:
\begin{enumerate}
\item 
for any $a \in \mathbb{Z} \setminus \{ 0 \}$, if $\psi(t^a) = 1$, then the normal closure of $t^a$ in $H$ is cyclic and equal to $\langle t^a \rangle$;
\item
if $\psi(t) \neq 1$, then the normal closure of $t$ in $H$ is not cyclic and not free; 
\end{enumerate}
\end{lemma}

\begin{proof}
Let $u = t^a$, and suppose that $\psi(u) = 1$. Then for any $k \in K$ we have $k u k^{-1} = u$ and since by construction $u$ commutes with $t$, then $u$ commutes with $H$, hence the normal closure $\langle \langle u \rangle \rangle = \langle u \rangle$ is infinite cyclic. 

Now, since $H$ is virtually cyclic and the subgroup of a free group is free, then the normal closure $N := \langle \langle t \rangle \rangle$ is free if and only if it is infinite cyclic.
Moreover, since $t$ generates $\mathbb{Z}$, the only cyclic group which contains $\langle t \rangle$ is $\langle t \rangle$ itself. Hence $\langle \langle t \rangle \rangle$
is free if and only if it coincides with $\langle t \rangle$.
If the image $\phi(t)$ is not trivial, then there exists $k \in K$ such that $k t k^{-1} \neq t$, hence the normal closure is larger than $\langle t \rangle$, 
hence not free.
\end{proof}

\begin{lemma} \label{L:conj-E}
Let $h \in G$ be a loxodromic, WPD element, and let $g \in G$. Then if $ghg^{-1} \in E_G(h)$, then $g \in E_G(h)$. 
\end{lemma}

\begin{proof}
Suppose that $ghg^{-1} \in E_G(h)$, and let $\Lambda := \{\lambda^+, \lambda^-\}$ be the set of fixed points of $h$ on $\partial X$. Then by the assumption $g h g^{-1}$ also fixes $\Lambda$, hence by conjugating $h$ fixes $g^{-1} \Lambda$. Since $h$ fixes exactly two points on the boundary, then $\Lambda = g^{-1} \Lambda$, 
which implies that $g \in E_G(h)$. 
\end{proof}

We are now ready to present the main Theorem (Theorems \ref{T:normal-intro} and \ref{T:inj-intro}) and its proof.

\begin{theorem}  \label{T:inj-later} 
Let $G$ be a group acting on a Gromov hyperbolic space $X$, and let $\mu$ be a countable, non-elementary, reversible, bounded, WPD
probability measure on $G$. Let $k = k(\mu)$ be the characteristic index of $\mu$. Then:
\begin{enumerate}
\item the probability that the normal closure $\langle \langle w_n \rangle \rangle$ of $w_n$ in $G$ is free satisfies
$$\mathbb{P}(\langle \langle w_n \rangle \rangle \textup{ is free}) \to \frac{1}{k}$$
as $n \to \infty$.
As a corollary, this probability tends to $1$ if and only if $E_\mu = Z(\Gamma_\mu)$. 
\item
Moreover, 
$$\mathbb{P}(\langle \langle w_n^k \rangle \rangle \textup{ is free}) \to 1$$
as $n \to \infty$, and indeed there exist constant $B > 0, c< 1$ such that
$$\mathbb{P}(\langle \langle w_n^k \rangle \rangle \textup{ is free}) \geq 1 - B c^{n}$$
for any $n$.
\item Finally, if $N_n := \langle \langle w_n^k \rangle \rangle$, then for any $R > 0$ the injectivity radius of $N_n$ satisfies for any $n$
$$\mathbb{P}(\textup{inj}(N_n) \geq R) \ge 1 - B c^{n}. $$
\end{enumerate}
\end{theorem}

\begin{proof}
Let us choose $\alpha > 0$. Then by \cite[Proposition 6.23]{DGO} there exist constants $(A, \epsilon)$ such that 
if a family $\{N_\lambda\}_{\lambda \in \Lambda}$ of subgroups, closed under conjugation,
satisfies the small cancellation condition, then $\{N_\lambda\}$ is $\alpha$-rotating on a hyperbolic graph 
$X'$. Note that $X'$ is obtained from $X$ in the following way. First, one chooses a hyperbolic graph $X''$ 
which is equivariantly quasi-isometric to $X$. This is chosen once and for all; let $K$ be the Lipschitz constant 
of the map $X \to X''$. Now, the coned off space $X'$ is obtained by coning off certain quasiconvex subsets of a rescaled copy $\lambda X''$.
However, by looking at the proof one realizes that one can make sure that $\lambda \leq 1$ in all cases
(indeed, in the language of \cite[Proposition 6.23]{DGO}, the correct choice is $\lambda = \min \left( \frac{\delta_c}{\delta}, \frac{\Delta_c}{\Delta}, 1\right)$, with $A = \max \left( \frac{\textup{inj}_c(r_0)}{\delta_c}, \frac{\textup{inj}_c(r_0)}{\delta} \right)$ and $\epsilon = \frac{\Delta_c}{\textup{inj}_c(r_0)}$.)
Thus, the map $X \to X'$ is $K$-Lipschitz, where $K$ only depends on $X$ and not on the constant $\alpha$. 

Let us fix $\alpha \geq 200$, and let $(A, \epsilon)$ chosen as above. 
Let $\ell > 0$ be the drift of the random walk. Then by Theorem
\ref{T:MT} (3), there are constants $B_1$ and $c_1 < 1$ such that
$$\mathbb{P}\left(\tau(w_n) \geq \frac{\ell n}{2} \right) \ge 1 - B_1
c_1^n.$$
Moreover, by Proposition \ref{P:Delta-small}, there are constants
$B_2$ and $c_2 < 1$ such that
$$\mathbb{P}\left(\Delta(w_n) \leq \frac{\epsilon \ell n}{2}\right)
\ge 1 - B_2 c_2^{n}.$$ 
Now by Corollary \ref{C:Eplus}, there are constants $B_3$ and
$c_3 < 1$ such that 
\[ \P \left( E_G^+(w_n) = \langle w_n \rangle \ltimes
E_\mu \right) \ge 1 - B_3 c_3^{n}. \]
Thus, for suitable choices of $B_4$ and $c_4 < 1$,
\begin{equation}\label{eq:rot}
\mathbb{P}\left(\tau(w_n) \geq A \delta, \Delta(w_n)
\leq \epsilon \tau(w_n) \textup{ and } E_G^+(w_n) = \langle w_n \rangle \ltimes
E_\mu \right) \ge 1 - B_4 c_4^{n}.
\end{equation}
In particular, with probability which tends to $1$ we have 
$$E_G(w_n) = \langle w_n \rangle \ltimes_{\varphi_n} E_\mu$$ 
where $\varphi_n = \varphi(w_n)$ is the image of $w_n$ under the homomorphism
$$\varphi: \Gamma_\mu \to \textup{Aut }E_\mu.$$
Now, we have two cases. 
\begin{enumerate}
\item if $\varphi(w_n) = 1$, then all conjugates of $w_n$ in $G$ belong to different elementary subgroups. 

In fact, suppose that there exists $g \in G$ such that $g w_n g^{-1} \in E_G(g)$. Then, by Lemma \ref{L:conj-E} one has $g \in E_G(w_n)$, and by Lemma 
\ref{L:semidirect} one has $g w_n g^{-1} = w_n$. 

Now, consider the family of subgroups $\mathcal{R}_n := \{ g w_n g^{-1} \}_{g \in G}$. Finally, let
$N_n = \langle \langle H_n \rangle \rangle$ be the normal closure of
$H_n$. By equation \eqref{eq:rot} above, with probability at least
$1 - B_4 c_4^{n}$, the family $\mathcal{R}_n$ satisfies the
$(A, \epsilon)$-small cancellation condition, hence it is an
$\alpha$-rotating family.  Then by \cite[Corollary 5.4]{DGO}, the
normal closure of $w_n$ is the free product of conjugates of
$\langle w_n \rangle$, hence it is free.

\item
if $\varphi(w_n) \neq 1$, then there exists $g \in \Gamma_\mu$ such that $g w_n g^{-1} \neq w_n$. This implies that the intersection 
$$\langle \langle w_n \rangle \rangle \cap E_G(w_n)$$
is larger than $\langle w_n \rangle$, hence the normal closure $\langle \langle w_n \rangle \rangle$ cannot be a free group. 
\end{enumerate}

By the above discussion, the probability that the normal closure of $w_n$ in $G$ is free converges to the probability that $w_n$ maps to 
the identity in $E_\mu$. In order to compute such probability, note that under the map 
$$\varphi: \Gamma_\mu \to \textup{Aut }E_\mu$$
the random walk on $\Gamma_\mu$ pushes forward to a random walk on $\textup{Aut }E_\mu$, which is a finite group. 
Hence, the random walk equidistributes on the elements of the image of $\varphi$ into $\textup{Aut }E_\mu$, hence the probability that $\varphi(w_n) = 1$
converges to $\frac{1}{\#H_\mu}$, where $\# H_\mu$ is the cardinality of the image of $\varphi$. 
That is, the normal closure of $w_n$ is free if and only if the image $\varphi(w_n) = 1$, and the probability of this happening 
tends to $\frac{1}{\# H_\mu}$, so 
$$\mathbb{P}(\langle \langle w_n \rangle \rangle \textup{ is free}) \to \frac{1}{\#H_\mu}.$$

Hence, this probability tends to $1$ if and only if the image group $H_\mu = \varphi(\Gamma_\mu)$ is the trivial group, hence by Lemma \ref{L:EZ} if and only if $E_\mu = Z(\Gamma_\mu)$. 

\smallskip
To prove (ii), if $k = \#H_\mu$, then every element in the image of $\varphi$ has order which divides $k$, hence $\varphi(w_n^k) = \varphi(w_n)^k = 1$. 
Thus, as in the previous argument, if one defines $H_n := \langle w_n^k \rangle$, the probability that the family $\mathcal{R}_n := \{ g w_n^k g^{-1} \}_{g \in G}$ satisfies the small cancellation condition tends to $1$, 
hence the probability that the normal closure $N_n := \langle \langle w_n^k \rangle \rangle$ is free satisfies
$$\mathbb{P}(\langle \langle w_n^k \rangle \rangle \textup{ is free}) \geq 1 - B c^{n}$$ 
for suitable choices of $B > 0,  c < 1$. 

\smallskip

Now, to prove (iii), given $R > 0$ let $\alpha$ be such that
$\frac{\delta \alpha}{K} = R$. Then one can choose $(A, \epsilon)$ as
before for such $\alpha$.  Then with probability at least
$1 - B_4 c_4^{n}$, the family $\mathcal{R}_n$ is
$\alpha$-rotating. Hence, by \cite[Theorem 5.3]{DGO}, for each
$g \in N_n$, either $g$ belongs to some conjugate of $H_n$ or is
loxodromic on $X'$ with translation length at least $\alpha \delta$.
Then since the map $X \to X'$ is $K$-Lipschitz, such elements have
translation length on $X$ at least $\frac{\alpha \delta}{K}$.  On the
other hand, by Theorem \ref{T:MT} (3) we know that with probability at
least $1 - B_1 c_1^n$, the isometry $w_n^k$ is loxodromic on $X$ with
translation length $\geq R$. Therefore for suitable choices of $B_5$
and $c_5 < 1$, the probability that the injectivity radius of $N_n$ is
at least $R$ is at least $1 - B_5 c_5^{n}$.  The stated result
then follows for suitable choices of $B$ and $c < 1$.
\end{proof}

\begin{corollary}
Let $G$ be a group acting on a Gromov hyperbolic space, and let $\mu$
be a countable, non-elementary, reversible, bounded, WPD
probability measure on $G$.  Let $k = k(\mu)$ be the characteristic index of $\mu$, and let $N_n(\omega) := \langle \langle w_n^k \rangle \rangle$ be
the normal closure of $w_n^k$ in $G$. Then for almost every
sample path $\omega$, the sequence
$$(N_1(\omega), N_2(\omega), \dots, N_n(\omega), \dots)$$
contains infinitely many different normal subgroups of $G$. 
\end{corollary}

\begin{proof}
Fix $M > 0$, and consider the set
$$A_M := \{ \omega \ : \  \sup_n \textup{inj}(N_n(\omega)) \leq M \}.$$
We claim that $\mathbb{P}(A_M) = 0$. Indeed, suppose $\mathbb{P}(A_M) = \epsilon > 0$. 
Then by Theorem \ref{T:inj-later}, there exists $n_0$ such that for $n \geq n_0$ 
$$\mathbb{P}(\textup{inj}(N_n) \geq M + 1) > 1 - \epsilon$$
which is a contradiction because such a set must be disjoint from $A_M$. 
Then for almost every $\omega$ we have 
$$\limsup_{n \to \infty} \textup{inj}(N_n(\omega)) = + \infty,$$
which implies the claim.
\end{proof}
This completes the proof of Theorem \ref{T:inj-intro} in the Introduction. 

\subsection{Application to the mapping class group} \label{S:mcg}

In the case of the mapping class group, we may answer
\cite{margalit}*{Problem 10.11} and establish Theorem \ref{T:mcg}, as we now explain.

\begin{corollary}
Let $S$ be a surface of finite type whose mapping class group
$\Mod(S)$ is infinite.  Let $\mu$ be a probability distribution on
$\Mod(S)$ such that the support of $\mu$ has bounded image in the
curve complex under the orbit map, and for which
$\Gamma_\mu = \Mod(S)$.  Then there are constants $B > 0$ and $c < 1$
such that the probability that the normal closure
$\langle \langle w_n \rangle \rangle$ is a free subgroup of $\Mod(S)$
is at least $1 - Bc^{n}$.
\end{corollary}

This follows immediately from Theorem \ref{T:inj-later}, and the fact
that if $G = \Mod(S_{g, n})$ is the mapping class group of a surface of finite type, 
the group $E^+_G(G)$ is equal to the center of $G$, as we now explain.

We shall write $S_{g, n}$ for the surface of genus $g$ with $n$
punctures.  The mapping class groups $S_{0, n}$ with $n \le 3$ are
finite, so the results of this paper do not apply to them, and we
shall ignore them for the purposes of this section.

\begin{proposition}\label{P:mcg}
Let $S_{g, n}$ be a surface of genus $g$ with $n$ punctures, and
suppose that its mapping class group $G = Mod(S_{g,n})$ is
infinite. Then $E^+_G(G)$ is equal to the center of $G$.
\end{proposition}

\begin{proof}
If the mapping class group $G = \Mod(S_{g,n})$ is infinite, 
then its center is trivial, unless
$S_{g, n}$ is one of the following four surfaces:
$S_{1, 0}, S_{1, 1}, S_{1, 2}$ or $S_{2, 0}$, in which case the center
$Z(G)$ is isomorphic to $\Z / 2 \Z$, generated by the hyperelliptic
involution, see for example \cite{ivanov}*{Remark 8.15} or
\cite{farb-margalit}*{Section 3.4}.

Recall that $E_G^+(G)$ is a subgroup of $E_G(G)$, which is equal to
the maximal finite normal subgroup of $G$.  By Ivanov
\cite{ivanov}*{Section 11, Exercise 5.b} any finite normal subgroup of
the mapping class group is contained in the center $Z(G)$.  In the
cases in which the center is non-trivial, it is generated by the
hyperelliptic involution, which acts trivially on the boundary, and so
fixes pointwise the endpoints of all pseudo-Anosov elements.  In
particular, the groups $Z(G)$, $E_G(G)$ and $E_G^+(G)$ are all equal.
\end{proof}



\noindent Joseph Maher \\
CUNY College of Staten Island and CUNY Graduate Center \\
\url{joseph.maher@csi.cuny.edu} \\

\noindent Giulio Tiozzo \\
University of Toronto\\
\url{tiozzo@math.toronto.edu} \\

\end{document}